\documentclass[10pt]{amsart}
\usepackage{graphicx}
\usepackage{amssymb,amscd,url}
\usepackage[all]{xy} 
\usepackage{color}
\newtheorem{theorem}{Theorem}[section]
\newtheorem{proposition}[theorem]{Proposition} 
\newtheorem{corollary}[theorem]{Corollary}
\newtheorem{lemma}[theorem]{Lemma}
\newtheorem{remark}[theorem]{Remark}
\newtheorem{remarks}[theorem]{Remarks}

\newtheorem{definition}{Definition}[section]

\makeatletter
\@addtoreset{equation}{section}

\makeatother

\begin{document}

\title[Matrix liberation process, II]{Matrix liberation process \\ \Small{II: Relation to Orbital free entropy}}
\dedicatory{Dedicated to Professor Dan-Virgil Voiculescu on the occasion of his 70th birthday}
\author{Yoshimichi Ueda}
\address{
Graduate School of Mathematics, 
Nagoya University, 
Furocho, Chikusaku, Nagoya, 464-8602, Japan
}
\email{ueda@math.nagoya-u.ac.jp}
\thanks{Supported by Grant-in-Aid for Challenging Exploratory Research 16K13762 and Grant-in-Aid for Scientific Research (B) JP18H01122.}
\date{}
\maketitle
\begin{abstract}
We investigate the concept of orbital free entropy from the viewpoint of matrix liberation process. We will  show that many basic questions around the definition of orbital free entropy are reduced to the question of full large deviation principle for the matrix liberation process.  We will also obtain a large deviation upper bound for a certain family of random matrices that is an essential ingredient to define the orbital free entropy. The resulting rate function is made up into a new approach to free mutual information.
\end{abstract}

\allowdisplaybreaks{

\section{Introduction}

This paper is a sequel to our previous one \cite{Ueda:JOTP19} on the matrix liberation process, and devoted to explaining how the matrix liberation process is connected to the orbital free entropy $\chi_\mathrm{orb}$. Here, the negative of orbital free entropy may be regarded as a possible microstate approach to mutual information in free probability.

The key concept of free probability theory, initiated by Voiculescu in the early 80s, is the so-called free independence, which is a kind of statistical independence. Voiculescu then discovered around 1990 that  the large $N$ limit of independent (suitable) random matrices produces freely independent non-commutative random variables. In the 90s, in order to understand the notion of free independence deeply, Voiculescu introduced and studied several notions of free entropy (the microstate and the microstate-free ones), which are both analogs of Shannon's entropy and expected to agree. Then, these notions of free entropy were further studied by Biane, Guionnet, Shlyakhtenko and many others from several viewpoints, including large deviation theory and optimal transportation theory. (See \cite{Voiculescu:BLMS02} for early history on free entropy.) 

On the other hand, the information theory suggests us to introduce a free probability analog of mutual information that should characterize the freely independent situation as a unique minimizer. The main difficulty in such an attempt is the lack of free probability analog of relative entropy, and thus a completely new idea was (and probably still is) necessary. It was also Voiculescu \cite{Voiculescu:AdvMath99} who first attempted to develop the theory of mutual information in free probability. His approach is based upon the  liberation theory that he started to develop there with the microstate-free approach to free entropy. The most important concept in the liberation theory is the liberation process, a natural non-commutative probabilistic interpolation between given non-commutative random variables and their freely independent copies. Voiculescu's idea of liberation theory is completely non-commutative in nature, and has no origin in the classical probability theory. Hence the liberation theory is quite attractive from the view point of noncommutative analysis.  

Almost a decade later, we introduced, in a joint work \cite{HiaiMiyamotoUeda:IJM09} with Hiai and Miyamoto, the second candidate for mutual information in free probability, which we call the orbital free entropy, and its definition involves the adjoint actions of Haar-distributed unitary random matrices to the matrix space $M_N^{sa}$ of $N\times N$ self-adjoint matrices and follows the basic idea of microstate approach to free entropy. (Some considerations looking for better variants of orbital free entropy were made by Biane and Dabrowski \cite{BianeDabrowski:AdvMath13}, and a direct generalization dropping the hyperfiniteness for given random multi-variables was then given by us \cite{Ueda:IUMJ14}.) The liberation process is exactly the large $N$ limit of the matrix liberation process introduced in \cite{Ueda:JOTP19} and its `invariant measure' (or its limit distribution as time goes to $\infty$) exactly arises as the `distribution' of the adjoint actions of Haar-distributed unitary random matrices. Thus it is natural to consider the matrix liberation process for the conjectural unification between Voiculescu's and our approaches to mutual information in free probability. 

As a very first step, we proved in \cite{Ueda:JOTP19}, following the idea of \cite{BianeCapitaineGuionnet:InventMath03}, the large deviation upper bound with a good rate function that completely characterizes the corresponding liberation process as a unique minimizer. The next ideal steps on this line of research should be: (1) proving the large deviation lower bound with the same rate function, (2) applying the contraction principle to the resulting large deviation upper/lower bounds at time $T=\infty$, and (3) identifying the resulting rate function with Voiculescu's free mutual information. 

In this paper, we will mainly work on item (2). As a consequence, we will clarify how the matrix liberation process might resolve several technical drawbacks around the definition of orbital free entropy. As another consequence, we will get a large deviation upper bound result by applying the established contraction principle at $T = \infty$ to the one for the matrix liberation process in our previous paper \cite{Ueda:JOTP19}. We will then make the resulting rate function up into a new microstate-free candiadate for free mutual information. Items (1) and (3) are left as sequels to this paper. 

The precise contents of this paper are as follows. Sections 2 and 3 are preliminaries, and sections 4, 5 and 6 form the main body of this paper. The subsequent sections concern related materials.  

In section 2, we will give one of the key technical lemmas. It is about the long time behavior of the large $N$ limit of the logarithm of the heat kernel on $\mathrm{U}(N)$ divided by $N^2$. This seems to be of independent interest. Then we will give a slightly modified definition of orbital free entropy in section 3. 

In section 4, building on the previous work \cite{Ueda:JOTP19} we will prove that any large deviation upper or lower bound with speed $N^2$ for the matrix liberation process starting at given several deterministic matrices, say $\xi_{ij}(N)$, with limit joint distribution implies the corresponding one with the same speed for the corresponding random matrices $U_N^{(i)}\xi_{ij}(N)U_N^{(i)}{^*}$ with independent Haar-distributed unitary random matrices $U_N^{(i)}$. This explicitly relates the matrix liberation process with the orbital free entropy. Combining this with the main result of \cite{Ueda:JOTP19} we will obtain a large deviation upper bound for $U_N^{(i)}\xi_{ij}(N)U_N^{(i)}{^*}$. 

In section 5, we will investigate the resulting rate function for $U_N^{(i)}\xi_{ij}(N)U_N^{(i)}{^*}$ in some detail; we will prove that it admits a unique minimizer, which is precisely given by freely independent copies of the initially given non-commutative random multi-variables. This fact supports the validity of full large deviation principle with speed $N^2$ and the same rate function for $U_N^{(i)}\xi_{ij}(N)U_N^{(i)}{}^*$, because this unique minimizer property also follows from the conjectural full large deviation principle as well as the fact that the orbital free entropy completely characterizes the free independence (under the assumption of having matricial microstates). Moreover, this unique minimizer property suggests that the rate function can be regarded as a possible microstate-free candidate for free mutual information, and hence that the rate function ought have to have a coordinate-free fashion. 

In section 6, we will give such a coordinate-free formulation. The coordinate-free formulation will be shown to be a quantity for a given finite family of subalgebras in a tracial $W^*$-probability space, which satisfies a desired set of properties (see subsection 6.7) that any kind of free mutual information has to satisify and, of course, Voiculescu's one does. 

In section 7, we will explain how the proofs given in the previous paper \cite{Ueda:JOTP19} also work well for several independent unitary Brownian motions with deterministic matrices (which are assumed to have the large $N$ limit joint distribution), and compare its consequences with the corresponding results on the matrix liberation process. In section 8, we will give an explicit description in terms of free cumulants for the conditional expectation of the (time-dependent) liberation cyclic derivative $E_{\mathcal{N}(\tau)}(\pi_{\tilde{\tau}}(\Pi^s(\mathfrak{D}_s^{(k)} P)))$ (see section 4 for the notation), which is the most essential component of the rate function. The description is a complement to a rather ad-hoc computation made in section 5. Finally, in the appendix, we explain some basic facts on universal free products of unital $C^*$-algebras for the reader's convenience. 

\subsection*{Glossary} 
\begin{itemize} 
\item $\Vert\,-\,\Vert_\infty$ denotes the operator norm. 
\item $M_N \supset M_N^{sa}$ denote the $N\times N$ complex matrices and the $N\times N$ self-adjoint matrices. For each $R > 0$, $(M_N^{sa})_R$ denotes the subset of $A \in M_N^{sa}$ with $\Vert A \Vert_\infty \leq R$. 
\item $\mathrm{Tr}_N$ denotes the usual (i.e., non-normalized) trace on $M_N$, and $\mathrm{tr}_N$ does its normalized one. We consider the Hilbert-Schmidt norm $\Vert A\Vert_{HS} := \sqrt{\mathrm{Tr}_N(A^* A)}$ on $M_N$. It is known that $M_N^{sa}$ equipped with $\Vert\,\cdot\,\Vert_{HS}$ is naturally identified with the $N^2$-dimensional Euclidean space $\mathbb{R}^{N^2}$. Thus $M_N = M_N^{sa} + \sqrt{-1}M_N^{sa}$ equipped with $\Vert\,\cdot\,\Vert_{HS}$ is also naturally identified with the $2N^2$-dimensional Euclidean space $\mathbb{R}^{2N^2} = \mathbb{R}^{N^2}\oplus\mathbb{R}^{N^2}$. 
\item $\mathrm{U}(N)$ denotes the $N\times N$ unitary matrices equipped with the Haar probability measure $\nu_N$; {\it n.b.}, the symbol $\nu_N$ differs from the one $\gamma_{\mathrm{U}(N)}$ in \cite{HiaiMiyamotoUeda:IJM09}, \cite{Ueda:IUMJ14}. A Haar-distributed $N\times N$ random unitary matrix means a random variable with values in $\mathrm{U}(N)$, whose probability distribution measure is exactly $\nu_N$. 
\item $TS(\mathcal{A})$ denotes the tracial states on a unital $C^*$-algebra $\mathcal{A}$. For a given subset $\mathcal{X}$ of a $W^*$-algebra, we denote by $\overline{\mathcal{X}}^w$ its closure in the $\sigma$-weak topology (i.e., the weak$^*$ topology induced from the predual). For a unital $*$-homomorphism $\pi : \mathcal{A} \to \mathcal{B}$ between unital $C^*$-algebras, $\pi^* : TS(\mathcal{B}) \to TS(\mathcal{A})$ denotes the dual map $\varphi \in TS(\mathcal{B}) \mapsto \varphi\circ\pi \in TS(\mathcal{A})$. 
\item For a random variable $X$ in the usual sense, $\mathbb{E}[X]$ denotes the expectation of $X$. Moreover, for a random variable $Y$ with values in a topological space, we write $\mathbb{P}(Y \in A) := \mathbb{E}[\mathbf{1}_A(Y)]$ for any Borel subset $A$; this is the distribution measure of $Y$. Here $\mathbf{1}_A$ denotes the indicator function of $A$.   
\end{itemize}  

\subsection*{Remark on Part I} We have investigated the matrix liberation process $\Xi^\mathrm{lib}(N)$ starting at (deterministic) $\Xi(N) = (\Xi_i(N))_{i=1}^{n+1}$ with $\Xi_i(N) = (\xi_{ij}(N))_{j=1}^{r(i)} \in (M_N^{sa})^{r(i)}$. Here, we remark that $r(i) = \infty$ is allowable; namely, each $\Xi_i(N)$ may be a countably infinite family of $N\times N$ self-adjoint matrices, and all the results given in part I still hold true in this more general situation without essential changes. In fact, we only need to change the metric $d$ on the continuous tracial states $TS^c\big(C^*_R\langle x_{\bullet\diamond}(\,\cdot\,)\rangle\big)$ (see subsection 4.3 below) as follows. Let $\mathcal{W}_{\leq\ell} $ be all the words of length \emph{not greater than $\ell$} in indeterminates $x_{ij}=x_{ij}^*$ with $1 \leq i \leq n+1$, $1 \leq j \leq \ell$ (remark this restriction on $j$, which guarantees that $\mathcal{W}_{\leq\ell}$ is a finite set), and we define
\begin{equation}\label{Eq1.1} 
d(\tau_1,\tau_2) = 
\sum_{m=1}^\infty\sum_{\ell=1}^\infty \frac{1}{2^{m+\ell}} \max_{w \in \mathcal{W}_{\leq\ell}} \sup_{(t_1,\dots,t_\ell) \in [0,m]^\ell} \big(|\tau_1(w(t_1,\dots,t_\ell))-\tau_2(w(t_1,\dots,t_\ell))|\wedge1\big)
\end{equation}
for $\tau_1,\tau_2 \in TS^c\big(C^*_R\langle x_{\bullet\diamond}(\,\cdots\,)\rangle\big)$. Here, $w(t_1,\dots,t_\ell)$ is constructed by substituting $x_{i_k j_k}(t_k)$ for $x_{i_k j_k}$ in a given word $w = x_{i_1 j_1}\cdots x_{i_{\ell'} j_{\ell'}}$ with $\ell' \leq \ell$. 

\subsection*{Acknowledgements} We thank the CRM, Montr\'{e}al and the organizers of the thematic one-month program `New Developments in Free Probability and Applications' held there in March 2019, where we wrote a part of this paper. We also thank Thierry L\'{e}vy for his inspiring lectures and discussions at Kyoto and Nagoya in Oct.\ 2018 and David Jekel for his detailed feedback to the first version of this paper, which enabled us to improve the presentation of this paper. Finally, we thank the referee for his/her careful reading and pointing out some typos. 

\subsection*{Added in proof} We have further investigated the rate functions in this paper after the submission. As one of its simple consequences, we confirmed that $I^{\mathrm{lib}}_{\sigma_0,\infty}(\tau)=I_{\sigma_0}^{\mathrm{lib}}(\tau)$ certainly holds if $I_{\sigma_0}^{\mathrm{lib}}(\tau) < +\infty$ (see subsection 4.6 for the notation). We will give those details elsewhere. 

\section{The long time behavior of the large $N$ limit of the Heat kernel on $\mathrm{U}(N)$}  

In this section, we will investigate the long time behavior of the large $N$ limit of the logarithm of the heat kernel on $\mathrm{U}(N)$ by utilizing a recent work on the Douglas and Kazakov transition due to Thierry L\'{e}vy and Ma\"{i}da \cite{LevyMaida:ESAIM:Proc15} (based on Guionnet and Ma\"{i}da's work \cite{GuionnetMaida:PTRF05}) as well as Li and Yau's classical work on parabolic kernels \cite{LiYau:ActaMath86}. The consequence (Lemma \ref{L2.1}) will play a key role in section 4 to establish the contraction principle \emph{at time $T = \infty$} for large deviation upper/lower bounds with speed $N^2$ for the matrix liberation process $\Xi^\mathrm{lib}(N)$.

\medskip
Consider $\mathrm{U}(N)$ as a Riemannian manifold of dimension $N^2$ by the inner product on the corresponding Lie algebra $\mathfrak{u}(N) = \sqrt{-1}M_N^{sa}$:  
\[
\langle X\,|\,Y\rangle := -N\mathrm{Tr}_N(XY), \quad X,Y \in \mathfrak{u}(N). 
\]
Let $\mathrm{Ric}$ be the Ricci curvature associated with this Riemannian structure. It is known, by e.g., \cite[Lemma F.27]{AndersonGuionnetZeitouni-Book}, that  
\[
\mathrm{Ric}(X,X) = \frac{N}{2}(\langle X\,|\,X\rangle - \langle X\,|\,(1/N)\sqrt{-1}I_N \rangle^2) \geq 0
\]
for every $X \in \mathfrak{u}(N)$. 

\medskip
Let $p_{N,t}(U)$ be the heat kernel on $\mathrm{U}(N)$ with respect to this Riemannian structure as in \cite[section 3.1]{LevyMaida:ESAIM:Proc15}. Looking at the Fourier expansion of $p_{N,t}$ (see e.g., \cite[Eq.(21)]{LevyMaida:ESAIM:Proc15}) we observe that 
\[
\max_{U \in \mathrm{U}(N)} p_{N,t}(U) = p_{N,t}(I_N) 
\]
holds for every $t > 0$. Recall that $p_{N,t}(U) = p_N(U,I_N,t/2)$, where $p_N(U,V,t)$, $U,V \in \mathrm{U}(N)$, $t > 0$, is a unique fundamental solution of the heat equation $\partial_t u = \Delta u$ with the Laplacian $\Delta$ on $\mathrm{U}(N)$ equipped with the above Riemannian structure. See e.g., \cite[p.135]{Chavel-Book} for the notion of fundamental solutions of heat equations. It is well known, see e.g.\ \cite[Theorem 1 in V.III.1]{Chavel-Book}, that $p_N$ is strictly positive. Since the Ricci curvature is non-negative as we saw before, we can apply Li--Yau's theorem \cite[Theorem 2.3]{LiYau:ActaMath86} to $u(U,t) := p_N(U,I_N,t)$ and obtain that  
\[
p_N(I_N,I_N,\varepsilon t) \leq p_N(U,I_N,t) \varepsilon^{-N^2/2} \exp\Big(\frac{d_N(I_N,U)^2}{4(1-\varepsilon)t}\Big)
\]
for every $t >0$, $0 < \varepsilon < 1$ and $U \in \mathrm{U}(N)$, where $d_N(I_N,U)$ denotes the Riemannian distance between $I_N$ and $U$. Since $\max_{U \in \mathrm{U}(N)} d_N(I_N,U) = N\pi$ (see e.g.\ the proof of \cite[Proposition 4.1]{LevyMaida:JFA10}), the above inequality with $t=T/2$ implies that 
\begin{equation*}
p_{N,\varepsilon T}(I_N) \,\varepsilon^{N^2/2}\exp\Big(-\frac{(N\pi)^2}{2(1-\varepsilon)T}\Big) \leq p_{N,T}(I_N)\,\varepsilon^{N^2/2}\exp\Big(-\frac{d_N(I_N,U)^2}{2(1-\varepsilon)T}\Big) \leq p_{N,T}(U)
\end{equation*} 
for every $T >0$, $0 < \varepsilon < 1$ and $U \in \mathrm{U}(N)$. Consequently, we have obtained that 
\begin{equation*}
\frac{1}{N^2} \log p_{N,\varepsilon T}(I_N) + \frac{1}{2}\log\varepsilon -\frac{\pi^2}{2(1-\varepsilon)T}
\leq \frac{1}{N^2}\log p_{N,T}(U) \leq \frac{1}{N^2}\log p_{N,T}(I_N). 
\end{equation*} 
for every $t>0$, $0 < \varepsilon < 1$ and $U \in \mathrm{U}(N)$. By \cite[Theorem 1.1]{LevyMaida:ESAIM:Proc15}, it is known that 
\[
F(T) := \lim_{N\to\infty}\frac{1}{N^2}\log p_{N,T}(I_N) = \lim_{N\to\infty} \frac{1}{N^2}\log\Big(\max_{U\in\mathrm{U}(N)}p_{N,T}(U)\Big)
\]
exists and defines a continuous function on $(0,+\infty)$. Thus, we have 
\begin{equation*}
F(\varepsilon T) + \frac{1}{2}\log\varepsilon - \frac{\pi^2}{2(1-\varepsilon)T} \leq \varliminf_{N\to\infty}\frac{1}{N^2}\log p_{N,T}(U) \leq \varlimsup_{N\to\infty}\frac{1}{N^2}\log p_{N,T}(U) \leq F(T)
\end{equation*}
for every $T>0$, $0 < \varepsilon < 1$ and $U \in \mathrm{U}(N)$. In particular, we obtain that
\begin{equation}\label{Eq2.1} 
F(\varepsilon T) + \frac{1}{2}\log\varepsilon - \frac{\pi^2}{2(1-\varepsilon)T} 
\leq \varliminf_{N\to\infty} \frac{1}{N^2}\log\Big(\min_{U\in\mathrm{U}(N)}p_{N,T}(U)\Big) \leq  F(T)
\end{equation}
for every $T > 0$ and $0 < \varepsilon < 1$. 

\medskip
Assume that $T > \pi^2$ in what follows. We need 
the complete elliptic functions of the first kind and the second kind: 
\[
K = K(k) := \int_0^1 \frac{ds}{\sqrt{(1-s^2)(1-k^2 s^2)}}, \quad 
E = E(k) := \int_0^1 \sqrt{\frac{1-k^2 s}{1-s^2}}\,ds.  
\]
With $T = 4K(2E-(1-k^2)K)$, \cite[Propositions 4.2, 5.2]{LevyMaida:ESAIM:Proc15} show that
\begin{align*}
F(T) 
&= 
\frac{K(2E-(1-k^2)K)}{6} + \frac{1}{2}\log\Big(\frac{1}{4}\frac{1}{(2E-(1-k^2)K)^2}(1-k^2)\Big) \\
&\qquad\qquad+ 
\frac{2(1+k^2)K}{3(2E-(1-k^2)K)} + \frac{((1-k^2)K)^2}{12(2E-(1-k^2)K)^2}.  
\end{align*}
Recall that 
\[
K = \log\frac{4}{\sqrt{1-k^2}} + o(1) = \frac{3}{2}\log2 - \frac{1}{2}\log(1-k) + o(1) \quad \text{as $k\to1-0$}
\]
(see e.g.\ \cite[p.11]{ByrdFriedman-Handbook}). This immediately implies that $\lim_{k\to1-0} (1-k)^\alpha K = 0$ for any $\alpha > 0$. We also have $E = 1$ at $k=1$. By the well-known formulas $dK/dk = (E-(1-k^2)K)/(k(1-k^2))$ and $dE/dk = (E-K)/k$,  $0 < k < 1$ (see \cite[p.282]{ByrdFriedman-Handbook}), we have $d(2E-(1-k^2)K)/dk = (1-k^2)dK/dk$. It is clear that $K$ is increasing in $k$. Hence $T$ is an increasing function in $k$. Then, we observe that $T \to +\infty$ if and only if $k \to 1-0$. Moreover, we have
\begin{align*}
F(T) 
&= 
\Big(\frac{E}{3}+\frac{2(1+k^2)}{3(2E-(1-k^2)K)}\Big)K - \frac{3}{2}\log 2 + \frac{1}{2}\log(1-k) + o(1) \\
&= 
\frac{(E-1)K}{3} + \frac{2((1-k^2)K^2 - (1-k^2)K - 2(E-1)K)}{3(2E - (1-k^2)K)} \\
&\qquad\qquad\qquad\qquad\qquad\quad+ \Big(K - \frac{3}{2}\log2 + \frac{1}{2}\log(1-k)\Big) + o(1) \\
&= 
\frac{(E-1)K}{3} + \frac{2((1-k^2)K^2 - (1-k^2)K - 2(E-1)K)}{3(2E - (1-k^2)K)} + o(1)
\end{align*}  
as $k\to1-0$. Since $dE/dk = (E-K)/k$, $0 < k < 1$ again, L'Hospital's rule (see e.g.\  \cite[Theorem 5.13]{Rudin:CalculusBook}) enables us to confirm that $\lim_{k\to1-0} (E-1)/(1-k)^{1/2} = 0$ and hence 
\[
\lim_{k\to1-0} (E-1)K = \lim_{k\to1-0} \Big(\frac{E-1}{(1-k)^{1/2}}\cdot(1-k)^{1/2}K\Big) = 0. 
\]
Consequently, we get $\lim_{T\to+\infty}F(T) = 0$. 

\medskip
Taking the limit of \eqref{Eq2.1} as $T\to+\infty$ we have 
\begin{align*}
\frac{1}{2}\log\varepsilon 
&\leq \varliminf_{T\to+\infty} \varliminf_{N\to\infty} \frac{1}{N^2}\log\Big(\min_{U\in\mathrm{U}(N)}p_{N,T}(U)\Big) \\
&\leq \lim_{T\to+\infty}\lim_{N\to\infty} \frac{1}{N^2}\log\Big(\max_{U\in\mathrm{U}(N)}p_{N,T}(U)\Big) = 0
\end{align*}
for all $0 < \varepsilon < 1$. Since $\varepsilon$ can arbitrarily be close to $1$, we finally obtain the next lemma, which will play a key role in \S4. 

\begin{lemma}\label{L2.1} With 
\[
L(T) :=  \varliminf_{N\to\infty} \frac{1}{N^2}\log\Big(\min_{U\in\mathrm{U}(N)}p_{N,T}(U)\Big), \quad 
U(T) := \varlimsup_{N\to\infty} \frac{1}{N^2}\log\Big(\max_{U\in\mathrm{U}(N)}p_{N,T}(U)\Big) = F(T)
\]
we have 
\[
\lim_{T\to+\infty} L(T)  = \lim_{T\to+\infty} U(T) = 0. 
\]
\end{lemma}

\section{Orbital free entropy revisited}

Let $\Xi = (\Xi_i)_{i=1}^{n+1}$ with $\Xi_i = (\Xi_i(N))_{N\in\mathbb{N}}$ be a finite family of sequences of (deterministic) multi-matrices such that each $\Xi_i(N) = (\xi_{ij}(N))_{j=1}^{r(i)}$, $1 \leq i \leq n+1$, is chosen from $((M_N^{sa})_R)^{r(i)}$ with $r(i) \in \mathbb{N}\cup\{\infty\}$ \emph{for some $R>0$.} We sometimes write $\Xi = (\Xi(N))_{N\in\mathbb{N}}$ with $\Xi(N) = ((\xi_{ij}(N))_{j=1}^{r(i)})_{i=1}^{n+1}$. As in \cite{Ueda:JOTP19} we consider the universal $C^*$-algebra $C^*_R\langle x_{\bullet \diamond}\rangle$ generated by $x_{ij} = x_{ij}^*$, $1 \leq i \leq n+1, j \geq 1$, such that $\Vert x_{ij}\Vert_\infty \leq R$ for all $i,j$, into which the universal unital $*$-algebra $\mathbb{C}\langle x_{\bullet\diamond}\rangle$ generated by the $x_{ij} = x_{ij}^*$ is faithfully and norm-densely embedded. Similarly, we define $\mathbb{C}\langle x_{i\diamond} \rangle \hookrightarrow C^*_R\langle x_{i\diamond}\rangle$ by fixing the first suffix $i$ of generators. These universal $C^*$-algebras are constructed as universal free products of copies of $C[-R,R]$, and each generator $x_{ij}$ is given by the coordinate function $f(t)=t$ in the $(i,j)$th copy of $C[-R,R]$. The above embedding properties are guaranteed by Proposition \ref{PA4}. The $*$-homomorphism given by $x_{ij} \mapsto \xi_{ij}(N)$ enables us to define tracial states $\mathrm{tr}^{\Xi(N)} \in TS(C^*_R\langle x_{\bullet \diamond}\rangle)$ as well as $\mathrm{tr}^{\Xi_i(N)} \in TS(C^*_R\langle x_{i\diamond}\rangle)$, $1 \leq i \leq n+1$, by $P = P(x_{\bullet\diamond}) \mapsto \mathrm{tr}_N(P(\xi_{\bullet\diamond}(N)))$ ({\it n.b.}, these notations differ a little bit from those in \cite{Ueda:JOTP19}). Remark that we can alternatively define $\mathrm{tr}^{\Xi_i(N)}$ to be the restriction of $\mathrm{tr}^{\Xi(N)}$ to $C^*_R\langle x_{i\diamond}\rangle$ ($\hookrightarrow C^*_R\langle x_{\bullet\diamond}\rangle$ faithfully by \cite[Theorem 3.1]{Blackadar:IUMJ78} with Lemma \ref{LA1}). \emph{We also assume that each $\Xi_i$, $1 \leq i \leq n+1$, has a limit distribution as $N \to \infty$; namely, there exists a $\sigma_{0,i} \in TS(C^*_R\langle x_{i\diamond}\rangle)$ such that $\lim_{N\to\infty} \mathrm{tr}^{\Xi_i(N)} = \sigma_{0,i}$ in the weak$^*$ topology.} (This is the minimum requirement for $\Xi$ to define $\chi_\mathrm{orb}(\sigma\mid\Xi)$ below.) In what follows, we denote by $TS_\mathrm{fda}(C^*_R\langle x_{i\diamond}\rangle)$ all the tracial states that arise in this way for a fixed $1\leq i \leq n+1$. We also define $TS_\mathrm{fda}(C^*_R\langle x_{\bullet\diamond}\rangle)$ similarly.

\medskip
Let us introduce a variant of orbital free entropy, say $\chi_\mathrm{orb}(\sigma\,|\,\Xi)$ for $\sigma \in TS(C^*_R\langle x_{\bullet \diamond}\rangle)$, which is essentially the same as the old one in \cite[section 4]{HiaiMiyamotoUeda:IJM09} for hyperfinite non-commutative random multi-variables. 

Define $\mathbf{U} = (U_i)_{i=1}^n \in \mathrm{U}(N)^n \mapsto \mathrm{tr}^{\Xi(N)}_\mathbf{U} \in TS\big(C^*_R\big\langle x_{\bullet\diamond}\big\rangle\big)$ by $\mathrm{tr}^{\Xi(N)}_\mathbf{U} := \mathrm{tr}_N\circ\Phi_\mathbf{U}^{\Xi(N)}$, where $\Phi_\mathbf{U}^{\Xi(N)} : C^*_R\big\langle x_{\bullet\diamond} \big\rangle \to M_N(\mathbb{C})$ is a unique $*$-homomorphism sending $x_{ij}$ $(1 \leq i \leq n+1)$ to $U_i\xi_{ij}(N)U_i^*$ with $\mathbf{U} = (U_i)_{i=1}^n$ and $x_{n+1\,j}$ to $\xi_{n+1\,j}(N)$, respectively. Consider an open neighborhood $O_{m,\delta}(\sigma)$, $m \in \mathbb{N}$, $\delta>0$, at $\sigma$ in the weak$^*$ topology on $TS(C^*_R\langle x_{\bullet \diamond}\rangle)$ defined to be all the $\sigma' \in TS(C^*_R\langle x_{\bullet \diamond}\rangle)$ such that 
\[
|\sigma'(x_{i_1 j_1}\cdots x_{i_p j_p}) - \sigma(x_{i_1 j_1}\cdots x_{i_p j_p})| < \delta
\]
whenever $1 \leq i_k \leq n+1$, $1 \leq j_k \leq m$, $1 \leq k \leq p$ and $1 \leq p \leq m$. Then we define 
\begin{equation}\label{Eq7}
\begin{aligned}
\chi_\mathrm{orb}(\sigma\,|\,\Xi(N)\,;N,m,\delta) &:=\log\nu_N^{\otimes n}\Big(
\big\{\mathbf{U} \in \mathrm{U}(N)^n\,\big|\,\mathrm{tr}_{\mathbf{U}}^{\Xi(N)} \in O_{m,\delta}(\sigma)\big\} 
\Big), \\
\chi_\mathrm{orb}(\sigma\,|\,\Xi\,;m,\delta) &:= \varlimsup_{N\to\infty} \frac{1}{N^2}\chi_\mathrm{orb}(\sigma\,|\,\Xi\,;N,m,\delta), \\
\chi_\mathrm{orb}(\sigma\,|\,\Xi) &:= \lim_{\substack{m\to\infty \\ \delta\searrow0}}\chi_\mathrm{orb}(\sigma\,|\,\Xi;m,\delta) 
\end{aligned}
\end{equation}
with $\log0 := -\infty$. Remark that $\chi_\mathrm{orb}(\sigma\,|\,\Xi) = -\infty$, if $\sigma$ does not agree with $\sigma_{0,i}$ on $C^*_R\langle x_{i\diamond}\rangle$ for some $1\leq i \leq n+1$. This is a natural property; see \cite[Proposition 3.1]{HiaiUeda:CMP15} as well as Remark \ref{R6.2}.

\medskip 
We could prove in \cite[Lemma 4.2]{HiaiMiyamotoUeda:IJM09} that $\chi_\mathrm{orb}(\sigma\,|\,\Xi)$ depends only on the given $\sigma_{0,i}$, $1 \leq i \leq n+1$, that is, it is independent of the choice of $\Xi$, when each tuple $(x_{ij})_{j=1}^{r(i)}$ produces a hyperfinite von Neumann algebra via the GNS construction associated with $\sigma_{0,i}$. However, we suspected that this is not always the case. Hence, in \cite{Ueda:IUMJ14}, in order to remove the dependency of $\Xi$ we took the supremum of $\chi_\mathrm{orb}(\sigma\,|\,\mathbf{A}\,;N,m,\delta)$ all over the tuples $\mathbf{A}$ of multi-matrices in place of $\Xi(N)$ to define $\chi_\mathrm{orb}(\mathbf{X}_1,\dots,\mathbf{X}_{n+1})$ (see the  review below). Here, we will examine another simpler way of removing the dependency. So far, we have only assumed that each $\Xi_i$ has a limit distribution as $N\to\infty$, that is, $\lim_{N\to\infty} \mathrm{tr}^{\Xi_i(N)} = \sigma_{0,i}$. In what follows, we need the stronger assumption that  the whole $\Xi$ has a limit distribution as $N\to\infty$, that is, $\lim_{N\to\infty} \mathrm{tr}^{\Xi(N)} = \sigma_0$.

Let another $\sigma_0 \in TS(C^*_R\langle x_{\bullet\diamond}\rangle)$ be given in such a way that its restriction to $C^*_R\langle x_{i\diamond}\rangle$ is $\sigma_{0,i}$ for every $1 \leq i \leq n+1$. Then we define 
\begin{equation}\label{Eq8}
\chi_\mathrm{orb}(\sigma\,|\,\sigma_0) := \sup\Big\{ \chi_\mathrm{orb}(\sigma\,|\,\Xi)\,\Big|\, \Xi = (\Xi(N))_{N\in\mathbb{N}}\,; \lim_{N\to\infty} \mathrm{tr}^{\Xi(N)} = \sigma_0 \Big\}.
\end{equation}
We define it to be $-\infty$ if $\sigma_0$ does not fall into $TS_\mathrm{fda}(C_R^*\langle x_{\bullet\diamond}\rangle)$. Remark that $\chi_\mathrm{orb}(\sigma\,|\,\Xi)$ is well defined in the above definition, since $\lim_{N\to\infty} \mathrm{tr}^{\Xi(N)} = \sigma_0$ implies that $\lim_{N\to\infty} \mathrm{tr}^{\Xi_i(N)} = \sigma_{0,i}$ for every $1 \leq i \leq n+1$. Moreover, taking the supremum all over the possible approximations $\Xi$ to $\sigma_0$ is motivated from the large deviation upper bound for the matrix liberation process starting at $\Xi(N)$ \cite{Ueda:JOTP19} (see the next section), because the rate function that we found there is independent of the choice of approximations $\Xi$. We will prove two propositions, which suggest that $\chi_\mathrm{orb}(\sigma\,|\,\sigma_0)$ should be the same for a large class of $\sigma_0$.

\medskip
We next recall the original orbital free entropy introduced in \cite{Ueda:IUMJ14} (with a non-essential modification \cite[Remark 3.3]{Ueda:ArchMath17}) in the current setting. Let $\pi_\sigma : C_R^*\langle x_{\bullet\diamond}\rangle \curvearrowright \mathcal{H}_\sigma$ be the GNS representation associated with $\sigma$. Set $X^\sigma_{ij} := \pi_\sigma(x_{ij})$, $1 \leq i \leq n+1, j \geq 1$, and then write $\mathbf{X}^\sigma_i = (X^\sigma_{ij})_{j=1}^{r(i)}$, $1 \leq i \leq n+1$. Remark that the joint distribution of those $\mathbf{X}^\sigma_1,\dots,\mathbf{X}^\sigma_{n+1}$ with respect to the tracial state on $\pi_\sigma(C_R^*\langle x_{\bullet\diamond}\rangle)''$ induced from $\sigma$ is exactly $\sigma$. On the other hand, if we have uniformly norm-bounded non-commutative self-adjoint random multi-variables $\mathbf{X}_1 = (X_{1j})_{j=1}^{r(1)},\dots,\mathbf{X}_{n+1}=(X_{n+1\,j})_{j=1}^{r(n+1)}$ in a $W^*$-probability space $(\mathcal{M},\tau)$, i.e., $X_{ij}^* = X_{ij}$ and $R := \sup_{i,j} \Vert X_{ij}\Vert_\infty < +\infty$, then we have a unique tracial state $\sigma_{(\mathbf{X}_i)} \in TS(C^*_R\langle x_{\bullet \diamond}\rangle)$ naturally, that is, $\sigma _{(\mathbf{X}_i)}(x_{i_1 j_1}\cdots x_{i_m j_m}) := \tau(X_{i_1 j_1}\cdots X_{i_m j_m})$ for example. For any $\mathbf{A} = (\mathbf{A}_i)_{i=1}^{n+1}$ with $\mathbf{A}_i = (A_{ij})_{j=1}^{r(i)} \in ((M_N^{sa})_R)^{r(i)}$, $1 \leq i \leq n+1$, we define 
\begin{align*}
\chi_\mathrm{orb}(\mathbf{X}_1,\dots,\mathbf{X}_{n+1}; \mathbf{A}, N,m,\delta) 
&:= 
\log\nu_N^{\otimes n}\Big(
\big\{ \mathbf{U} \in \mathrm{U}(N)^n\,\big|\, \mathrm{tr}_\mathbf{U}^\mathbf{A} \in O_{m,\delta}(\sigma_{(\mathbf{X}_i)}) \big\}\Big), \\
\bar{\chi}_\mathrm{orb}(\mathbf{X}_1,\dots,\mathbf{X}_{n+1}; N,m,\delta) 
&:= 
\sup_\mathbf{A} \chi_\mathrm{orb}(\mathbf{X}_1,\dots,\mathbf{X}_{n+1}; \mathbf{A}, N,m,\delta), \\
\bar{\chi}_\mathrm{orb}(\mathbf{X}_1,\dots,\mathbf{X}_{n+1};m,\delta) 
&:=
\varlimsup_{N\to\infty} \frac{1}{N^2}\bar{\chi}_\mathrm{orb}(\mathbf{X}_1,\dots,\mathbf{X}_{n+1}; N,m,\delta), \\
\chi_\mathrm{orb}(\mathbf{X}_1,\dots,\mathbf{X}_{n+1}) 
&:= 
\lim_{\substack{m\to\infty \\ \delta \searrow 0}} \chi_\mathrm{orb}(\mathbf{X}_1,\dots,\mathbf{X}_{n+1};m,\delta),  
\end{align*}
where $\mathrm{tr}_\mathbf{U}^\mathbf{A}$ is defined in the same manner as the $\mathrm{tr}_\mathbf{U}^{\Xi(N)}$ above. Note that the above definition clearly works even when $r(i)=\infty$ for every $1 \leq i \leq n+1$. 

\medskip
The next proposition suggests which approximating sequences $\Xi$ are suitable to define the orbital free entropy. 

\begin{proposition}\label{P3.2} We have 
\[
\chi_\mathrm{orb}(\sigma\,|\,\sigma_0) \leq \chi_\mathrm{orb}(\mathbf{X}^\sigma_1,\dots,\mathbf{X}^\sigma_{n+1}),
\]
and equality holds when $\sigma = \sigma_0$. 
\end{proposition}
\begin{proof}
Let $\Xi = (\Xi(N))_{N\in\mathbb{N}}$ with $\Xi_i(N) = (\xi_{ij}(N))_{j=1}^{r(i)}$, $1 \leq i \leq n+1$, be as in definition \eqref{Eq8}. Clearly, 
\[
\chi_\mathrm{orb}(\sigma\,|\,\Xi\,;N,m,\delta) 
=
\chi_\mathrm{orb}(\mathbf{X}^\sigma_1,\dots,\mathbf{X}^\sigma_{n+1};\Xi(N),N,m,\delta)
\leq 
\bar{\chi}_\mathrm{orb}(\mathbf{X}^\sigma_1,\dots,\mathbf{X}^\sigma_{n+1};N,m,\delta) 
\]
holds for every $N$, $m$ and $\delta$. This immediately implies $\chi_\mathrm{orb}(\sigma\,|\,\Xi) \leq \chi_\mathrm{orb}(\mathbf{X}^\sigma_1,\dots,\mathbf{X}^\sigma_{n+1})$. Since $\Xi$ has arbitrarily been chosen, we obtain $\chi_\mathrm{orb}(\sigma\,|\,\sigma_0) \leq \chi_\mathrm{orb}(\mathbf{X}^\sigma_1,\dots,\mathbf{X}^\sigma_{n+1})$. 

\medskip
We next prove the latter assertion. We may and do assume that $\chi_\mathrm{orb}(\mathbf{X}^{\sigma}_1,\dots,\mathbf{X}^{\sigma}_{n+1}) > -\infty$; otherwise the desired equality trivially holds as $-\infty = -\infty$ by the first part. We can inductively choose an increasing sequence $N_k$ in such a way that 
\begin{align*}
\bar{\chi}_\mathrm{orb}(\mathbf{X}^{\sigma}_1,\dots,\mathbf{X}^{\sigma}_{n+1};k,1/k) - \frac{1}{k} 
&< 
\frac{1}{N_k^2}\bar{\chi}_\mathrm{orb}(\mathbf{X}^{\sigma}_1,\dots,\mathbf{X}^{\sigma}_{n+1};N_k,k,1/k) \\
&< 
\bar{\chi}_\mathrm{orb}(\mathbf{X}^{\sigma}_1,\dots,\mathbf{X}^{\sigma}_{n+1};k,1/k) + \frac{1}{k}
\end{align*}
holds for every $k$; hence  
\[
\chi_\mathrm{orb}(\mathbf{X}^{\sigma}_1,\dots,\mathbf{X}^{\sigma}_{n+1}) = 
\lim_{k\to\infty}\frac{1}{N_k^2}\bar{\chi}_\mathrm{orb}(\mathbf{X}^{\sigma}_1,\dots,\mathbf{X}^{\sigma}_{n+1}; N_k,k,1/k). 
\]
For each $k$ one can choose $\mathbf{A}(N_k) = (\mathbf{A}_i(N_k))_{i=1}^{n+1}$ with $\mathbf{A}_i(N_k) = (A_{ij}(N_k))_{j=1}^{r(i)} \in ((M_{N_k}^{sa})_R)^{r(i)}$, $1 \leq i \leq n+1$, in such a way that 
\[
-\infty < \bar{\chi}_\mathrm{orb}(\mathbf{X}^{\sigma}_1,\dots,\mathbf{X}^{\sigma}_{n+1};N_k,k,1/k) -1 < \chi_\mathrm{orb}(\mathbf{X}^{\sigma}_1,\dots,\mathbf{X}^{\sigma}_{n+1};\mathbf{A}(N_k),N_k,k,1/k). 
\]
By definition, for each $k$ there exists $\mathbf{U}(N_k) \in \mathrm{U}(N_k)^n$ such that $\mathrm{tr}_{\mathbf{U}(N_k)}^{\mathbf{A}(N_k)} \in O_{k,1/k}(\sigma)$. With $\mathbf{U}(N_k) = (U_i(N_k))_{i=1}^n$ we define $\mathbf{B}(N_k) = ((B_{ij}(N_k))_{j=1}^{r(i)})_{i=1}^{n+1}$ by
\[
B_{ij}(N_k) := 
\begin{cases} 
U_i(N_k) A_{ij}(N_k) U_i(N_k)^* & (1 \leq i \leq n), \\
A_{n+1 j}(N_k) & (i = n+1). 
\end{cases}
\]
Let $\Xi = (\Xi(N))_{N\in\mathbb{N}}$ with $\Xi_i(N) = (\xi_{ij}(N))_{j=1}^{r(i)}$, $1 \leq i \leq n+1$, be the one chosen at the beginning of this proof. (The existence of such a sequence follows from $\chi_\mathrm{orb}(\mathbf{X}^{\sigma}_1,\dots,\mathbf{X}^{\sigma}_{n+1}) > -\infty$; see e.g.\ \cite[Lemma 2.1]{HiaiUeda:CMP15}.) Define $\Xi' = (\Xi'(N))_{N\in\mathbb{N}}$ by 
\[
\Xi'(N) := 
\begin{cases} 
\mathbf{B}(N_k) & (N = N_k), \\
\Xi(N) & (\text{otherwise}). 
\end{cases}
\]
Since 
\[
\mathrm{tr}^{\Xi'(N_k)} = \mathrm{tr}_{\mathbf{U}(N_k)}^{\mathbf{A}(N_k)} \in O_{k,1/k}(\sigma),
\]
it is easy to see that $\mathrm{tr}^{\Xi'(N)}$ converges to $\sigma$ in the weak$^*$ topology on $TS(C^*_R\langle x_{\bullet\diamond}\rangle)$. Since
\[
\mathrm{tr}^{\Xi'(N_k)}_\mathbf{U} = \mathrm{tr}_{(U_i U_i(N_k))_{i=1}^n}^{\mathbf{A}(N_k)}, 
\qquad \mathbf{U} = (U_i)_{i=1}^n \in \mathrm{U}(N_k)^n
\]
for every $k$ and since $\nu_N$ is invariant under right-multiplication, we observe that 
\[
\chi_\mathrm{orb}(\mathbf{X}^{\sigma}_1,\dots
,\mathbf{X}^{\sigma}_{n+1}; \mathbf{A}(N_k), N_k, k, 1/k) = \chi_\mathrm{orb}(\sigma\,|\,\Xi'\,;N_k,k,1/k) 
\]
for every $k$. Thus, for each $m \in \mathbb{N}$, $\delta>0$, we have 
\[
\chi_\mathrm{orb}(\sigma\,|\,\Xi'\,;N_k,k,1/k) 
\leq 
\chi_\mathrm{orb}(\sigma\,|\,\Xi'\,;N_k,m,\delta)
\]
for all sufficiently large $k$. Thus, for every $m \in \mathbb{N}$, $\delta>0$, we obtain that 
\begin{align*} 
\chi_\mathrm{orb}(\mathbf{X}^{\sigma}_1,\dots,\mathbf{X}^{\sigma}_{n+1}) 
&= 
\lim_{k\to\infty}\frac{1}{N_k^2} \bar{\chi}_\mathrm{orb}(\mathbf{X}^{\sigma}_1,\dots,\mathbf{X}^{\sigma}_{n+1}; N_k,k,1/k) \\
&= 
\lim_{k\to\infty} \frac{1}{N_k^2}\Big(\bar{\chi}_\mathrm{orb}(\mathbf{X}^{\sigma}_1,\dots,\mathbf{X}^{\sigma}_{n+1}; N_k,k,1/k) - 1\Big) \\
&\leq 
\varlimsup_{k\to\infty} \frac{1}{N_k^2} \chi_\mathrm{orb}(\sigma\,|\,\Xi'\,;N_k,m,\delta)  \\
&\leq 
\varlimsup_{N\to\infty}\frac{1}{N^2} \chi_\mathrm{orb}(\sigma\,|\,\Xi'\,;N,m,\delta) \\
&= 
\chi_\mathrm{orb}(\sigma\,|\,\Xi'; m,\delta).
\end{align*}
Therefore, by taking the limit as $m\to\infty$, $\delta\searrow0$ we have 
\[
\chi_\mathrm{orb}(\mathbf{X}^{\sigma}_1,\dots,\mathbf{X}^{\sigma}_{n+1}) \leq \chi_\mathrm{orb}(\sigma_0\,|\,\Xi') \leq \chi_\mathrm{orb}(\sigma\,|\,\sigma).
\]
With the former assertion we are done.  
\end{proof}

Another natural choice of initial tracial state $\sigma_0$ is available; the tracial state is determined by making the resulting random multi-variables $\mathbf{X}^{\sigma_0}_i$, $1 \leq i \leq n+1$, freely independent. The $\chi_\mathrm{orb}(\sigma\,|\,\sigma_0)$ with this choice of $\sigma_0$ is nothing but an unpublished variation of orbital free entropy due to Dabrowski, and the proposition below shows that it turns out to be the same as our original  $\chi_\mathrm{orb}(\mathbf{X}^\sigma_1,\dots,\mathbf{X}^\sigma_{n+1})$ in \cite{Ueda:IUMJ14}.

\begin{proposition}\label{P3.3} 
When the $\mathbf{X}^{\sigma_0}_i$, $1 \leq i \leq n+1$, are freely independent, then $\chi_\mathrm{orb}(\sigma\,|\,\sigma_0) = \chi_\mathrm{orb}(\mathbf{X}^\sigma_1,\dots,\mathbf{X}^\sigma_{n+1})$. 
\end{proposition}
\begin{proof} 
By Proposition \ref{P3.2} we may and do assume $\chi_\mathrm{orb}(\mathbf{X}^\sigma_1,\dots,\mathbf{X}^\sigma_{n+1}) > -\infty$, and it suffices to prove 
$$
\chi_\mathrm{orb}(\sigma\,|\,\sigma_0) \geq \chi_\mathrm{orb}(\sigma\,|\,\sigma)\ \big(= \chi_\mathrm{orb}(\mathbf{X}^\sigma_1,\dots,\mathbf{X}^\sigma_{n+1})\big). 
$$
To this end, let $\Xi = (\Xi(N))_{N=1}^\infty$ with $\Xi(N) = (\Xi_i(N))_{i=1}^{n+1}$, $\Xi_i(N) = (\xi_{ij}(N))_{j=1}^{r(i)} \in ((M_N^{sa})_R)^{r(i)}$, $1 \leq i \leq n+1$, be such that $\lim_{N\to\infty}\mathrm{tr}^{\Xi(N)} = \sigma$ in the weak$^*$ topology. Choose an independent family of Haar-distributed unitary random matrices $V_N^{(i)}$, $1 \leq i \leq n$. It is known, see e.g.\ \cite[Theorem 4.3.1]{HiaiPetz:Book}, that $V_N^{(1)},\dots,V_N^{(n)}, \Xi(N)$ are asymptotically free  almost surely as $N\to\infty$ and moreover that the subfamily $V_N^{(1)},\dots,V_N^{(n)}$ converges to a freely independent family of Haar unitaries in distribution almost surely as $N\to\infty$ too. Thus, thanks to the almost sure convergence, we can choose deterministic sequences $V_i(N)$, $1 \leq i \leq n$, from random sequences $V_N^{(i)}$, $1 \leq i \leq n$ such that $V_1(N),\dots,V_n(N),\Xi(N)$ converge to the same family of non-commutative random variables in distribution as $N\to\infty$. Define $\Xi' = (\Xi'(N))_{N=1}^\infty$ with $\Xi'(N) = (\Xi'_i(N))_{i=1}^{n+1}$, $\Xi'_i(N) = (\xi'_{ij}(N))_{j=1}^{r(i)}$ by  
\[
\xi'_{ij}(N) := 
\begin{cases} 
V_i(N)\xi_{ij}(N)V_i(N)^* &(1 \leq i \leq n), \\
\xi_{n+1 j}(N) &(i=n+1). 
\end{cases}
\]
Then, the $\Xi'_i(N)$, $1 \leq i \leq n+1$, are asymptotically free as $N\to\infty$. Therefore, we conclude that $\lim_{N\to\infty}\mathrm{tr}^{\Xi'(N)} = \sigma_0$ in the weak$^*$ topology. Remark that 
\[
\mathrm{tr}^{\Xi'(N)}_\mathbf{U} = \mathrm{tr}_{(U_i V_i(N))_{i=1}^n}^{\Xi(N)}, 
\qquad \mathbf{U} = (U_i)_{i=1}^n \in \mathrm{U}(N)^n 
\]
holds for every $N$. Therefore, thanks to the invariance of $\nu_N$ under right-multiplication, we conclude, as in the proof of Proposition \ref{P3.2}, that 
\[
\chi_\mathrm{orb}(\sigma\,|\,\Xi) = \chi_\mathrm{orb}(\sigma\,|\,\Xi') \leq \chi_\mathrm{orb}(\sigma\,|\,\sigma_0). 
\]
Since $\Xi$ has arbitrarily been chosen, we are done. 
\end{proof}  

The above proof suggests that $\chi_\mathrm{orb}(\sigma\,|\,\sigma_0)$ coincides with $\chi_\mathrm{orb}(\mathbf{X}^\sigma_1,\dots,\mathbf{X}^\sigma_{n+1})$ for a large class of tracial states $\sigma_0 \in TS_\mathrm{fda}(C^*_R\langle x_{\bullet\diamond}\rangle)$. 

\section{Orbital free entropy and Matrix liberation process} 

Building on our previous work \cite{Ueda:JOTP19} we will clarify how some fundamental questions concerning the orbital free entropy $\chi_\mathrm{orb}$ are precisely reduced to the conjectural large deviation principle for the matrix liberation process. Lemma \ref{L2.1} will play a key role in what follows. 

\subsection{Non-commutative coordinates} 
Let $C^*_R\langle x_{\bullet\diamond}(\,\cdot\,)\rangle \subset C^*_R\langle x_{\bullet\diamond}(\,\cdot\,),v_\bullet(\,\cdot\,)\rangle$ be the universal unital $C^*$-algebras generated by $x_{ij}(t) = x_{ij}(t)^*$, $1 \leq i \leq n+1, j \geq 1, t \geq 0$, and $v_i(t)$, $1 \leq i \leq n$, $t \geq 0$, with subject to $\Vert x_{ij}(t)\Vert_\infty \leq R$ and $v_i(t)^* v_i(t) = v_i(t)v_i(t)^* = 1 = v_i(0)$. These universal $C^*$-algebras are constructed as universal free products of uncountably many $C[-R,R]$ and $C(\mathbb{T})$, and generators $x_{ij}(t)$ and $u_i(t)$ are given by coordinate functions $f(t) = t$ in $t \in [-R,R]$ or $g(z) = z$ in $z \in \mathbb{T}$ of component algebras. Proposition \ref{PA3} guarantees the inclusion of two universal $C^*$-algebras. Recall that $j$ may run over the natural numbers $\mathbb{N}$ as we remarked at the end of section 1. The universal $*$-algebras $\mathbb{C}\langle x_{\bullet\diamond}(\,\cdot\,)\rangle \subset \mathbb{C}\langle x_{\bullet\diamond}(\,\cdot\,),v_\bullet(\,\cdot\,)\rangle$ generated by the same indeterminates $x_{ij}(t)$ and $v_i(t)$ can naturally be regarded as norm-dense $*$-subalgebras of  $C^*_R\langle x_{\bullet\diamond}(\,\cdot\,)\rangle \subset C^*_R\langle x_{\bullet\diamond}(\,\cdot\,),v_\bullet(\,\cdot\,)\rangle$, respectively. Proposition \ref{PA4} guarantees this fact. For each $T\geq0$, the correspondence $x_{ij} \mapsto x_{ij}(T)$, $1\leq i \leq n+1, j \geq 1$, defines a unique (injective) $*$-homomorphism $\pi_T : C^*_R\langle x_{\bullet\diamond}\rangle \to C^*_R\langle x_{\bullet\diamond}(\,\cdot\,)\rangle$ with notation $C^*_R\langle x_{\bullet\diamond}\rangle $ in section 3.  

\subsection{Time-dependent liberation derivative}
We introduce the derivation 
\[
\delta_s^{(k)} : \mathbb{C}\langle x_{\bullet\diamond}(\,\cdot\,)\rangle \to  \mathbb{C}\langle x_{\bullet\diamond}(\,\cdot\,),v_\bullet(\,\cdot\,)\rangle\otimes_\mathrm{alg} \mathbb{C}\langle x_{\bullet\diamond}(\,\cdot\,),v_\bullet(\,\cdot\,)\rangle, 
\quad 1 \leq k \leq n, s \geq 0, 
\]
which sends each $x_{ij}(t)$ to 
\[ 
\delta_{i,k}\mathbf{1}_{[0,t]}(s)\big(x_{kj}(t)v_k(t-s)\otimes v_k(t-s)^* - v_k(t-s)\otimes v_k(t-s)^* x_{kj}(t)\big).
\]
Then we write $\mathfrak{D}_s^{(k)} := \theta\circ\delta_s^{(k)}$, $1\leq k\leq n, s\geq0$, where $\theta$ denotes the flip-multiplication mapping $a\otimes b \mapsto ba$. 

\subsection{Continuous tracial states}
A tracial state $\tau$ on $C^*_R\langle x_{\bullet\diamond}(\,\cdot\,)\rangle$ is said to be continuous if $t \mapsto \pi_\tau(x_{ij}(t))$ is strongly continuous for every $1\leq i \leq n+1, j \geq 1$, where $\pi_\tau : C^*_R\langle x_{\bullet\diamond}(\,\cdot\,)\rangle \curvearrowright \mathcal{H}_\tau$ is the GNS representation associated with $\tau$. We denote by $TS^c(C^*_R\langle x_{\bullet\diamond}(\,\cdot\,)\rangle)$ all the continuous tracial states. The space $TS^c(C^*_R\langle x_{\bullet\diamond}(\,\cdot\,)\rangle)$ becomes a complete metric space endowed with metric $d$ defined by \eqref{Eq1.1}, which defines the topology of uniform convergence on finite time intervals. 

\subsection{Liberation process $\tau^s$ starting at a given time}
We extend a given $\tau \in TS^c(C^*_R\langle x_{\bullet\diamond}(\,\cdot\,)\rangle)$ to a unique $\tilde{\tau} \in TS^c(C^*_R\langle x_{\bullet\diamond}(\,\cdot\,),v_\bullet(\,\cdot\,)\rangle)$ in such a way that the $v_i(t)$ are $*$-freely independent of $C^*_R\langle x_{\bullet\diamond}(\,\cdot\,)\rangle$ and form a $*$-freely independent family of left-multiplicative free unitary Brownian motions under this extension $\tilde{\tau}$. This extension of tracial state can be constructed, via the GNS representation $\pi_\tau : C^*_R\langle x_{\bullet\diamond}(\,\cdot\,)\rangle \curvearrowright \mathcal{H}_\tau$, by taking a suitable reduced free product. We write 
\[
(\mathcal{N}(\tau) \subset \mathcal{M}(\tau)) := \big(\pi_{\tilde{\tau}}(C^*_R\langle x_{\bullet\diamond}(\,\cdot\,)\rangle)'' \subset \pi_{\tilde{\tau}}(C^*_R\langle x_{\bullet\diamond}(\,\cdot\,),v_\bullet(\,\cdot\,)\rangle)''\big)
\] 
on $\mathcal{H}_\tau$, where $\pi_{\tilde{\tau}} : C_R^*\langle x_{\bullet\diamond}(\,\cdot\,),v_\bullet(\,\cdot\,)\rangle \curvearrowright \mathcal{H}_ {\tilde{\tau}}$ is the GNS representation associated with ${\tilde{\tau}}$. Write $x_{ij}^\tau(t) := \pi_{\tilde{\tau}}(x_{ij}(t))$ and $v_i^\tau(t) := \pi_{\tilde{\tau}}(v_i(t))$ and the canonical extension of ${\tilde{\tau}}$ to $\mathcal{M}(\tau)$ is still denoted by the same symbol ${\tilde{\tau}}$ for simplicity. We denote by $E_{\mathcal{N}(\tau)}$ the ${\tilde{\tau}}$-preserving conditional expectation from $\mathcal{M}(\tau)$ onto $\mathcal{N}(\tau)$, which is known to exist and to be unique as a standard fact on von Neumann algebras. Consider an `abstract' non-commutative process in $C^*_R\langle x_{\bullet\diamond}(\,\cdot\,),v_\bullet(\,\cdot\,)\rangle$
\[
t \mapsto x_{ij}^s(t) := 
\begin{cases} 
v_i((t-s)\vee0)x_{ij}(s\wedge t)v_i((t-s)\vee0)^*  & (1 \leq i \leq n), \\
x_{n+1\,j}(t) & (i=n+1) 
\end{cases}
\]
and the corresponding `concrete' non-commutative stochastic process in $\mathcal{M}(\tau)$
\[
t \mapsto x_{ij}^{\tau^s}(t) := \pi_{\tilde{\tau}}(x_{ij}^s(t)) = 
\begin{cases}
v_i^\tau((t-s)\vee0) x_{ij}^\tau(s\wedge t) v_i^\tau((t-s)\vee0)^* & (1 \leq i \leq n), \\
x_{n+1\,j}^\tau(t) & (i=n+1). 
\end{cases} 
\] 
By universality, this process $x_{ij}^{\tau^s}(t)$ clearly defines a tracial state $\tau^s \in TS^c(C^*_R\langle x_{\bullet\diamond}(\,\cdot\,)\rangle)$. 

By the $*$-homomorphism $\Gamma : C^*_R\langle x_{\bullet\diamond}(\,\cdot\,)\rangle \to C^*_R\langle x_{\bullet\diamond}\rangle$ sending each $x_{ij}(t)$ to $x_{ij}$, we obtain $\Gamma^*(\sigma_0) := \sigma_0 \circ \Gamma \in TS(C^*_R\langle x_{\bullet\diamond}(\,\cdot\,) \rangle)$ with a given $\sigma_0 \in TS(C^*_R\langle x_{\bullet\diamond}\rangle)$ and set $\sigma_0^\mathrm{lib} := \Gamma^*(\sigma_0)^0 \in TS(C^*_R\langle x_{\bullet\diamond}(\,\cdot\,) \rangle)$ ($\Gamma^*(\sigma_0)^0$ is defined in the same way as $\tau^s$ with $s=0$), which we call \emph{the liberation process starting at $\sigma_0$} (precisely its empirical distribution).  

\subsection{New description of $\tau^s$}
By universality, we have a unique unital $*$-homomorphism  
$\Pi^s : C^*_R\langle x_{\bullet\diamond}(\,\cdot\,),v_\bullet(\,\cdot\,)\rangle \to C^*_R\langle x_{\bullet\diamond}(\,\cdot\,),v_\bullet(\,\cdot\,)\rangle$ sending $x_{ij}(t)$ and $v_i(t)$ to $x_{ij}^s(t)$ and $v_i(t)$, respectively. By using this $*$-homomorphism we obtain a unital $*$-homomorphism 
\[
\begin{matrix} 
\pi_{\tilde{\tau}}\circ\Pi^s & : & C_R^*\langle x_{\bullet\diamond}(\,\cdot\,),v_\bullet(\,\cdot\,)\rangle 
& \overset{\Pi^s}{\longrightarrow} & 
C_R^*\langle x_{\bullet\diamond}(\,\cdot\,),v_\bullet(\,\cdot\,)\rangle 
&\overset{\pi_{\tilde{\tau}}}{\longrightarrow} & 
\mathcal{M}(\tau) \\
&&(x_{ij}(t), v_i(t))  
& \mapsto & 
(x_{ij}^s(t), v_i(t)) 
& \mapsto & 
(x_{ij}^{\tau^s}(t), v_i^\tau(t)). 
\end{matrix} 
\]
Then $\pi_{\tilde{\tau}}(\Pi^s(\mathfrak{D}_s^{(k)}P))$, $P \in \mathbb{C}\langle x_{\bullet\diamond}(\,\cdot\,)\rangle$, becomes the element of $\mathcal{M}(\tau)$ obtained by substituting $(x_{ij}^{\tau^s}(t), v^\tau_i(t))$ for $(x_{ij}(t), v_i(t))$ in $\mathfrak{D}_s^{(k)}P$. Moreover, we have $\tau^s = \tilde{\tau}\circ\Pi^s$ on $C^*_R\langle x_{\bullet\diamond}(\,\cdot\,)\rangle$.    

\subsection{Rate function}
To a given $\sigma_0 \in TS(C^*_R\langle x_{\bullet\diamond}\rangle)$
we associate two functionals $I_{\sigma_0}^\mathrm{lib}, I_{\sigma_0,\infty}^\mathrm{lib} : TS^c(C_R^*\langle x_{\bullet\diamond}(\,\cdot\,)\rangle) \to [0,+\infty]$ as follows. For any $\tau \in TS^c(C_R^*\langle x_{\bullet\diamond}(\,\cdot\,)\rangle$, $P = P^* \in \mathbb{C}\langle x_{\bullet\diamond}(\,\cdot\,)\rangle$ and $t \in [0,\infty]$ we first define 
\begin{equation}\label{Eq9}
I_{\sigma_0,t}^\mathrm{lib}(\tau,P) := \tau^t(P) - \sigma_0^\mathrm{lib}(P) - \frac{1}{2}\sum_{k=1}^n \int_0^t \Vert E_{\mathcal{N}(\tau)}(\pi_{\tilde{\tau}}(\Pi^s(\mathfrak{D}_s^{(k)}P)))\Vert^2_{\tilde{\tau},2}\,ds
\end{equation}
with regarding $\tau$ as $\tau^\infty$ (since $\tau^t(P) = \tau(P)$ when $t$ is large enough), where $\Vert\,-\,\Vert_{\tilde{\tau},2}$ denotes the $2$-norm on the tracial $W^*$-probability space $(\mathcal{M}(\tau),\tilde{\tau})$. We remark that the integrand in \eqref{Eq9} agrees with that given in  \cite{Ueda:JOTP19} (though their representations are different at first glance), and moreover that the integration above is well defined even when $t=\infty$, because $\mathfrak{D}^{(k)}_s P = 0$ when $s$ is large enough. Then we define  
\[
I_{\sigma_0}^\mathrm{lib}(\tau) := \sup_{\substack{ P = P^* \in \mathbb{C}\langle x_{\bullet\diamond}(\,\cdot\,)\rangle \\ t > 0}} I_{\sigma_0,t}^\mathrm{lib}(\tau,P), \quad 
I_{\sigma_0,\infty}^\mathrm{lib}(\tau) := \sup_{P = P^* \in \mathbb{C}\langle x_{\bullet\diamond}(\,\cdot\,)\rangle} I_{\sigma_0,\infty}^\mathrm{lib}(\tau,P). 
\]
\emph{Each of the functionals $I_{\sigma_0}^\mathrm{lib}, I_{\sigma_0,\infty}^\mathrm{lib}$ is shown, in \cite[Proposition 5.6, Proposition 5.7(3)]{Ueda:JOTP19}} ({\it n.b.}, their proofs work well even for the modification $I_{\sigma_0,\infty}^\mathrm{lib}$ without any essential changes), \emph{to be a well-defined, good rate function with unique minimizer. Moreover, the minimizer for both functionals is identified with the liberation process $\sigma_0^\mathrm{lib}$ starting at $\sigma_0$ for both functionals.} Remark that the proofs of \cite[Proposition 5.6, Proposition 5.7(3)]{Ueda:JOTP19} do not use the assumption that $\sigma_0$ falls into $TS_\mathrm{fda}(C^*_R\langle x_{\bullet\diamond}\rangle)$, and thus the functionals $I_{\sigma_0}^\mathrm{lib}, I_{\sigma_0,\infty}^\mathrm{lib}$ can be considered in the general setting. Remark that $I_{\sigma_0,\infty}^\mathrm{lib}(\tau) \leq I_{\sigma_0}^\mathrm{lib}(\tau)$ obviously holds, but it is a question whether equality holds or not.   

\medskip
Here is a simple lemma, which can be applied to $I = I_{\sigma_0}^\mathrm{lib}$ or $I = I_{\sigma_0,\infty}^\mathrm{lib}$. Recall that $\pi_T : C^*_R\langle x_{\bullet\diamond}\rangle \to C^*_R\langle x_{\bullet\diamond}(\,\cdot\,)\rangle$ is the unique injective $*$-homomorphism sending each $x_{ij}$ to $x_{ij}(T)$. In the lemma below, we will use the map $\pi_T^* : TS^c(C^*\langle x_{\bullet\diamond}(\,\cdot\,)\rangle) \to TS^c(C^*\langle x_{\bullet\diamond}\rangle)$ induced from $\pi_T$, see the glossary in section 1.

\begin{lemma}\label{L4.1} For any functional $I : TS^c(C^*\langle x_{\bullet\diamond}(\,\cdot\,)\rangle) \to [0,+\infty]$, the new one $J : TS(C^*_R\langle x_{\bullet\diamond}\rangle) \to [0,+\infty]$ defined by  
\begin{align*}
J(\sigma) 
:&= \lim_{\substack{m\to\infty \\ \delta \searrow 0}} \varlimsup_{T\to\infty} \inf\big\{ I(\tau) \mid \tau \in TS^c(C^*_R\langle x_{\bullet\diamond}(\,\cdot\,)\rangle), \pi_T^*(\tau) \in O_{m,\delta}(\sigma) \big\} \\
&= \sup_{\substack{m\in\mathbb{N} \\ \delta > 0}} \varlimsup_{T\to\infty} \inf\big\{ I(\tau) \mid \tau \in TS^c(C^*_R\langle x_{\bullet\diamond}(\,\cdot\,)\rangle), \pi_T^*(\tau) \in O_{m,\delta}(\sigma) \big\}
\end{align*}
for any $\sigma \in TS(C^*_R\langle x_{\bullet\diamond}\rangle)$ (with notation $O_{m,\delta}(\sigma)$ in the previous section) is a well-defined rate function, where $TS(C^*_R\langle x_{\bullet\diamond}\rangle)$ is endowed with the weak$^*$ topology and the infimum over the empty set is taken to be $+\infty$. Moreover, replacing $O_{m,\delta}(\sigma)$ with the closed neighborhood $F_{m,\delta}(\sigma)$ in the above definition of $J(\sigma)$ does not affect its value, where $F_{m,\delta}(\sigma)$ is all the $\sigma' \in TS(C^*_R\langle x_{\bullet\diamond}\rangle)$ such that 
\[
|\sigma'(x_{i_1 j_1}\cdots x_{i_p j_p}) - \sigma(x_{i_1 j_1}\cdots x_{i_p j_p})| \leq \delta
 \]
whenever $1 \leq i_k \leq n+1$, $1 \leq j_k \leq m$, $1 \leq k \leq p$ and $1 \leq p \leq m$. 
\end{lemma}
\begin{proof} 
If $m_1 \leq m_2$ and $\delta_1 \geq \delta_2 > 0$, then $O_{m_1,\delta_1}(\sigma) \supseteq O_{m_2,\delta_2}(\sigma)$ so that 
\begin{align*}
&\varlimsup_{T\to\infty}\inf\big\{ I(\tau) \mid \tau \in TS^c(C_R^*\langle x_{\bullet\diamond}(\,\cdot\,)\rangle), \pi_T^*(\tau) \in O_{m_1,\delta_1}(\sigma) \big\} \\
&\qquad \leq 
\varlimsup_{T\to\infty}\inf\big\{ I(\tau) \mid \tau \in TS^c(C_R^*\langle x_{\bullet\diamond}(\,\cdot\,)\rangle), \pi_T^*(\tau) \in O_{m_2,\delta_2}(\sigma) \big\}.
\end{align*}
Therefore, taking $\lim_{m\to\infty, \delta\searrow0}$ in the definition of $J(\sigma)$ is actually well defined and coincides with taking the supremum all over $m \in \mathbb{N}$ and $\delta>0$. 

We then confirm that $J$ is lower semicontinuous. Assume that $\sigma_k \to \sigma$ in $TS(C^*_R\langle x_{\bullet\diamond}\rangle)$ as $k\to\infty$. Choose an arbitrary $0 \leq L < J(\sigma)$. Then there exist $m_0 \in \mathbb{N}$ and $\delta_0 > 0$ such that 
\[
\varlimsup_{T\to\infty} \inf\big\{ I(\tau) \mid \tau \in TS^c(C_R^*\langle x_{\bullet\diamond}(\,\cdot\,)\rangle), \pi_T^*(\tau) \in O_{m_0,\delta_0}(\sigma_k) \big\} > L. 
\]
Then, there exists $k_0 \in \mathbb{N}$ such that if $k \geq k_0$, then $O_{m_0,\delta_0/2}(\sigma_k) \subseteq O_{m_0,\delta_0}(\sigma)$ and hence
\begin{align*} 
J(\sigma_k)
&\geq  
\varlimsup_{T\to\infty} \inf\big\{ I(\tau) \mid \tau \in TS^c(C_R^*\langle x_{\bullet\diamond}(\,\cdot\,)\rangle), \pi_T^*(\tau) \in O_{m_0,\delta_0/2}(\sigma_k) \big\} \\
&\geq 
\varlimsup_{T\to\infty} \inf\big\{ I(\tau) \mid \tau \in TS^c(C_R^*\langle x_{\bullet\diamond}(\,\cdot\,)\rangle), \pi_T^*(\tau) \in O_{m_0,\delta_0}(\sigma) \big\}  > L, 
\end{align*}
where the first inequality follows from the fact that $\lim_{m\to\infty, \delta\searrow 0} = \sup_{m,\delta}$ in the definition of $J(\sigma)$ as remarked before. Therefore, we obtain that $\varliminf_{k\to\infty} J(\sigma_k) \geq L$, which guarantees that $J$ is lower semicontinuous. 

Since $O_{m,\delta}(\sigma) \subseteq F_{m,\delta}(\sigma) \subseteq O_{m,2\delta}(\sigma)$, we have 
\begin{align*}
&\inf\big\{ I(\tau) \mid \tau \in TS^c(C_R^*\langle x_{\bullet\diamond}(\,\cdot\,)\rangle), \pi_T^*(\tau) \in O_{m,\delta}(\sigma_k) \big\} \\
&\qquad\geq 
\inf\big\{ I(\tau) \mid \tau \in TS^c(C_R^*\langle x_{\bullet\diamond}(\,\cdot\,)\rangle), \pi_T^*(\tau) \in F_{m,\delta}(\sigma_k) \big\} \\
&\qquad\qquad\geq
\inf\big\{ I(\tau) \mid \tau \in TS^c(C_R^*\langle x_{\bullet\diamond}(\,\cdot\,)\rangle), \pi_T^*(\tau) \in O_{m,2\delta}(\sigma_k) \big\}
\end{align*}
for every $m\in\mathbb{N}$ and $\delta>0$. This implies the last assertion.  
\end{proof} 

The above lemma clearly holds true even if $\varlimsup_{T\to\infty}$ is replaced with $\varliminf_{T\to\infty}$ in the definition of $J$. We also remark that $TS(C^*_R\langle x_{\bullet\diamond}\rangle)$ is weak$^*$ compact, and hence $J$ is trivially a good rate function. 

\subsection{Matrix liberation process}
Let $\Xi(N) = ((\xi_{ij}(N))_{j=1}^{r(i)})_{i=1}^{n+1}$ with $\xi_{ij}(N) \in (M_N^{sa})_R$ be an approximation to a given $\sigma_0 \in TS_\mathrm{fda}(C^*_R\langle x_{\bullet\diamond}\rangle)$. Let $U_N^{(i)}(t)$, $1\leq i\leq n$, be independent, left-increment unitary Brownian motions on $\mathrm{U}(N)$, and we define the matrix liberation process $\Xi^\mathrm{lib}(N)(t) = ((\xi_{ij}^\mathrm{lib}(N)(t))_{j=1}^{r(i)})_{i=1}^n$, $t\geq0$,  starting at $\Xi(N)$ by
\[
\xi_{ij}^\mathrm{lib}(N)(t) := 
\begin{cases}
U_N^{(i)}(t)\xi_{ij}(N)U_N^{(i)}(t)^* & (1 \leq i \leq n), \\
\xi_{n+1 j}(N) & (i=n+1). 
\end{cases}
\]
Then, via the $*$-homomorphism $\pi_{\Xi^\mathrm{lib}(N)} : C^*_R\langle x_{\bullet\diamond}(\,\cdot\,)\rangle \to M_N$ determined by $x_{ij}(t) \mapsto \xi_{ij}^\mathrm{lib}(N)(t)$, $1\leq i \leq n+1, j \geq 1, t\geq0$, we obtain a tracial state $\tau_{\Xi^\mathrm{lib}(N)} := \mathrm{tr}_N\circ\pi_{\Xi^\mathrm{lib}(N)}$, which falls into $TS^c(C^*_R\langle x_{\bullet\diamond}(\,\cdot\,)\rangle)$. This tracial state is a random variable in $TS^c(C^*_R\langle x_{\bullet\diamond}(\,\cdot\,)\rangle)$ in the ordinary sense, and hence we can consider the probability $\mathbb{P}(\tau_{\Xi^\mathrm{lib}(N)} \in \Theta)$ of any Borel subset $\Theta \subseteq TS^c(C^*_R\langle x_{\bullet\diamond}(\,\cdot\,)\rangle)$. By \cite[Theorem 5.8]{Ueda:JOTP19} we already know that \emph{the sequence of probability measures $\mathbb{P}(\tau_{\Xi^\mathrm{lib}(N)} \in \cdot\,)$ satisfies the large deviation upper bound with speed $N^2$ and the above rate function $I^\mathrm{lib}_{\sigma_0}$}. 

\subsection{Contraction principle at $T=\infty$}
Let $\mathbf{U}_N = (U_N^{(i)})_{i=1}^n$ be an $n$-tuple of independent $N\times N$ unitary random matrices distributed under the Haar probability measure $\nu_N$ on $\mathrm{U}(N)$. The random tracial state $\mathrm{tr}_{\mathbf{U}_N}^{\Xi(N)} \in TS(C_R^*\langle x_{\bullet\diamond}\rangle)$ is defined in the same manner as in \S3. A well-known, standard result on the heat kernel measure on $\mathrm{U}(N)$ implies that $\mathbb{E}[\pi_T^*(\tau_{\Xi^\mathrm{lib}(N)})(a)]$ converges to $\mathbb{E}[\mathrm{tr}_{\mathbf{U}_N}^{\Xi(N)}(a)]$ as $T\to\infty$ for every $a \in C^*_R\langle x_{\bullet\diamond}\rangle$. The usual method to obtain the large deviation upper/lower bound with speed $N^2$ for $\mathbb{P}(\mathrm{tr}_{\mathbf{U}_N}^{\Xi(N)}\in\,\cdot\,)$ from that for $\mathbb{P}(\tau_{\Xi^\mathrm{lib}(N)} \in \,\cdot\,)$ in the same scale is to show that (a kind of) the exponential convergence of $\pi_T^*(\tau_{\Xi^\mathrm{lib}(N)})$ to $\mathrm{tr}_{\mathbf{U}_N}^{\Xi(N)}$ as $T\to\infty$ (see e.g.\ \cite[\S4.2.2]{DemboZeitouni:Book}). Nevertheless, we will be able to prove the next proposition by utilizing Lemma \ref{L2.1} without establishing the exponential convergence.

\begin{proposition}\label{P4.2} Assume that the sequence of probability measures $\mathbb{P}(\tau_{\Xi^\mathrm{lib}(N)} \in \,\cdot\,)$ satisfies the large deviation upper (lower) bound with speed $N^2$ and rate function $I^+$ (resp.\ $I^-$). Then $\mathbb{P}(\mathrm{tr}_{\mathbf{U}_N}^{\Xi(N)}\in\,\cdot\,)$ also satisfies the large deviation upper (resp.\ lower) bound with speed $N^2$ and the following rate function:
\begin{align*}
J^+(\sigma) &:= \lim_{\substack{m\to\infty \\ \delta\searrow0}}\varlimsup_{T\to\infty}\inf\{I^+(\tau) \mid \tau \in TS^c(C_R^*\langle x_{\bullet\diamond}(\,\cdot\,)\rangle), \pi_T^*(\tau) \in O_{m,\delta}(\sigma)\} \\
\Big(\text{resp.}\quad
J^-(\sigma) &:= \lim_{\substack{m\to\infty \\ \delta\searrow0}}\varliminf_{T\to\infty}\inf\{I^-(\tau) \mid \tau \in TS^c(C_R^*\langle x_{\bullet\diamond}(\,\cdot\,)\rangle), \pi_T^*(\tau) \in O_{m,\delta}(\sigma)\}\,\Big)
\end{align*}
for every $\sigma \in TS(C^*_R\langle x_{\bullet\diamond}\rangle)$, where the infimum over the empty set is taken to be $+\infty$. 

In particular, if the sequence of probability measures $\mathbb{P}(\tau_{\Xi^\mathrm{lib}(N)} \in \,\cdot\,)$ satisfies the full large deviation principle with speed $N^2$, that is, the above large deviation upper and lower bounds with $I^+=I^-$, then $J := J^+ = J^-$ and
\begin{align*}
\chi_\mathrm{orb}(\sigma\,|\,\sigma_0) 
&= 
\chi_\mathrm{orb}(\sigma\,|\,\Xi) = -J(\sigma) \\
&= 
\lim_{\substack{m\to\infty\\\delta\searrow0}}\varliminf_{N\to\infty}\frac{1}{N^2}\log \nu_N^{\otimes n}(\{ \mathbf{U} \in \mathrm{U}(N)^n \mid \mathrm{tr}_{\mathbf{U}}^{\Xi(N)} \in O_{m,\delta}(\sigma)\})
\end{align*}
holds for every $\sigma \in TS(C^*_R\langle x_{\bullet\diamond}\rangle)$ and any choice of approximating sequence $\Xi = (\Xi(N))_{N\in\mathbb{N}}$ to $\sigma_0 \in TS_\mathrm{fda}(C^*_R\langle x_{\bullet\diamond}\rangle)$.
\end{proposition}
\begin{proof} 
Set 
\[
I_T^\pm(\sigma) := \inf\{ I^\pm(\tau) \mid \tau \in TS^c(C_R^*\langle x_{\bullet\diamond}(\,\cdot\,)\rangle),  \pi_T^*(\tau) = \sigma \}, \quad \sigma \in TS(C^*_R\langle x_{\bullet\diamond}\rangle). 
\]
By the contraction principle (see e.g.\ \cite[Theorem 4.2.1]{DemboZeitouni:Book}), $\mathbb{P}(\pi_T^*(\tau_{\Xi^\mathrm{lib}(N)}) \in \,\cdot\,)$ satisfies the large deviation upper (resp.\ lower) bound with speed $N^2$ and the rate function $I_T^+$ (resp.\ $I_T^-$). Write $\mathbf{U}_N(t) = \big(U_N^{(i)}(t)\big)_{i=1}^n$, $t \geq 0$, and define the random tracial state $\mathrm{tr}_{\mathbf{U}_N(T)}^{\Xi(N)}$ in the same manner as $\mathrm{tr}_{\mathbf{U}_N}^{\Xi(N)}$. Let $L(T) \leq U(T)$ as well as $\nu_{N,T}$ and $\nu_N$ be as in the previous sections. Observe that 
\[
\mathbb{P}(\pi_T^*(\tau_{\Xi^\mathrm{lib}(N)}) \in \,\cdot\,) 
= 
\mathbb{P}(\mathrm{tr}_{\mathbf{U}_N(T)}^{\Xi(N)} \in \,\cdot\,) 
= \nu_{N,T}^{\otimes n}(\{ \mathbf{U} \in \mathrm{U}(N)^n \mid \mathrm{tr}_{\mathbf{U}}^{\Xi(N)} \in \,\cdot\,\})   
\]
as well as 
\begin{equation}\label{Eq10}
\mathbb{P}(\mathrm{tr}_{\mathbf{U}_N}^{\Xi(N)}\in\,\cdot\,) = \nu_N^{\otimes n}(\{ \mathbf{U} \in \mathrm{U}(N)^n \mid \mathrm{tr}_{\mathbf{U}}^{\Xi(N)} \in \,\cdot\,\}). 
\end{equation}
Since 
\[
\Big(\min_{U\in\mathrm{U}(N)}p_{N,T}(U)\Big)\nu_N \leq 
\nu_{N,T} \leq \Big(\max_{U\in\mathrm{U}(N)}p_{N,T}(U)\Big)\nu_N,
\]
we observe that 
\begin{align*}
&\frac{n}{N^2}\log\min_{U\in\mathrm{U}(N)}p_{N,T}(U) + 
\frac{1}{N^2}\log\mathbb{P}(\mathrm{tr}_{\mathbf{U}_N}^{\Xi(N)}\in\,\cdot\,) \\
&\qquad\leq 
\frac{1}{N^2}\log\mathbb{P}(\pi_T^*(\tau_{\Xi^\mathrm{lib}(N)}) \in\,\cdot\,) \\
&\qquad\qquad\leq 
\frac{n}{N^2}\log\max_{U\in\mathrm{U}(N)}p_{N,T}(U) +
\frac{1}{N^2}\log\mathbb{P}(\mathrm{tr}_{\mathbf{U}_N}^{\Xi(N)}\in\,\cdot\,).
\end{align*}
Now, we will use the functions $L(T), U(T)$ in $T$ introduced in Lemma \ref{L2.1}. If we assume the large deviation upper (resp.\ lower) bound for $\mathbb{P}(\pi_T^*(\tau_{\Xi^\mathrm{lib}(N)}) \in \,\cdot\,)$, then   
\begin{align*}
&nL(T) + \varlimsup_{N\to\infty} \frac{1}{N^2}\log\mathbb{P}(\mathrm{tr}_{\mathbf{U}_N}^{\Xi(N)}\in\Lambda) \\
&\leq \varlimsup_{N\to\infty} \frac{1}{N^2}\log\mathbb{P}(\pi_T^*(\tau_{\Xi^\mathrm{lib}(N)}) \in\Lambda) \leq 
-\inf\{ I_T^+(\sigma) \mid \sigma \in \Lambda\} 
\end{align*}
for any closed $\Lambda \subset TS(C^*_R\langle x_{\bullet\diamond}\rangle)$ (resp.\ 
\begin{align*}
&nU(T) + \varliminf_{N\to\infty} \frac{1}{N^2}\log\mathbb{P}(\mathrm{tr}_{\mathbf{U}_N}^{\Xi(N)}\in\Gamma) \\
&\geq \varliminf_{N\to\infty} \frac{1}{N^2}\log\mathbb{P}(\pi_T^*(\tau_{\Xi^\mathrm{lib}(N)})\in\Gamma) \geq 
-\inf\{ I_T^-(\sigma) \mid \sigma \in \Gamma\} 
\end{align*}
for any open $\Gamma \subset TS(C^*_R\langle x_{\bullet\diamond}\rangle)$). It follows by Lemma \ref{L2.1} that 
\begin{align*}
\lim_{\substack{m\to\infty \\ \delta \searrow 0}}\varlimsup_{N\to\infty}\log\frac{1}{N^2}\mathbb{P}(\mathrm{tr}_{\mathbf{U}_N}^{\Xi(N)}\in O_{m,\delta}(\sigma)) 
&\leq 
-\lim_{\substack{m\to\infty \\ \delta \searrow 0}}\varlimsup_{T\to\infty}\inf\{ I_T^+(\sigma') \mid \sigma' \in F_{m,\delta}(\sigma)\} \\
\Big(\text{resp.\ }\lim_{\substack{m\to\infty \\ \delta\searrow0}}\varliminf_{N\to\infty}\log\frac{1}{N^2}\mathbb{P}(\mathrm{tr}_{\mathbf{U}_N}^{\Xi(N)}\in O_{m,\delta}(\sigma)) 
&\geq 
-\lim_{\substack{m\to\infty \\ \delta\searrow0}}\varliminf_{T\to\infty}\inf\{ I_T^-(\sigma') \mid \sigma' \in O_{m,\delta}(\sigma)\}\,\Big)
\end{align*}
for every $\sigma \in TS(C^*_R\langle x_{\bullet\diamond}\rangle)$. Observe that 
\[
\inf\{I_T^\pm(\sigma') \mid \sigma' \in \Theta\} = \inf\{I^\pm(\tau) \mid \tau \in TS^c(C^*_R\langle x_{\bullet\diamond}(\,\cdot\,)\rangle), \pi_T^*(\tau)\in \Theta \}
\]
for any $\Theta \subset TS(C^*_R\langle x_{\bullet\diamond}\rangle)$. By Lemma \ref{L4.1}, 
\[
\lim_{\substack{m\to\infty \\ \delta\searrow0}}\varlimsup_{T\to\infty}\inf\{ I_T^+(\sigma') \mid \sigma' \in O_{m,\delta}(\sigma)\} = 
\lim_{\substack{m\to\infty \\ \delta \searrow 0}}\varlimsup_{T\to\infty}\inf\{ I_T^+(\sigma') \mid \sigma' \in F_{m,\delta}(\sigma)\}
\]
(resp.\ the same identity with replacing $\varlimsup_{T\to\infty}$ and $I_T^+$ with $\varliminf_{T\to\infty}$ and $I_T^-$, respectively) holds and defines a rate function. Since $TS(C_R^*\langle x_{\bullet\diamond}\rangle)$ is weak$^*$ compact, we finally conclude by \cite[Theorem 4.1.11, Lemma 1.2.18]{DemboZeitouni:Book} that $\mathbb{P}(\mathrm{tr}_{\mathbf{U}_N}^{\Xi(N)}\in\,\cdot\,)$ satisfies the large deviation upper (resp.\ lower) bound with speed $N^2$ and the rate function $J^+$ (resp.\ $J^-$). 

For the last assertion, we first point out that
\begin{equation}\label{Eq11}
\begin{aligned}
-J^-(\sigma) 
&\leq 
\lim_{\substack{m\to\infty \\ \delta\searrow0}}\varliminf_{N\to\infty} \frac{1}{N^2}\log\mathbb{P}(\mathrm{tr}_{\mathbf{U}_N}^{\Xi(N)}\in O_{m,\delta}(\sigma)) \\
&\leq 
\lim_{\substack{m\to\infty \\ \delta\searrow0}}\varlimsup_{N\to\infty} \frac{1}{N^2}\log\mathbb{P}(\mathrm{tr}_{\mathbf{U}_N}^{\Xi(N)}\in O_{m,\delta}(\sigma)) 
\leq -J^+(\sigma). 
\end{aligned}
\end{equation}
Since $I^+ = I^-$, we have $-J^-(\sigma) \geq -J^+(\sigma)$ for every $\sigma \in TS(C^*_R\langle x_{\bullet\diamond}\rangle)$. Therefore, we conclude that equality holds in \eqref{Eq11}. This together with \eqref{Eq10} immediately implies the last assertion. 
\end{proof}

It is plausible that the definition of the orbital free entropy $\chi_\mathrm{orb}(\mathbf{X}_1,\dots,\mathbf{X}_{n+1})$ can still be defined independently of the choice of approximating sequence $\Xi = (\Xi(N))_{N\in\mathbb{N}}$ (under the constraint that $\mathrm{tr}^{\Xi(N)}$ converges to the joint distribution of the $\mathbf{X}_i$) without assuming the hyperfiniteness of each random multi-variable $\mathbf{X}_i$. 

As mentioned before, we have already established that the sequence of probability measures $\mathbb{P}(\tau_{\Xi^\mathrm{lib}(N)} \in \,\cdot\,)$ satisfies the large deviation upper bound with speed $N^2$ and the rate function $I_{\sigma_0}^\mathrm{lib}$. Hence, we can prove the next corollary. 

\begin{corollary}\label{C4.3} The sequence of probability measures $\mathbb{P}(\mathrm{tr}_{\mathbf{U}_N}^{\Xi(N)}\in\,\cdot\,)$ satisfies the large deviation upper bound with speed $N^2$ and the rate function 
\begin{equation*}
J_{\sigma_0}^\mathrm{lib}(\sigma) := \lim_{\substack{m\to\infty \\ \delta\searrow0}}\varlimsup_{T\to\infty}\inf\{I^\mathrm{lib}_{\sigma_0}(\tau) \mid \tau \in TS^c(C_R^*\langle x_{\bullet\diamond}(\,\cdot\,)\rangle), \pi_T^*(\tau) \in O_{m,\delta}(\sigma)\},
\end{equation*}
where the infimum over the empty set is taken to be $+\infty$. 
Moreover, $\chi_\mathrm{orb}(\sigma\,|\,\sigma_0) \leq -J_{\sigma_0}^\mathrm{lib}(\sigma)$ holds for every $\sigma \in TS(C^*_R\langle x_{\bullet\diamond}\rangle)$.
\end{corollary}
\begin{proof}
The first assertion immediately follows from Lemma \ref{L4.1} and Proposition \ref{P4.2}. 

For the second assertion, we first observe that 
\begin{align*} 
\chi_\mathrm{orb}(\sigma\,|\,\Xi) 
=
\lim_{\substack{m\to\infty \\ \delta\searrow0}}\varlimsup_{N\to\infty}\frac{1}{N^2}\log\mathbb{P}(\mathrm{tr}_{\mathbf{U}_N}^{\Xi(N)} \in O_{m,\delta}(\sigma)) \leq -J_{\sigma_0}^\mathrm{lib}(\sigma)
\end{align*}
for every $\sigma \in TS(C^*_R\langle x_{\bullet\diamond}\rangle)$. 
Since $J_{\sigma_0}^\mathrm{lib}$ is independent of the choice of approximation $\Xi$ to $\sigma_0$, we conclude that $\chi_\mathrm{orb}(\sigma\,|\,\sigma_0) \leq -J_{\sigma_0}^\mathrm{lib}(\sigma)$ for every $\sigma \in TS(C^*_R\langle x_{\bullet\diamond}\rangle)$.
\end{proof}

\begin{remark}\label{R4.4} {\rm 
Several questions on the matrix liberation process $\Xi^\mathrm{lib}(N)$ toward the completion of developing the theory of orbital free entropy are in order. 
\begin{itemize} 
\item[(Q1)] Show that $J_{\sigma_0}^\mathrm{lib}(\sigma)=0$ implies that the $\mathbf{X}^{\sigma}_i$ are freely independent. (This is a question about minimizers of $J_{\sigma_0}^\mathrm{lib}$.)
\item[(Q2)] Identify $J_{\sigma_0}^\mathrm{lib}(\sigma)$ with Voiculescu's free mutual information $i^*(W^*(\mathbf{X}_1^{\sigma});\dots;W^*(\mathbf{X}_{n+1}^{\sigma}))$ (at least when $\sigma=\sigma_0$ or when the $\mathbf{X}_i^{\sigma_0}$ are freely independent) if possible. Here each $W^*(\mathbf{X}_i^{\sigma})$ denotes the von Neumann subalgebra generated by $\mathbf{X}_i^{\sigma} = (X_{ij}^{\sigma})_{j=1}^{r(i)}$.  
\item[(Q3)] Prove a large deviation lower bound with speed $N^2$ for the sequence of probability measures $\mathbb{P}(\tau_{\Xi^\mathrm{lib}(N)} \in \,\cdot\,)$. It is preferable to identify its rate function with $I_{\sigma_0}^\mathrm{lib}$. 
\end{itemize} 
The affirmative answer to (Q2) shows $\chi_\mathrm{orb} \leq -i^*$. On the other hand, as we saw in Proposition \ref{P4.2}, the affirmative complete answer to (Q3) enables one to define $\chi_\mathrm{orb}$ independently of the choice of approximating sequence at least when $\sigma=\sigma_0$ or when $\sigma_0$ is the `empirical distribution' of a freely independent family as in (Q2). Also, the affirmative complete answers to both (Q2) and (Q3) show $\chi_\mathrm{orb} = -i^*$. Finally, the affirmative answer to (Q2) or (Q3) solves (Q1) in the affirmative; hence (Q1) is a test for both (Q2) and (Q3).} 
\end{remark}

\section{Minimizer of the Rate function $J_{\sigma_0}^\mathrm{lib}$}

In this section, we will solve (Q1) of Remark \ref{R4.4} in the affirmative. 

\medskip
The next lemma is probably known to specialists, but we include its proof for the sake of the completeness of this paper. 

\begin{lemma}\label{L5.1} The limit $\sigma_0^\mathrm{fr} := \lim_{T\to\infty}\pi_T^*(\sigma_0^\mathrm{lib})$ exists in $TS(C^*_R\langle x_{\bullet\diamond}\rangle)$, and we have
\begin{itemize}
\item[(i)] $\sigma_0^\mathrm{fr}$ agrees with $\sigma_0$ on each $C^*_R \langle x_{i\diamond}\rangle $, $i=1,\dots,n+1$;
\item[(ii)] the $\mathbf{X}_i^{\sigma_0^\mathrm{fr}}$, $1 \leq i \leq n+1$, are freely independent.
\end{itemize}
\end{lemma}
\begin{proof}
By construction it is clear that $\pi_T^*(\sigma_0^\mathrm{lib})$ agrees with $\sigma_0$ on $C^*_R\langle x_{i\diamond}\rangle$ for each $1 \leq i \leq n+1$. Hence (i) trivially holds. Thus it suffices to prove only (ii). 

Let $(\mathcal{M},\tau)$ be a tracial $W^*$-probability space and $\mathcal{N} \subset \mathcal{M}$ be a $W^*$-subalgebra. Let $\{v_i(t)\}_{i=1}^n$ be a $*$-freely independent family of free left unitary Brownian motions in $\mathcal{M}$ such that the family is $*$-freely independent of $\mathcal{N}$. Set $v_{n+1}(t) := 1$ for all $t \geq 0$ for the ease of notations. In order to prove (ii), it suffices to prove that  
\begin{align*}
|\tau(v_{i_1}(T) x_1^\circ v_{i_1}(T)^* v_{i_2}(T) x_2^\circ v_{i_2}(T)^* \cdots v_{i_m}(T) x_m^\circ v_{i_m}(T)^*)|& \\
\leq (2^{m-1}-1)& \Big(\sup_{1\leq j\leq m}\Vert x_j^\circ\Vert_\infty\Big)^m e^{-T/2} 
\end{align*}
whenever $m \geq 1$, $i_k \neq i_{k+1}$ ($1 \leq k \leq m-1$) and $x_k^\circ \in \mathcal{N}$ with $\tau(x_k^\circ) = 0$ ($1 \leq k \leq m$). When $m=1$, the left-hand side must be $0$; thus the desired fact trivially holds. Thus we may assume $m \geq 2$. 

Recall that $\tau(v_i(t)) = e^{-t/2}$ for every $t \geq 0$ and $1 \leq i \leq n$. This is a particular case of Biane's result \cite[Lemma 1]{Biane:Fields97}. Since $v_{i_k}(T)$ and $v_{i_{k+1}}(T)$ are $*$-freely independent, we have  
\begin{equation}\label{Eq5.1}
0 \leq \tau(v_{i_k}(T)^* v_{i_{k+1}}(T)) = \overline{\tau(v_{i_k}(T))}\tau(v_{i_{k+1}}(T)) 
= \begin{cases} 
e^{-T/2} &(\text{$i_k$ or $i_{k+1}$ is $n+1$}), \\    
e^{-T} \leq e^{-T/2} & (\text{otherwise})
\end{cases}
\end{equation}
for every $1 \leq k \leq m-1$. Hence we obtain that 
\begin{align*}
&|\tau(v_{i_1}(T) x_1^\circ v_{i_1}(T)^* v_{i_2}(T) x_2^\circ v_{i_2}(T)^* \cdots v_{i_m}(T) x_m^\circ v_{i_m}(T)^*)| \\
&\leq 
\tau(v_{i_1}(T)^* v_{i_2}(T)) |\tau(v_{i_1}(T) x_1^\circ x_2^\circ v_{i_2}(T)^* \cdots v_{i_m}(T) x_m^\circ v_{i_m}(T)^*)| \\ 
&\quad+ 
|\tau(v_{i_1}(T) x_1^\circ (v_{i_1}(T)^* v_{i_2}(T))^\circ x_2^\circ v_{i_2}(T)^* \cdots v_{i_m}(T) x_m^\circ v_{i_m}(T)^*)| \\  
&\leq 
\Big(\sup_{1\leq j\leq m}\Vert x_j^\circ\Vert_\infty\Big)^m e^{-T/2} +  |\tau(v_{i_1}(T) x_1^\circ (v_{i_1}(T)^* v_{i_2}(T))^\circ x_2^\circ v_{i_2}(T)^* \cdots v_{i_m}(T) x_m^\circ v_{i_m}(T)^*)|
\end{align*}
with $(v_{i_1}(T)^* v_{i_2}(T))^\circ := v_{i_1}(T)^* v_{i_2}(T) - \tau(v_{i_1}(T)^* v_{i_2}(T))1$.  We continue this procedure for $v_{i_2}(T)^* v_{i_3}(T)$ and so on until $v_{i_{m-1}}(T)^* v_{i_m}(T)$ inductively, and obtain 
\begin{align*} 
&|\tau(v_{i_1}(T) x_1^\circ v_{i_1}(T)^* v_{i_2}(T) x_2^\circ v_{i_2}(T)^* \cdots v_{i_m}(T) x_m^\circ v_{i_m}(T)^*)| \\
&\leq (1+2+\cdots+2^{m-2}) \Big(\sup_{1\leq j\leq m}\Vert x_j^\circ\Vert_\infty\Big)^m e^{-T/2} \\
&\quad+  |\tau(v_{i_1}(T) x_1^\circ (v_{i_1}(T)^* v_{i_2}(T))^\circ x_2^\circ (v_{i_2}(T)^* v_{i_3}(T))^\circ \cdots (v_{i_{m-1}}(T)v_{i_m}(T))^\circ x_m^\circ v_{i_m}(T)^*)|, 
\end{align*}
where we used $\Vert (v_{i_1}(T)^* v_{i_2}(T))^\circ \Vert_\infty \leq 2$. 
By the $*$-free independence between $\mathcal{N}$ and $\{v_i(t)\}_{i=1}^n$, 
\[
\tau(v_{i_1}(T) x_1^\circ (v_{i_1}(T)^* v_{i_2}(T))^\circ x_2^\circ (v_{i_2}(T)^* v_{i_3}(T))^\circ \cdots (v_{i_{m-1}}(T)^* v_{i_m}(T))^\circ x_m^\circ v_{i_m}(T)^*) = 0,
\]
implying the desired estimate. 
\end{proof}

\begin{lemma}\label{L5.2} For any $\tau \in TS^c(C^*_R\langle x_{\bullet\diamond}(\,\cdot\,)\rangle)$ with $I_{\sigma_0,\infty}^\mathrm{lib}(\tau) < +\infty$ and any $P \in \mathbb{C}\langle x_{\bullet\diamond} \rangle$ we have 
$$
\Vert E_{\mathcal{N}(\tau)}(\pi_{\tilde{\tau}}(\Pi^s(\mathfrak{D}_s^{(k)}\pi_T(P)))) \Vert_\infty \leq C\,\mathbf{1}_{[0,T]}(s)\,e^{(s-T)/2}
$$
for some constant $C=C(P) > 0$ depending only on $P$. 
\end{lemma}
\begin{proof} Iteratively performing the decomposition $Q = \sigma_0(Q)1+Q^\circ$ with $Q^\circ = Q-\sigma_0(Q)1$ we observe that $P$ is a sum of a scalar and several monomials of the form: 
\[
Q_1^\circ \cdots Q_m^\circ, 
\]
where $Q_\ell^\circ \in \mathbb{C}\langle x_{i_\ell \diamond}\rangle$ with $\sigma_0(Q_\ell^\circ) = 0$ such that $m \geq 1$ and $i_\ell \neq i_{\ell+1}$ ($1 \leq \ell \leq m-1$). Hence we may and do assume that $P = Q_1^\circ \cdots Q_m^\circ$ in what follows, since any scalar term vanishes under $\mathfrak{D}^{(k)}_s$. We also observe that each $\delta_s^{(k)} \pi_T(Q_\ell^\circ)$, $1 \leq \ell \leq m$, becomes 
\[
\begin{cases} 
\pi_T(Q_\ell^\circ) v_k(T-s) \otimes v_k(T-s)^* -  v_k(T-s) \otimes v_k(T-s)^* \pi_T(Q_\ell^\circ) & (k = i_\ell, s \leq T), \\
0 & (\text{otherwise}). 
\end{cases}
\]
Hence we may and do restrict our consideration to the case $s \leq T$, and obtain that 
\begin{equation}\label{Eq5.2}
Z^{(k)}(s) := E_{\mathcal{N}(\tau)}(\pi_{\tilde{\tau}}(\Pi^s(\mathfrak{D}_s^{(k)}\pi_T(P)))) 
= \sum_{\ell=1}^m [Z^{(k)}_\ell(s), (Q_\ell^\circ)_s], 
\end{equation}
where $Z^{(k)}_\ell(s)$ is defined to be $0$ when $i_\ell \neq k$; otherwise to be
\[
\begin{cases} 
E_{\mathcal{N}(\tau)}(w_{i_\ell,i_{\ell+1}} (Q_{\ell+1}^\circ)_s \cdots w_{i_{m-1},i_m} (Q_m^\circ)_s w_{i_m,i_1} (Q_1^\circ)_s w_{i_1,i_2} \cdots (Q_{\ell-1}^\circ)_s w_{i_{\ell-1},i_\ell}) & (i_m \neq i_1), \\
E_{\mathcal{N}(\tau)}(w_{i_\ell,i_{\ell+1}} (Q_{\ell+1}^\circ)_s \cdots w_{i_{m-1},i_1} (Q_m^\circ Q_1^\circ)_s w_{i_1,i_2} \cdots (Q_{\ell-1}^\circ)_s w_{i_{\ell-1},i_\ell}) & (i_m = i_1) 
\end{cases}
\]
and we write $w_{i,i'} := v_i^\tau(T-s)^* v_{i'}^\tau(T-s)$ ($1 \leq i \neq i' \leq n+1$). ({\it n.b.}, $v^\tau_{n+1}(t) := 1$ for all $t \geq 0$) and $(Q)_s := \pi_{\tilde{\tau}}(\pi_s(Q))$ for $Q \in \mathbb{C}\langle x_{\bullet\diamond} \rangle$. By \cite[Proposition 5.7(1),(2)]{Ueda:JOTP19}, which still holds for $I_{\sigma_0,\infty}^\mathrm{lib}$ without any essential changes, $I_{\sigma_0,\infty}^\mathrm{lib}(\tau) < +\infty$ guarantees that $\tilde{\tau}((Q)_s) = \sigma_0(Q)$ for all $Q \in \mathbb{C}\langle x_{i\diamond}\rangle$ with each fixed $i = 1,\dots,n+1$. Hence the first case $i_m \neq i_1$ can be treated essentially in the same way as in the proof of Lemma \ref{L5.1}. Namely, when $i_m \neq i_1$ (and $i_\ell = k$), we have, for any $y \in \mathcal{N}(\tau)$ (see subsection 4.4 for this notation), 
\begin{align*} 
&\tau(y Z_\ell^{(k)}(s)) \\
&= 
\tilde{\tau}(y w_{i_\ell,i_{\ell+1}} (Q_{\ell+1}^\circ)_s \cdots w_{i_{m-1},i_m} (Q_m^\circ)_s w_{i_m,i_1} (Q_1^\circ)_s w_{i_1,i_2} \cdots (Q_{\ell-1}^\circ)_s w_{i_{\ell-1},i_\ell}) \\
&= 
\tilde{\tau}(w_{i_\ell,i_{\ell+1}})\tilde{\tau}(y (Q_{\ell+1}^\circ)_s \cdots w_{i_{m-1},i_m} (Q_m^\circ)_s w_{i_m,i_1} \pi_s(Q_1^\circ) w_{i_1,i_2} \cdots (Q_{\ell-1}^\circ)_s w_{i_{\ell-1},i_\ell})\\
&\quad+
\tilde{\tau}(y(w_{i_\ell,i_{\ell+1}})^\circ (Q_{\ell+1}^\circ)_s \cdots w_{i_{m-1},i_m} (Q_m^\circ)_s w_{i_m,i_1} (Q_1^\circ)_s w_{i_1,i_2} \cdots (Q_{\ell-1}^\circ)_s w_{i_{\ell-1},i_\ell}), 
\end{align*}
and obtain that 
\begin{align*}
&Z_\ell^{(k)}(s) = \\
&\quad\tilde{\tau}(w_{i_\ell,i_{\ell+1}}) E_{\mathcal{N}(\tau)}((Q_{\ell+1}^\circ)_s \cdots w_{i_{m-1},i_m}(Q_m^\circ)_s w_{i_m,i_1} (Q_1^\circ)_s w_{i_1,i_2} \cdots (Q_{\ell-1}^\circ)_s w_{i_{\ell-1},i_\ell}) \\
&\qquad+ 
E_{\mathcal{N}(\tau)}((w_{i_\ell,i_{\ell+1}})^\circ (Q_{\ell+1}^\circ)_s \cdots w_{i_{m-1},i_m} (Q_m^\circ)_s w_{i_m,i_1} (Q_1^\circ)_s w_{i_1,i_2} \cdots (Q_{\ell-1}^\circ)_s w_{i_{\ell-1},i_\ell})
\end{align*}
with $(w_{i,i'})^\circ := w_{i,i'} - \tilde{\tau}(w_{i,i'})1$. Making the same computation for the second term and iterating this procedure until $w_{i_{\ell-2},i_{\ell-1}}$, we finally arrive at the following formula: $Z_\ell^{(k)}(s)$ is the sum of $\tilde{\tau}(w_{i_j,i_{j+1}})$ times 
\[
E_{\mathcal{N}(\tau)}((w_{i_\ell,i_{\ell+1}})^\circ (Q_{\ell+1}^\circ)_s \cdots (w_{i_{j-1},i_j})^\circ (Q_j^\circ)_s \overbrace{w_{i_j,i_{j+1}}}^\text{delete} (Q_{j+1}^\circ)_s w_{i_{j+1},i_{j+2}} \cdots (Q_{\ell-1}^\circ)_s w_{i_{\ell-1},i_\ell})
\]
over all $j=l,\dots,m,1,\dots,\ell-2$ (where we read $m+1$ as $1$). Therefore, we have obtained that 
\[
\Vert Z_\ell^{(k)}(s) \Vert_\infty \leq (2^{m-1}-1) \Big(\sup_{1\leq j \leq m} \Vert Q_j^\circ\Vert_\infty\Big)^{m-1} e^{(s-T)/2} 
\]
since $\Vert (w_{i,i'})^\circ\Vert_\infty \leq 2$ and $0 \leq \tilde{\tau}(w_{i,i'}) = \tilde{\tau}(v_i^\tau(T-s)^* v_{i'}^\tau(T-s)) \leq e^{(s-T)/2}$ with $i\neq i'$ (see \eqref{Eq5.1} for a similar computation). Hence we get 
\begin{equation}\label{Eq5.3} 
\Vert [Z^{(k)}_\ell(s), (Q_\ell^\circ)_s]\Vert_\infty 
\leq 
C_1\,e^{(s-T)/2}  
\end{equation}
with a positive constant $C_1$ depending only on $P$ and $\ell$. 

We then consider the case $i_m = i_1$ (and $s \leq T$). This case is a bit complicated, but can still be treated similarly as above. In fact, if $i_{m-1} \neq i_2$, then 
\begin{align*}
Z_\ell^{(k)}(s) &= \tilde{\tau}((Q_m^\circ Q_1^\circ)_s) E_{\mathcal{N}(\tau)}(w_{i_\ell,i_{\ell+1}} (Q_{l+1}^\circ)_s \cdots w_{i_{m-1},i_2} \cdots (Q_{\ell-1}^\circ)_s w_{i_{\ell-1},i_\ell}) \\
&\quad+ 
E_{\mathcal{N}(\tau)}(w_{i_\ell,i_{\ell+1}} (Q_{\ell+1}^\circ)_s \cdots w_{i_{m-1},i_1} ((Q_m^\circ Q_1^\circ)^\circ)_s w_{i_1,i_2} \cdots (Q_{\ell-1}^\circ)_s w_{i_{\ell-1},i_\ell})
\end{align*}
since $w_{i_{m-1},i_1} w_{i_1,i_2} = w_{i_{m-1},i_2}$. Thus, we apply the previous procedure to the first and the second terms, respectively, and conclude 
\[
\Vert Z_\ell^{(k)}(s) \Vert_\infty \leq \big\{ (2^{m-3}-1) + (2^{m-2}-1)\}  \Big(\sup_{1\leq j \leq m} \Vert Q_j^\circ\Vert_\infty\Big)^{m-1} e^{(s-T)/2}. 
\]
Iterating this procedure in the cases e.g.\ $i_m = i_1$, $i_{m-1} = i_2$ and $i_{m-2} \neq i_3$, we can estimate $\Vert Z_\ell^{(k)}(s) \Vert_\infty$ by $e^{(s-T)/2}$ times a positive constant only depending on $P$ except the case when $i_m=i_1, i_{m-1} = i_2,\dots, i_{\ell+1} = i_{\ell-1}$ (i.e, $m$ is odd and $\ell= (m+1)/2$). In the remaining case, we can easily observe that 
\[
Z_\ell^{(k)}(s) = \sigma_0(Q_m^\circ Q_1^\circ)\sigma_0(Q_{m-1}^\circ Q_2^\circ) \cdots \sigma_0(Q_{\ell+1}^\circ Q_{\ell-1}^\circ)1 + Z_\ell^{(k)}(s)^\sim
\]
with an element $Z_\ell^{(k)}(s)^\sim \in \mathcal{N}(\tau)$ whose operator norm $\Vert Z_\ell^{(k)}(s)^\sim\Vert_\infty$ is not greater than $e^{(s-T)/2}$ times a positive constant only depending on $P$. Then we have 
\begin{equation}\label{Eq5.4}
\Vert [Z^{(k)}_\ell(s), (Q_\ell^\circ)_s] \Vert_\infty = 
\Vert [Z^{(k)}_\ell(s)^\sim, (Q_\ell^\circ)_s] \Vert_\infty 
\leq 
2 \Vert Z_\ell^{(k)}(s)^\sim \Vert_\infty \Vert Q_\ell^\circ\Vert_\infty 
\leq C_2\,e^{(s-T)/2}. 
\end{equation}
with a positive constant $C_2$ depending only on $P$ and $\ell$. 

Consequently, the expansion \eqref{Eq5.2} of $Z^{(k)}(s)$ together with the above norm estimates \eqref{Eq5.3}, \eqref{Eq5.4} shows the desired norm estimate. 
\end{proof}

A more explicit description on $E_{\mathcal{N}(\tau)}(\pi_{\tilde{\tau}}(\Pi^s(\mathfrak{D}_s^{(k)}P)))$ is possible based on the combinatorial techniques introduced by Speicher (see e.g.\ Nica--Speicher \cite{NicaSpeicher:Book} as a standard textbook). See section 8. 

\medskip
With the above lemmas we will prove that the rate function $J_{\sigma_0}^\mathrm{lib}$ admits a unique minimizer, and moreover, we will explicitly compute the minimizer. Moreover, we will also prove that the modification $J_{\sigma_0,\infty}^\mathrm{lib}$ of $J_{\sigma_0}^\mathrm{lib}$ by replacing $I_{\sigma_0}^\mathrm{lib}$ with $I_{\sigma_0,\infty}^\mathrm{lib}$, i.e.,
\[
J_{\sigma_0,\infty}^\mathrm{lib}(\sigma) := \lim_{\substack{m\to\infty \\ \delta\searrow0}}\varlimsup_{T\to\infty}\inf\{I^\mathrm{lib}_{\sigma_0,\infty}(\tau) \mid \tau \in TS^c(C_R^*\langle x_{\bullet\diamond}(\,\cdot\,)\rangle), \pi_T^*(\tau) \in O_{m,\delta}(\sigma)\}
\] 
admits the same unique minimizer. 

\begin{theorem}\label{T5.3} For any $\sigma \in TS(C^*_R\langle x_{\bullet\diamond}\rangle)$ the following are equivalent:
\begin{itemize}
\item[(1)] $\sigma=\sigma_0^\mathrm{fr}$. 
\item[(2)] $J_{\sigma_0}^\mathrm{lib}(\sigma) = 0$. 
\item[(3)] $J_{\sigma_0,\infty}^\mathrm{lib}(\sigma) = 0$.
\end{itemize}
\end{theorem}
\begin{proof} 
(1) $\Rightarrow$ (2): Since $I_{\sigma_0}^\mathrm{lib}(\sigma_0^\mathrm{lib}) = 0$ and moreover since $\pi_T^*(\sigma_0^\mathrm{lib}) \rightarrow \sigma_0^\mathrm{fr}$ as $T \to +\infty$ by Lemma \ref{L5.1}, we have $J_{\sigma_0}^\mathrm{lib}(\sigma_0^\mathrm{fr}) = 0$. 

\medskip
(2) $\Rightarrow$ (3): Trivial because $0 \leq J_{\sigma_0,\infty}^\mathrm{lib} \leq J_{\sigma_0}^\mathrm{lib}$, which follows from $0 \leq I_{\sigma_0,\infty}^\mathrm{lib} \leq I_{\sigma_0}^\mathrm{lib}$. 

\medskip
(3) $\Rightarrow$ (1): $J_{\sigma_0,\infty}^\mathrm{lib}(\sigma) = 0$ implies that for every $m \in \mathbb{N}$ and $\delta>0$ we have 
$$
\lim_{T\to\infty} \inf\{ I_{\sigma_0,\infty}^\mathrm{lib}(\tau)\mid 
\tau \in TS^c(C^*\langle x_{\bullet\diamond}(\,\cdot\,)\rangle), \pi_T^*(\tau) \in O_{m,\delta}(\sigma) \big\} = 0. 
$$
Thus we can choose a sequence $0 < T_1 < T_2 < \cdots < T_m \nearrow +\infty$ as $m \nearrow \infty$ and $\tau_{T_m} \in  TS^c(C^*\langle x_{\bullet\diamond}(\,\cdot\,)\rangle)$ for each $m \in \mathbb{N}$ such that $\pi_{T_m}^*(\tau_{T_m}) \in O_{m,1/m}(\sigma)$ and $I_{\sigma_0,\infty}^\mathrm{lib}(\tau_{T_m}) < 1/m$ for every $m \in \mathbb{N}$. 
For each $P = P^* \in \mathbb{C}\langle x_{\bullet\diamond} \rangle$ we have 
\begin{align*} 
|\sigma(P) - \sigma_0^\mathrm{fr}(P)| 
&\leq 
|\sigma(P) - \pi_{T_m}^*(\tau_{T_m})(P)| \\
&\qquad+ 
|\tau_{T_m}(\pi_{T_m}(P)) - \sigma_0^\mathrm{lib}(\pi_{T_m}(P))| \\
&\qquad\qquad+ 
|\pi_{T_m}^*(\sigma_0^\mathrm{lib})(P) - \sigma_0^\mathrm{fr}(P)| \\
&\leq 
|\sigma(P) - \pi_{T_m}^*(\tau_{T_m})(P)| + |\pi_{T_m}^*(\sigma_0^\mathrm{lib})(P) - \sigma_0^\mathrm{fr}(P)| \\
&\quad+ 
\sqrt{2 I_{\sigma_0,\infty}^\mathrm{lib}(\tau_{T_m}) \sum_{k=1}^n \int_0^\infty \Vert E_{\mathcal{N}(\tau)}(\pi_{\tilde{\tau}}(\Pi^s(\mathfrak{D}_s^{(k)}(\pi_{T_m}(P)))))\Vert_{\tilde{\tau},2}^2\,ds}  
\end{align*}
by \cite[Lemma 5.3]{Ueda:JOTP19} that still holds true for $I_{\sigma_0,\infty}^\mathrm{lib}$ without any essential changes. Now, we use Lemma \ref{L5.2} to get 
\[
\sum_{k=1}^n \int_0^\infty \Vert E_{\mathcal{N}(\tau)}(\pi_{\tilde{\tau}}(\Pi^s(\mathfrak{D}_s^{(k)}(\pi_{T_m}(P)))))\Vert_{\tilde{\tau},2}^2\,ds 
\leq C\,\int_0^{T_m} e^{s-T_m}\,ds = C(1-e^{-T_m}) \leq C
\]
for all $m$ with a constant $C > 0$ only depending on $P$. Consequently, we obtain that 
\[
|\sigma(P) - \sigma_0^\mathrm{fr}(P)| 
\leq 
|\sigma(P) - \pi_{T_m}^*(\tau_{T_m})(P)| + |\pi_{T_m}^*(\sigma_0^\mathrm{lib})(P) - \sigma_0^\mathrm{fr}(P)| + \sqrt{\frac{2C}{m}}, 
\]
whose right-hand side converges to $0$ as $m\to\infty$ thanks to $\pi_{T_m}^*(\tau_{T_m}) \in O_{m,1/m}(\sigma)$ (that guarantees that $\sigma = \lim_{m\to\infty} \pi_{T_m}^*(\tau_{T_m})$ in $TS(C^*_R\langle x_{\bullet\diamond}\rangle)$) and Lemma \ref{L5.1}. Hence we conclude that $\sigma = \sigma_0^\mathrm{fr}$. 
\end{proof} 

Thanks to the standard Borel-Cantelli argument (see e.g.\ the proof of \cite[Corollary 5.9]{Ueda:JOTP19}) the above proposition together with Corollary \ref{C4.3} implies that \emph{$\mathrm{tr}_{\mathbf{U}_N}^{\Xi(N)}$ converges to $\sigma_0^\mathrm{fr}$ almost surely as $N\to\infty$}. This is nothing less than a consequence of the asymptotic freeness of independent Haar-distributed unitary random matrices. On the other hand, the corresponding result for the matrix liberation process \cite[Corollary 5.9]{Ueda:JOTP19} was not known prior to it. 

\medskip
We would also like to point out that both $J_{\sigma_0}^\mathrm{lib}, J_{\sigma_0,\infty}^\mathrm{lib}$ can be regarded as a kind of mutual information in free probability, since they characterize the free independence as a unique minimizer (see the third paragraph of section 1). Thus it is natural to reformulate the functionals $J_{\sigma_0}^\mathrm{lib}, J_{\sigma_0,\infty}^\mathrm{lib}$ as well as their sources $I_{\sigma_0}^\mathrm{lib}, I_{\sigma_0,\infty}^\mathrm{lib}$ in a coordinate-free fashion. This will be done in the next section. 

\section{A coordinate-free approach: A new kind of free mutual information} 

Let $(\mathcal{M},\tau)$ be a tracial $W^*$-probability space. We consider unital $C^*$-subalgebras $\mathcal{A}_i \subset \mathcal{M}$, $1 \leq i \leq n+1$, and define a kind of free mutual information $i^{**}(\mathcal{A}_1;\dots;\mathcal{A}_n:\mathcal{A}_{n+1})$, without appealing to any kind of (matricial) microstates, whose definition comes from the rate functions discussed so far. 

\subsection{Universal algebras}
Let $\mathfrak{A} := \bigstar_{i=1}^{n+1} \mathcal{A}_i$ be the universal free product $C^*$-algebra. Let $\mathfrak{A}(t)$, $t \geq 0$, be copies of $\mathfrak{A}$, and define $\mathfrak{A}(\mathbb{R}_+)$ to be the universal free product $C^*$-algebra $\bigstar_{t \geq 0} \mathfrak{A}(t)$. (Here we write $\mathbb{R}_+ = [0,+\infty)$.) We denote by $\lambda_i : \mathcal{A}_i \to \mathfrak{A}$ and $\rho_t : \mathfrak{A} \twoheadrightarrow \mathfrak{A}(t) \subset \mathfrak{A}(\mathbb{R}_+)$ the canonical $*$-homomorphisms, which are known to be injective, see the appendix for an explicit reference about this fact. Write $\rho_{t,i} := \rho_t\circ\lambda_i : \mathcal{A}_i \to \mathfrak{A}(\mathbb{R}_+)$. By Lemma \ref{LA1}, $\mathfrak{A}(\mathbb{R}_+)$ with $*$-homomorphisms $\rho_{t,i}$ can naturally be identified with the universal free product of the copies of $\mathcal{A}_i$, $1 \leq i \leq n+1$, over $\mathbb{R}_+$.  

\subsection{Time-dependent liberation derivatives}
Let $\mathfrak{P}$ be the $*$-subalgebra of $\mathfrak{A}$ algebraically generated by $\lambda_i(\mathcal{A}_i)$, $1 \leq i \leq n+1$. Consider the $*$-subalgebra $\mathfrak{P}(\mathbb{R}_+)$ of $\mathfrak{A}(\mathbb{R}_+)$ algebraically generated by $\rho_t(\mathfrak{P})$, $t \geq 0$. Remark that $\lambda_i(\mathcal{A}_i)$, $1 \leq i \leq n+1$, and $\rho_{t,i}(\mathcal{A}_i)$, $1 \leq i \leq n+1$, $t \geq 0$, are algebraically free families of $*$-subalgebras, and the resulting $\mathfrak{P}$ and $\mathfrak{P}(\mathbb{R}_+)$ are naturally identified with the algebraic free products of the $\lambda_i(\mathcal{A}_i)$, $1 \leq i \leq n+1$, and of the $\rho_{t,i}(\mathcal{A}_i)$, $1 \leq i \leq n+1$, $t \geq 0$, respectively. See Proposition \ref{PA4}. 

We extend $\mathfrak{A}(\mathbb{R}_+)$ to $\tilde{\mathfrak{A}}(\mathbb{R}_+)$ by taking its universal free product with the universal $C^*$-algebra generated by $u_i(t)$, $1 \leq i \leq n$, $t \geq 0$ with subject to $u_i(t)^* u_i(t) = u_i(t)u_i(t)^* = 1$ and $u_i(0) = 1$. This procedure is justified by Proposition \ref{PA3}. Consider the derivation $\Delta_s^{(k)} : \mathfrak{P}(\mathbb{R}_+) \to \tilde{\mathfrak{A}}(\mathbb{R}_+)\otimes_\mathrm{alg}\tilde{\mathfrak{A}}(\mathbb{R}_+)$, $1 \leq k \leq n$, sending each $\rho_{t,i}(x)$ with $x \in \mathcal{A}_i$ to     
\[
\delta_{i,k}\mathbf{1}_{[0,t]}(s)\,(\rho_{t,k}(x) u_k(t-s)\otimes u_k(t-s)^* - u_k(t-s)\otimes u_k(t-s)^* \rho_{t,k}(x))
\]
({\it n.b.}, the algebraic freeness among the $\rho_{t,i}(\mathcal{A}_i)$ makes every $\Delta_s^{(k)}$ well-defined). Therefore, with the flip-multiplication map $\theta : \tilde{\mathfrak{A}}(\mathbb{R}_+)\otimes_\mathrm{alg}\tilde{\mathfrak{A}}(\mathbb{R}_+) \to \tilde{\mathfrak{A}}(\mathbb{R}_+)$ sending $a\otimes b$ to $ba$, we obtain the cyclic derivative $\nabla_s^{(k)} := \theta\circ\Delta_s^{(k)} : \mathfrak{P}(\mathbb{R}_+) \to \tilde{\mathfrak{A}}(\mathbb{R}_+)$. 

\subsection{Continuous tracial states}
\emph{Differently from the previous sections we will use symbols $\varphi, \psi$, etc., instead of $\tau$ for tracial states on $\mathfrak{A}(\mathbb{R}_+)$, etc., in order to avoid any confusion of symbols.}  

\medskip
A tracial state $\varphi \in TS(\mathfrak{A}(\mathbb{R}_+))$ is said to be continuous, if $t \mapsto \pi_\varphi(\rho_t(x))$ is strongly continuous for every $x \in \mathfrak{A}$, where $\pi_\varphi : \mathfrak{A}(\mathbb{R}_+) \curvearrowright \mathcal{H}_\tau$ denotes the GNS representation associated with $\tau$. 
In what follows, we denote by $TS^c(\mathfrak{A}(\mathbb{R}_+))$ all the continuous tracial states on $\mathfrak{A}(\mathbb{R}_+)$. 

\begin{lemma}\label{L6.1} For a given $\varphi \in TS(\mathfrak{A}(\mathbb{R}_+))$ the following are equivalent: 
\begin{itemize}
\item[(i)] $\varphi$ is continuous. 
\item[(ii)] For every $m \in \mathbb{N}$ and every $x_1,\dots,x_m \in \mathfrak{A}$ the function 
\[
(t_1,\dots,t_m) \mapsto \varphi(\rho_{t_1}(x_1)\cdots\rho_{t_m}(x_m)) 
\] 
is continuous. 
\item[(iii)] For every $m \in \mathbb{N}$ and every $x_k \in \mathcal{A}_{i_j}$, $1 \leq i_k \leq n+1$, $1 \leq k \leq m$, the function 
\[
(t_1,\dots,t_m) \mapsto \varphi(\rho_{t_1,i_1}(x_1)\cdots\rho_{t_m,i_m}(x_m)) 
\]
is continuous.  
\item[(iv)] For every $1 \leq i \leq n+1$, there exists a $C^*$-generating set $\mathcal{X}_i$ consisting of self-adjoint elements in $\mathcal{A}_i$ such that for every $m \in \mathbb{N}$ and every $x_j \in \mathcal{X}_{i_j}$, $1 \leq i_j \leq n+1$, $1 \leq j \leq m$, the function 
\[
(t_1,\dots,t_m) \mapsto \varphi(\rho_{t_1,i_1}(x_1)\cdots\rho_{t_m,i_m}(x_m)) 
\]
is continuous.  

\end{itemize} 
\end{lemma} 
\begin{proof} Since $\Vert \rho_t(x) \Vert_\infty = \Vert x \Vert_\infty$ for every $x \in \mathfrak{A}$ and since the $\rho_t(\mathfrak{A})$ over $t \geq 0$ generate $\mathfrak{A}(\mathbb{R}_+)$ as a $C^*$-algebra, the proof of \cite[Lemma 2.1]{Ueda:JOTP19} works for showing that item (i) $\Leftrightarrow$ item (ii) without any essential changes. Item (ii) $\Rightarrow$ item (iii) is trivial. The standard approximation argument using the norm density of the unital $*$-algebra algebraically generated by $\lambda_i(\mathcal{A}_i)$ in $\mathfrak{A}$ shows that item (iii) $\Rightarrow$ item (ii). Item (iii) $\Leftrightarrow$ item (iv) is also confirmed similarly by using the norm density of the unital $*$-algebra algebraically generated by $\mathcal{X}_i$ in $\mathcal{A}_i$. 
\end{proof} 

We extend each $\varphi \in TS^c(\mathfrak{A}(\mathbb{R}_+))$ to a unique $\tilde{\varphi} \in TS(\tilde{\mathfrak{A}}(\mathbb{R}_+))$ in such a way that the $u_i(t)$'s are $*$-freely independent of $\mathfrak{A}(\mathbb{R}_+)$ and form a $*$-freely independent family of left-multiplicative free unitary Brownian motions under this extension $\tilde{\varphi}$. It is not difficult to see that $\tilde{\varphi}$ is `continuous', that is, both $t \mapsto \pi_{\tilde{\varphi}}(\rho_t(x))$ with $x \in \mathfrak{A}$ and $t \mapsto \pi_{\tilde{\varphi}}(u_i(t))$ are strongly continuous. Denote by $\pi_{\tilde{\varphi}} : \tilde{\mathfrak{A}}(\mathbb{R}_+) \curvearrowright \mathcal{H}_{\tilde{\varphi}}$ the GNS representation associated with $\tilde{\varphi}$. We have a unique surjective unital $*$-homomorphism $\Lambda^s : \tilde{\mathfrak{A}}(\mathbb{R}_+) \to \tilde{\mathfrak{A}}(\mathbb{R}_+)$ sending each $\rho_{t,i}(x)$ with $x \in \mathcal{A}_i$, $t \geq 0$ to 
\begin{equation}\label{Eq6.1}
\rho^s_{t,i}(x) := 
\begin{cases} 
u_i((t-s)\vee0)\rho_{s\wedge t,i}(x)u_i((t-s)\vee0)^* & (1 \leq i \leq n), \\
\rho_{t,n+1}(x) & (i=n+1)  
\end{cases}
\end{equation}
and keeping each $u_i(t)$ as it is. Note that each $\rho_{t,i}^s$ clearly defines a unital $*$-homomorphism from $\mathcal{A}_i$ to $\tilde{\mathfrak{A}}(\mathbb{R}_+)$ for every $1 \leq i \leq n+1$, and moreover, by universality, those $\rho_{t,i}^s$ give rise to a unital $*$-homomorphism $\rho_t^s : \mathfrak{A} \to \tilde{\mathfrak{A}}(\mathbb{R}_+)$. Observe that $\Lambda^s\circ\rho_t := \rho_t^s$ holds for every $s, t \geq 0$. We define $\varphi^s := \tilde{\varphi}\circ\Lambda^s$ on $\mathfrak{A}(\mathbb{R}_+)$. Since  
\[
\tilde{\varphi}\circ\Lambda^s(\rho_{t_1,i_1}(x_1)\cdots\rho_{t_m,i_m}(x_m)) 
= 
\tilde{\varphi}(\rho_{t_1,i_1}^s(x_1)\cdots\rho_{t_m,i_m}^s(x_m)), 
\]
we observe, by \eqref{Eq6.1}, that $\varphi^s$ is a continuous tracial state. 

By the $*$-homomorphism $\Gamma : \mathfrak{A}(\mathbb{R}_+) \to \mathfrak{A}$ sending each $\rho_{t,i}(x)$ with $x \in \mathcal{A}_i$ to $\lambda_i(x)$ we construct $\Gamma^*(\sigma_0) := \sigma_0\circ\Gamma \in TS^c(\mathfrak{A}(\mathbb{R}_+))$ with a given $\sigma_0 \in TS(\mathfrak{A})$ and set $\sigma_0^\mathrm{lib} := \Gamma^*(\sigma_0)^0 \in TS^c(\mathfrak{A}(\mathbb{R}_+))$.  

\subsection{The new free mutual information}
For a given $\sigma_0 \in TS(\mathfrak{A})$ let us define two functionals $\mathcal{I}_{\sigma_0}^\mathrm{lib}, \mathcal{I}_{\sigma_0,\infty}^\mathrm{lib} : TS^c(\mathfrak{A}(\mathbb{R}_+)) \to [0,+\infty]$ as follows. Let $\varphi \in TS^c(\mathfrak{A}(\mathbb{R}_+))$ be arbitrarily given. Let $E_{\mathcal{Q}(\varphi)}$ denote the $\tilde{\varphi}$-preserving conditional expectation from $\mathcal{P}(\varphi) := \pi_{\tilde{\varphi}}(\tilde{\mathfrak{A}}(\mathbb{R}_+))''$ onto $\mathcal{Q}(\varphi) := \pi_{\tilde{\varphi}}(\mathfrak{A}(\mathbb{R}_+))''$, where the double commutants are taken on $\mathcal{H}_{\tilde{\varphi}}$. For any $P=P^* \in \mathfrak{P}(\mathbb{R}_+)$ and $t \in [0,\infty]$ we define
\[
\mathcal{I}_{\sigma_0,t}^\mathrm{lib}(\varphi,P) 
= \varphi^t(P) - \sigma_0^\mathrm{lib}(P) - \frac{1}{2}\sum_{k=1}^n \int_0^t \Vert E_{\mathcal{Q}(\varphi)}(\pi_{\tilde{\varphi}}(\Lambda^s(\nabla_s^{(k)}P)))\Vert_{\tilde{\varphi},2}^2\,\mathrm{d}s
\]
with regarding $\varphi$ as $\varphi^\infty$ (since $\varphi^t(P) = \varphi(P)$ when $t$ is large enough). We observe that $s \mapsto \Vert E_{\mathcal{Q}(\varphi)}(\pi_{\tilde{\varphi}}(\Lambda^s(\nabla_s^{(k)}P)))\Vert_{\tilde{\varphi},2}^2$ is piecewise continuous in $s$ and becomes zero when $s$ is large enough thanks to $P \in \mathfrak{A}(\mathbb{R}_+)$. These two facts guarantee that $\mathcal{I}_{\sigma_0,t}^\mathrm{lib}(\varphi,P)$ is well defined for every $t$ possibly with $t = \infty$. Then we define
\[
\mathcal{I}_{\sigma_0}^\mathrm{lib}(\varphi) 
= \sup_{\substack{P = P^* \in \mathfrak{P}(\mathbb{R}_+) \\ t \geq 0}} 
\mathcal{I}_{\sigma_0,t}^\mathrm{lib}(\varphi,P), \quad
\mathcal{I}_{\sigma_0,\infty}^\mathrm{lib}(\varphi) 
= \sup_{P = P^* \in \mathfrak{A}(\mathbb{R}_+)} 
\mathcal{I}_{\sigma_0,\infty}^\mathrm{lib}(\varphi,P).
\]
Clearly, $\mathcal{I}_{\sigma_0}^\mathrm{lib}(\varphi) \geq \mathcal{I}_{\sigma_0,\infty}^\mathrm{lib}(\varphi)$ holds, and it is a question again whether equality holds or not. 

We then introduce two functionals $\mathcal{J}_{\sigma_0}^\mathrm{lib}, \mathcal{J}_{\sigma_0,\infty}^\mathrm{lib} : TS(\mathfrak{A}) \to [0,+\infty]$ as before. To this end, we have to endow $TS(\mathfrak{A})$ with the weak$^*$ topology. Let $\sigma \in TS(\mathfrak{A})$ be arbitrarily given. Let $\mathcal{O}(\sigma)$ be the open neighborhoods at $\sigma$ in the weak$^*$ topology on $TS(\mathfrak{A})$. Then we define 
\begin{equation}\label{Eq6.2} 
\mathcal{J}_{\sigma_0}^\mathrm{lib}(\sigma) := \sup_{O \in \mathcal{O}(\sigma)}\varlimsup_{T\to\infty}\inf\{\mathcal{I}^\mathrm{lib}_{\sigma_0}(\varphi) \mid \varphi \in TS^c(\mathfrak{A}(\mathbb{R}_+)), \rho_T^*(\varphi) \in O \} 
\end{equation}
and also $\mathcal{J}_{\sigma_0,\infty}^\mathrm{lib}(\sigma)$ in the same manner as above with replacing $\mathcal{I}^\mathrm{lib}_{\sigma_0}(\varphi)$ with $\mathcal{I}^\mathrm{lib}_{\sigma_0,\infty}(\varphi)$. Here the infimum over the empty set is taken to be $+\infty$ as usual. Remark that the supremum over $O \in \mathcal{O}(\sigma)$ coincides with the limit over a neighborhood basis at $\sigma$. We also remark that $\mathcal{O}(\sigma)$ can be replaced with the smaller neighborhood basis consisting of  
\[
O_{\mathcal{W},\delta}(\sigma) := 
\{ 
\sigma' \in TS(\mathfrak{A}) \mid \text{
$|\sigma'(W) - \sigma(W)| < \delta$ for all $W \in \mathcal{W}$} \}
\]
all over the finite collections $\mathcal{W}$ of words $W$ like $\lambda_{i_1}(a_1)\cdots\lambda_{i_m}(a_m)$ with $a_{i_k} \in \mathcal{A}_{i_k}$ and $\delta>0$, since all the linear combinations of words form a norm dense $*$-subalgebra of $\mathfrak{A}$.   

\begin{definition}\label{D6.1} Thanks to the universality of $\mathfrak{A}$, we have a unique $*$-homomorphism  $\Upsilon : \mathfrak{A} \to \mathcal{M}$ sending each $\lambda_i(x)$ to $x$ with $x \in \mathcal{A}_i \subset \mathcal{M}$, $1 \leq i \leq n+1$. Then we define 
\[
\mathcal{J}_{\sigma_0}^\mathrm{lib}(\mathcal{A}_1;\dots;\mathcal{A}_n:\mathcal{A}_{n+1}) := 
\mathcal{J}_{\sigma_0}^\mathrm{lib}(\Upsilon^*(\tau)) \geq 
\mathcal{J}_{\sigma_0,\infty}^\mathrm{lib}(\Upsilon^*(\tau)) =:
\mathcal{J}_{\sigma_0,\infty}^\mathrm{lib}(\mathcal{A}_1;\dots;\mathcal{A}_n:\mathcal{A}_{n+1}).  
\]
Moreover, we write
\[
i^{**}(\mathcal{A}_1;\dots;\mathcal{A}_n:\mathcal{A}_{n+1}) 
:= 
\mathcal{J}_{\Upsilon^*(\tau),\infty}^\mathrm{lib}(\mathcal{A}_1;\dots;\mathcal{A}_n:\mathcal{A}_{n+1}).
\]
\end{definition} 

These quantities will be shown to satisfy that (i) characterizing free independence, (ii) invariance under taking closure $\overline{\mathcal{A}_i}^w$ and (iii) the monotonicity in $\mathcal{A}_i$. Hence they can be understood as a kind of mutual information in free probability. Here is a remark on the choice of $\sigma_0$. 

\begin{remark}\label{R6.2} If $\mathcal{J}_{\sigma_0}^\mathrm{lib}(\mathcal{A}_1;\dots;\mathcal{A}_n:\mathcal{A}_{n+1})$ is finite, then $\lambda_i^*(\sigma_0)$ must agree with $\tau$ on $\mathcal{A}_i$ for every $1 \leq i \leq n+1$. 
\end{remark} 
\begin{proof} 
Assume that $\lambda_i^*(\sigma_0)$ does not agree with $\tau$ for some $i$. Namely, there is an element $x \in \mathcal{A}_i$ such that $\sigma_0(\lambda_i(x)) \neq \tau(x)$. Remark that $\tau(x) = \Upsilon^*(\tau)(\lambda_i(x))$. Then we can choose an open neighborhood $O \in \mathcal{O}(\Upsilon^*(\tau))$ in such a way that $\sigma(\lambda_i(x)) \neq \sigma_0(\lambda_i(x))$ for every $\sigma \in O$. As in the proof of \cite[Proposition 5.7]{Ueda:JOTP19} we have
\[
r(\rho_T^*(\varphi)(\lambda_i(x)) - \sigma_0(\lambda_i(x))) 
=
\mathcal{I}_{\sigma_0,\infty}^\mathrm{lib}(\varphi,\rho_{T,i}(x)) \leq \mathcal{I}_{\sigma_0,\infty}^\mathrm{lib}(\varphi)
\]
for all $r \in \mathbb{R}$ and $T \geq 0$.  It follows that $\mathcal{I}_{\sigma_0,\infty}^\mathrm{lib}(\varphi) = +\infty$ as long as $\rho_T^*(\varphi) \in O$. It follows that $\mathcal{J}_{\sigma_0}^\mathrm{lib}(\mathcal{A}_1;\dots;\mathcal{A}_n:\mathcal{A}_{n+1}) = \mathcal{J}_{\sigma_0}^\mathrm{lib}(\Upsilon^*(\tau)) \geq \mathcal{J}_{\sigma_0,\infty}^\mathrm{lib}(\Upsilon^*(\tau)) = +\infty$. 
\end{proof}

\emph{Consequently, we will assume that $\lambda_i^*(\sigma_0)$ agrees with $\tau$ on $\mathcal{A}_i$ for every $1 \leq i \leq n+1$ throughout the rest of this section. In particular, the natural two choices of $\sigma_0$ are $\Upsilon^*(\tau)$ and the so-called free product state $\bigstar_{i=1}^{n+1}(\lambda_i^{-1})^*(\tau)$.} 

\subsection{Relation to the matrix liberation process}
Assume that each $\mathcal{A}_i$, $1 \leq i \leq n+1$, is generated by a self-adjoint random multi-variable $\mathbf{X}_i = (X_{ij})_{j=1}^{r(i)}$ as in section 3, that is, $\mathcal{A}_i = C^*(\mathbf{X}_i)$. Assume further that $R := \sup_{i,j} \Vert X_{ij}\Vert_\infty < +\infty$. Then we have two unique surjective unital $*$-homomorphisms $\Phi : C^*_R\langle x_{\bullet\diamond}\rangle \to \mathfrak{A}$ and $\Psi : C^*_R\langle x_{\bullet\diamond}(\,\cdot\,), v_\bullet(\,\cdot\,)\rangle \to \tilde{\mathfrak{A}}(\mathbb{R}_+)$ sending $x_{ij}$, $x_{ij}(t)$ and $v_i(t)$ to $\lambda_i(X_{ij})$, $\rho_{t,i}(X_{ij}) = \rho_t(\lambda_i(X_{ij}))$ and $u_i(t)$, respectively. Clearly, $\Psi(C^*_R\langle x_{\bullet\diamond}(\,\cdot\,)\rangle) = \mathfrak{A}(\mathbb{R}_+)$ and $\Psi(x_{ij}(t)) = \rho_t(\Phi(x_{ij}))$ hold. In particular, the latter implies that $\Psi\circ\pi_0 = \rho_0\circ\Phi$. 

\medskip
For the reader's convenience we summarize the notations of algebras and maps that we have introduced so far. The algebras and the maps between them are: 
\[
\xymatrix{
& C^*_R\langle x_{\bullet\diamond} \rangle \ar[r]^-{\pi_t} \ar[d]_-\Phi & 
C^*_R\langle x_{\bullet\diamond}(\,\cdot\,)\rangle \ar@{^{(}-^{>}}[r] \ar[d]^-\Psi & 
C^*_R\langle x_{\bullet\diamond}(\,\cdot\,),v_\bullet(\,\cdot\,)\rangle \ar[d]^-\Psi 
\\
\mathcal{A}_i  \ar[r]^-{\lambda_i}  \ar@(dr,dl)[rr]_{\rho_{t,i} = \rho_t\circ\lambda_i} &  \mathfrak{A} \ar[r]^-{\rho_t} & \mathfrak{A}(\mathbb{R}_+) \ar@{^{(}-^{>}}[r]  & \tilde{\mathfrak{A}}(\mathbb{R}_+).  
}
\]
The liberation cyclic derivatives $\mathfrak{D}_s^{(k)}$ (see subsection 4.2) and the maps $\Pi^s$ (see subsection 4.5) on the upper line of the above diagram correspond to $\nabla_s^{(k)}$ (see  subsection 6.2) and $\Lambda^s$ (see subsection 6.3) on the lower line, respectively. Moreover, the spaces of (continuous) tracial states and the dual maps between them are:
\[
\xymatrix{ 
& TS(C^*_R\langle x_{\bullet\diamond} \rangle)  & TS^c(C^*_R\langle x_{\bullet\diamond}(\,\cdot\,)\rangle) \ar[l]_-{\pi_t^*} \ar[r]^-{\tau \mapsto \tilde{\tau}} & TS^c(C^*_R\langle x_{\bullet\diamond}(\,\cdot\,),v_\bullet(\,\cdot\,)\rangle) \\ 
TS(\mathcal{A}_i) & TS(\mathfrak{A}) \ar[l]_{\lambda_i^*} \ar[u]^-{\Phi^*} & TS^c(\mathfrak{A}(\mathbb{R}_+)) \ar[l]_-{\rho_t^*} \ar[u]_-{\Psi^*} \ar@(dl,dr)[ll]^{\rho_{t,i}^* = \lambda_i^*\circ\rho_t^*} \ar[r]_-{\varphi \mapsto \tilde{\varphi}} & TS^c(C^*_R\langle x_{\bullet\diamond}(\,\cdot\,),v_\bullet(\,\cdot\,)\rangle). \ar[u]_-{\Psi^*}
}
\]

\begin{lemma}\label{L6.3} For any $\varphi \in TS^c(\mathfrak{A}(\mathbb{R}_+))$ we have $\Psi^*(\varphi) := \varphi\circ\Psi \in TS^c(C^*_R\langle x_{\bullet\diamond}(\,\cdot\,)\rangle)$ and $\Psi^*(\tilde{\varphi}) = \Psi^*(\varphi)^\sim$. Hence $\Psi^*(\varphi)^s = \Psi^*(\varphi^s)$ holds for every $s \geq 0$. Moreover, for any $P \in \mathbb{C}\langle x_{\bullet\diamond}(\,\cdot\,)\rangle$, we have 
\[
\Vert E_{\mathcal{Q}(\varphi)}(\pi_{\tilde{\varphi}}(\Lambda^s(\nabla_s^{(k)}\Psi(P)))) \Vert_{\tilde{\varphi},2} 
= 
\Vert E_{\mathcal{N}(\Psi^*(\varphi))}(\pi_{\Psi^*(\varphi)^\sim}(\Pi^s(\mathfrak{D}_s^{(k)}P))) \Vert_{\Psi^*(\varphi)^\sim,2} 
\]
for every $1 \leq k \leq n$ and $s \geq 0$. 
\end{lemma} 
\begin{proof} 
Observe that 
\[
\Psi^*(\varphi)(x_{i_1 j_1}(t_1) \cdots x_{i_m j_m}(t_m)) = \varphi(\rho_{t_1,i_1}(X_{i_1 j_1})\cdots\rho_{t_m,i_m}(X_{i_m j_m})),
\]
which implies that $\Psi^*(\varphi)$ falls in $TS^c(C^*_R\langle x_{\bullet\diamond}(\,\cdot\,)\rangle)$ by \cite[Lemma 2.1]{Ueda:JOTP19} and Lemma \ref{L6.1}. Moreover, we have 
\[
\Psi^*(\tilde{\varphi})(a_1 v_{i_1}(t_1)^{\epsilon_1} \cdots a_m v_{i_m}(t_m)^{\epsilon_m}) 
= 
\tilde{\varphi}(\Psi(a_1) u_{i_1}(t_1)^{\epsilon_1} \cdots \Psi(a_m) u_{i_m}(t_m)^{\epsilon_m})  
\]
for any $a_k \in C^*_R\langle x_{\bullet\diamond}(\,\cdot\,)\rangle$, $1 \leq i_k \leq n$, $t_k \geq 0$ and $\epsilon_k = \pm1$. Since $\Psi(C^*_R\langle x_{\bullet\diamond}(\,\cdot\,)\rangle) = \mathfrak{A}(\mathbb{R}_+)$, we conclude that the $v_i(t)$ are freely independent of $C^*_R\langle x_{\bullet\diamond}(\,\cdot\,)\rangle$ and form a freely independent family of left-multiplicative free unitary Brownian motions under $\Psi^*(\tilde{\varphi})$. Therefore, we conclude that $\Psi^*(\tilde{\varphi}) = \Psi^*(\varphi)^\sim$. We observe that 
\begin{align*}
&\Psi(\Pi^s(x_{ij}(t))) 
= 
\Psi(x_{ij}^s(t)) \\
&= 
\begin{cases}
\Psi(v_i((t-s)\wedge0)x_{ij}(s\wedge t)v_i((t-s)\wedge0)^*) \\
\phantom{aaaaaaaaaaaaa} = 
u_i((t-s)\wedge0)\rho_{s\wedge t,i}(X_{ij})u_i((t-s)\wedge0)^* & (1 \leq i \leq n), \\
\Psi(x_{n+1\,j}(t)) = \rho_{t,n+1}(X_{n+1\,j}) & (i=n+1)
\end{cases} \\
&= 
\rho_{t,i}^s(X_{ij}) \\
&= 
\Lambda^s(\rho_{t,i}(X_{ij})) = \Lambda^s(\Psi(x_{ij}(t))), 
\end{align*} 
implying that $\Psi\circ\Pi^s = \Lambda^s\circ\Psi$ on $C^*_R\langle x_{\bullet\diamond}(\,\cdot\,)\rangle$. Therefore, we obtain that  
\[
\Psi^*(\varphi^s) = \tilde{\varphi}\circ\Lambda^s\circ\Psi = \tilde{\varphi}\circ\Psi\circ\Pi^s = \Psi^*(\tilde{\varphi})\circ\Pi^s = \Psi^*(\varphi)^\sim \circ \Pi^s = \Psi^*(\varphi)^s.
\] 

Choose an arbitrary monomial $P = x_{i_1 j_1}(t_1) \cdots x_{i_m j_m}(t_m) \in \mathbb{C}\langle x_{\bullet\diamond}(\,\cdot\,)\rangle$. By definition we have $\Psi(P) = \rho_{t_1,i_1}(X_{i_1 j_1})\cdots\rho_{t_m,i_m}(X_{i_m j_m})$. We observe that 
\begin{equation}\label{Eq6.6}
\begin{aligned} 
&\Pi^s(\mathfrak{D}_s^{(k)}P) \\
&= 
\sum_{\substack{i_l = k \\ t_l \geq s}} 
\Pi^s([v_k(t_l-s)^* x_{i_{l+1}j_{l+1}}(t_{l+1})\cdots x_{i_{l-1}j_{l-1}}(t_{l-1})v_k(t_l - s), x_{i_l j_l}(s)])) \\
&= 
\sum_{\substack{i_l = k \\ t_l \geq s}} 
[v_k(t_l-s)^* x^s_{i_{l+1}j_{l+1}}(t_{l+1})\cdots x^s_{i_{l-1}j_{l-1}}(t_{l-1})v_k(t_l - s), x^s_{i_l j_l}(s)], \\
&\Lambda^s(\nabla_s^{(k)}(\Psi(P))) \\
&=   
\sum_{\substack{i_l = k \\ t_l \geq s}} 
\Lambda^s([u_k(t_l-s)^* \rho_{t_{l+1},i_{l+1}}(X_{i_{l+1}j_{l+1}})\cdots \rho_{t_{l-1},i_{l-1}}(X_{i_{l-1}j_{l-1}})u_k(t_l - s), \rho_{s,i_l}(X_{i_l j_l})])) \\
&=   
\sum_{\substack{i_l = k \\ t_l \geq s}} 
([u_k(t_l-s)^* \rho^s_{t_{l+1},i_{l+1}}(X_{i_{l+1}j_{l+1}})\cdots \rho^s_{t_{l-1},i_{l-1}}(X_{i_{l-1}j_{l-1}})u_k(t_l - s),\rho^s_{s,i_l}(X_{i_l j_l})])). 
\end{aligned}
\end{equation}
Since $\Psi^*(\varphi)^\sim = \Psi^*(\tilde{\varphi})$ and since $\Psi(x_{ij}(t)) = \rho_{t,i}(X_{ij})$ and $\Psi(v_i(t)) = u_i(t)$, we observe that the joint distribution of the $x_{ij}(t)$ and the $v_i(t)$ under $\Psi^*(\varphi)^\sim$ coincides with that of the $\rho_{t,i}(X_{ij})$ and the $u_i(t)$ under $\tilde{\varphi}$. Moreover, $\mathcal{N}(\Psi^*(\varphi))$ is generated by the $\pi_{\Psi^*(\varphi)^\sim}(x_{ij}(t))$ and also $\mathcal{Q}(\varphi)$ is by the $\pi_{\tilde{\varphi}}(\rho_{t,i}(X_{ij}))$. These together with the definitions of $x_{ij}^s(t)$ and $\rho_{t,i}^s(X_{ij})$ imply the desired $2$-norm equality.  
\end{proof} 

\begin{proposition}\label{P6.4} With $\Phi^*(\sigma_0) := \sigma_0\circ\Phi \in TS(C^*_R\langle x_{\bullet\diamond}\rangle)$ we have 
\[
I_{\Phi^*(\sigma_0)}^\mathrm{lib}(\Psi^*(\varphi)) = \mathcal{I}_{\sigma_0}^\mathrm{lib}(\varphi), \quad 
I_{\Phi^*(\sigma_0),\infty}^\mathrm{lib}(\Psi^*(\varphi)) = \mathcal{I}_{\sigma_0}^\mathrm{lib}(\varphi).  
\]
for any $\varphi \in TS^c(\mathfrak{A}(\mathbb{R}_+))$. Moreover, $\Psi^*(TS^c(\mathfrak{A}(\mathbb{R}_+)))$ is an essential domain of both the functionals $I_{\Phi^*(\sigma_0)}^\mathrm{lib}, I_{\Phi^*(\sigma_0),\infty}^\mathrm{lib}$, that is, the functionals take $+\infty$ outside it. 
\end{proposition} 
\begin{proof} 
We first remark the following facts: 
\begin{itemize} 
\item $\Psi^*(\varphi)^t(P) = \Psi^*(\varphi^t)(P) = \varphi^t(\Psi(P))$ for any $P \in \mathbb{C}\langle x_{\bullet\diamond}(\,\cdot\,)\rangle$.  
\item If $\rho_0^*(\varphi) = \sigma_0$, then $\pi_0^*(\Psi^*(\varphi)) = \varphi\circ\Psi\circ\pi_0 = \varphi\circ\rho_0\circ\Phi = \Phi^*(\sigma_0)$. Thus, $\Phi^*(\sigma_0)^\mathrm{lib}(P) = \sigma_0^\mathrm{lib}(\Psi(P))$ for any $P \in \mathbb{C}\langle x_{\bullet\diamond}(\,\cdot\,)\rangle$.    
\end{itemize}
Thus, (the last equation in) Lemma \ref{L6.3} shows that 
\[
I_{\Phi^*(\sigma_0),t}^\mathrm{lib}(\Psi^*(\varphi),P) = \mathcal{I}_{\sigma_0,t}^\mathrm{lib}(\varphi,\Psi(P))
\]
holds for any $P \in \mathbb{C}\langle x_{\bullet\diamond}(\,\cdot\,)\rangle$. Note that $\Psi(\mathbb{C}\langle x_{\bullet\diamond}(\,\cdot\,)\rangle) \subset \mathfrak{P}(\mathbb{R}_+)$. Hence the above identity at least gives 
\[
I_{\Phi^*(\sigma_0)}^\mathrm{lib}(\Psi^*(\varphi)) \leq \mathcal{I}_{\sigma_0}^\mathrm{lib}(\varphi), \quad I_{\Phi^*(\sigma_0),\infty}^\mathrm{lib}(\Psi^*(\varphi)) \leq \mathcal{I}_{\sigma_0,\infty}^\mathrm{lib}(\varphi). 
\]
To show the reverse inequality in both, it suffices to prove: \begin{center} 
($\diamondsuit$) For any $Q = Q^* \in \mathfrak{A}(\mathbb{R}_+)$ there is a sequence $Q_k = Q_k^*$ in $\Psi(\mathbb{C}\langle x_{\bullet\diamond}(\,\cdot\,)\rangle)$ \\ such that $\mathcal{I}_{\sigma_0,t}^\mathrm{lib}(\tau,Q_k) \to \mathcal{I}_{\sigma_0,t}^\mathrm{lib}(\tau,Q)$ for all $t \in [0,\infty]$.
\end{center} 
Remark that $Q$ is a finite sum of monomials, say $W = \rho_{t_1,i_1}(x_1)\cdots\rho_{t_m,i_m}(x_m)$ with $x_\ell \in \mathcal{A}_{i_\ell}$. Since the unital $*$-subalgebra $\mathcal{A}_{i,0}$ algebraically generated by $(X_{ij})_{j=1}^{r(i)}$ is norm-dense in $\mathcal{A}_i$, we can choose norm-bounded sequences $x_\ell^{(p)}$ in $\mathcal{A}_{i_{\ell},0}$ in such a way that $x_\ell^{(p)} \to x_\ell$ in norm as $p \to \infty$ for every $1 \leq \ell \leq m$. Since $\Psi(x_{ij}(t)) = \rho_{t,i}(X_{ij})$ and $\rho_{t,i}$ is a unital $*$-homomorphism, $W_p := \rho_{t_1,i_1}(x_1^{(p)})\cdots\rho_{t_m,i_m}(x_m^{(p)})$ falls into $\Psi(\mathbb{C}\langle x_{\bullet\diamond}(\,\cdot\,)\rangle)$ and converges to $W$ in norm as $p\to\infty$. Moreover, using expression \eqref{Eq6.6} we can easily see that both $\Lambda^s(\nabla_s^{(k)} W_p) \to \Lambda^s(\nabla_s^{(k)}W)$ and $\Lambda^s(\nabla_s^{(k)} W_p^*) \to \Lambda^s(\nabla_s^{(k)}W^*)$ in norm and uniformly in $s$ as $p \to \infty$. Since all the maps involved are linear, we have proved the desired assertion ($\diamondsuit$) by taking, if necessary, the (operator-theoretic) real part of the approaching sequence that we have obtained. Hence, we complete the proof of the first part of the proposition. 

We will then prove the second part of the proposition. Choose $\psi \in TS^c(C^*_R\langle x_{\bullet\diamond}(\,\cdot\,)\rangle)$ with $I_{\Phi^*(\sigma_0),\infty}^\mathrm{lib}(\psi) < +\infty$. By (the proof of) \cite[Proposition 5.7]{Ueda:JOTP19} we have $\pi_t^*(\psi) = \Phi^*(\sigma_0)$ on $C_R^*\langle x_{i\diamond}\rangle$, the unital $C^*$-subalgebra generated by the $x_{ij}$, $j \geq 1$, with fixing $i$, for each $1 \leq i \leq n+1$. Denote by $\Phi_i$ the restriction of $\Phi : C^*_R\langle x_{\bullet\diamond} \rangle \to \mathfrak{A}$ to each $C^*_R\langle x_{i\diamond} \rangle$. Since $\Phi_i :  C^*_R\langle x_{i\diamond} \rangle \to \lambda_i(\mathcal{A}_i)$ is a surjective $*$-homomorphism, we obtain a bijective unital $*$-homomorphism $\lambda_i(\mathcal{A}_i) \cong C^*_R\langle x_{i\diamond} \rangle/\mathrm{Ker}(\Phi_i)$ sending $\lambda_i(X_{ij})$ to $x_{ij} + \mathrm{Ker}(\Phi_i)$ for $j \geq 1$. Consider the GNS representation $\pi_\psi : C^*_R\langle x_{\bullet\diamond}(\,\cdot\,)\rangle \curvearrowright \mathcal{H}_\psi$. For any $y \in \mathrm{Ker}(\Phi_i)$ we have 
\[
\psi(\pi_t(y)^* \pi_t(y)) = \pi_t^*(\psi)(y^* y) = \Phi^*(\sigma_0)(y^* y) = \sigma_0(\Phi_i(y)^* \Phi_i(y)) = 0, 
\]
and hence $\pi_\psi(\pi_t(y)) = 0$ thanks to the trace property of $\psi$. Therefore, by the $C^*$-algebraic freeness among the $\rho_{t,i}(\mathcal{A}_i)$ ($\cong \lambda_i(\mathcal{A}_i) \cong C^*_R\langle x_{i\diamond} \rangle/\mathrm{Ker}(\Phi_i)$ by $\rho_{t,i}(X_{ij}) \leftrightarrow \lambda_i(X_{ij}) \leftrightarrow x_{ij} + \mathrm{Ker}(\Phi_i)$ as remarked before), we obtain a unique unital $*$-homomorphism from $\mathfrak{A}(\mathbb{R}_+)$ to $B(\mathcal{H}_{\tau'})$ sending each $\rho_{t,i}(X_{ij})$ to $\pi_\psi(\pi_t(x_{ij})) = \pi_\psi(x_{ij}(t))$. Then the pull-back of $\psi$ by this $*$-homomorphism defines a tracial state $\varphi$ on $\mathfrak{A}(\mathbb{R}_+)$, under which the $\rho_{t,i}(X_{ij})$ have the same joint distribution as that of the $x_{ij}(t)$ under $\psi$. This means that $\Psi^*(\varphi) = \psi$ and the continuity of $\varphi$ follows thanks to Lemma \ref{L6.1}. Hence we are done. 
\end{proof} 

\begin{corollary}\label{C6.5} In the same setting as in Proposition \ref{P6.4} we have 
\begin{equation}\label{Eq6.4} 
J^\mathrm{lib}_{\Phi^*(\sigma_0)}(\Phi^*(\sigma)) = \mathcal{J}^\mathrm{lib}_{\sigma_0}(\sigma), \quad 
J^\mathrm{lib}_{\Phi^*(\sigma_0),\infty}(\Phi^*(\sigma)) = \mathcal{J}^\mathrm{lib}_{\sigma_0,\infty}(\sigma) 
\end{equation}
for any $\sigma \in TS(\mathfrak{A})$. In particular, the following are equivalent: 
\begin{itemize}
\item[(1)] $\mathcal{A}_i$, $1 \leq i \leq n+1$, are freely independent. 
\item[(2)] $\mathcal{J}_{\sigma_0}^\mathrm{lib}(\mathcal{A}_1;\dots;\mathcal{A}_n:\mathcal{A}_{n+1}) = 0$. 
\item[(3)] $\mathcal{J}^\mathrm{lib}_{\sigma_0,\infty}(\mathcal{A}_1;\dots;\mathcal{A}_n:\mathcal{A}_{n+1}) = 0$. 
\end{itemize}
Moreover, 
\begin{equation}\label{Eq6.5} 
\chi_\mathrm{orb}(\mathbf{X}_1,\dots,\mathbf{X}_{n+1}) 
\leq 
-\mathcal{J}^\mathrm{lib}_{\sigma_0}(\mathcal{A}_1;\dots;\mathcal{A}_n:\mathcal{A}_{n+1})
\leq 
-\mathcal{J}^\mathrm{lib}_{\sigma_0,\infty}(\mathcal{A}_1;\dots;\mathcal{A}_n:\mathcal{A}_{n+1}), 
\end{equation}
at least when $\sigma_0$ is either $\Upsilon^*(\tau)$ or $\bigstar_{i=1}^{n+1} (\lambda_i^{-1})^*(\tau)$.  
\end{corollary} 
\begin{proof} 
We will first prove two identities \eqref{Eq6.4}, which enables us to derive the equivalence of (1) -- (3) from Theorem \ref{T5.3} immediately.  
In the current setting, an open neighborhood basis at $\sigma$ in $TS(\mathfrak{A})$ should be given as a collection of $O_{m,\delta}(\sigma)$, where $O_{m,\delta}(\sigma)$ is all the $\sigma' \in TS(\mathfrak{A})$ such that 
\[
|\sigma'(\lambda_{i_1}(X_{i_1 j_1})\cdots \lambda_{i_p}(X_{i_p j_p})) - \sigma(\lambda_{i_1}(X_{i_1 j_1})\cdots \lambda_{i_p}(X_{i_p j_p}))| < \delta
\]
whenever $1 \leq i_k \leq n+1$, $1 \leq j_k \leq m$, $1 \leq k \leq p$ and $1 \leq p \leq m$. Thus, $\sup_{O \in \mathcal{O}(\sigma)}$ and $\rho_T^*(\varphi) \in O$ can/should be replaced with $\lim_{m,\delta}$ and $\rho_T^*(\varphi) \in O_{m,\delta}(\sigma)$, respectively. By definition we observe that 
\begin{align*}
&|\pi_T^*(\Psi^*(\varphi))(x_{i_1 j_1}\cdots x_{i_p j_p}) - \Phi^*(\sigma)(x_{i_1 j_1}\cdots x_{i_p j_p})| \\
&= 
|\rho_T^*(\varphi)(\lambda_{i_1}(X_{i_1 j_1})\cdots \lambda_{i_p}(X_{i_p j_p})) - \sigma(\lambda_{i_1}(X_{i_1 j_1})\cdots \lambda_{i_p}(X_{i_p j_p}))|. 
\end{align*}
Hence $\pi_T^*(\Psi(\tau)) \in O_{m,\delta}(\Phi^*(\sigma))$ if and only if $\rho_T^*(\varphi) \in O_{m,\delta}(\sigma)$. Moreover, $\Psi^*(TS^c(C^*_R\langle x_{\bullet\diamond}\rangle)$ is an essential domain for the functionals by Proposition \ref{P6.4}. Therefore, the main identities in Proposition \ref{P6.4} imply two identities \eqref{Eq6.5}.  

Since 
\begin{align*} 
\Phi^*(\Upsilon^*(\tau))(x_{i_1 j_1}\cdots x_{i_m j_m}) 
&= 
\tau(X_{i_1 j_1}\cdots X_{i_m j_m}), \\
\Phi^*(\bigstar_{i=1}^{n+1} (\lambda_i^{-1})^*(\tau))(x_{i_1 j_1}\cdots x_{i_m j_m}) 
&= 
\bigstar_{i=1}^{n+1} (\lambda_i^{-1})^*(\tau)(\lambda_{i_1}(X_{i_1 j_1})\cdots\lambda_{i_m}(X_{i_m j_m})), 
\end{align*} 
Corollary \ref{C4.3} together with Propositions \ref{P3.2}, \ref{P3.3} implies inequality \eqref{Eq6.5}. 
\end{proof} 

\begin{remarks}\label{R6.6} (1) The part characterizing free independence by $\mathcal{J}_{\sigma_0}^\mathrm{lib}$ as well as $\mathcal{J}_{\sigma_0,\infty}^\mathrm{lib}$ in the above corollary can directly be proved by using the same argument as in \S5 without appealing to generators of each $\mathcal{A}_i$. 

(2) The last two assertions of the above corollary suggests that $\mathcal{J}_{\sigma_0}^\mathrm{lib}(\mathcal{A}_1;\cdots;\mathcal{A}_n:\mathcal{A}_{n+1})$ may be independent of $\sigma_0$, at least under some constraint. However, this question is untouched yet due to the lack of techniques to discuss `minimal paths' of tracial states under the functionals. 
\end{remarks}  

\subsection{Invariance under weak closure}
Corollary \ref{C6.5} suggests that $\mathcal{J}_{\sigma_0}^\mathrm{lib}(\mathcal{A}_1;\dots;\mathcal{A}_n:\mathcal{A}_{n+1})$ as well as $\mathcal{J}^\mathrm{lib}_{\sigma_0,\infty}(\mathcal{A}_1;\dots;\mathcal{A}_n:\mathcal{A}_{n+1})$ are $W^*$-invariants, that is, they are unchanged if each $\mathcal{A}_i$ is replaced with its $\sigma$-weak closure $\overline{\mathcal{A}_i}^w$. This is indeed the case, as we will see below. The proof is rather technical, but the idea behind it is simple. 

\medskip
Let us denote by $\mathfrak{M}$ and $\mathfrak{M}(\mathbb{R}_+) \subset \tilde{\mathfrak{M}}(\mathbb{R_+})$ the $C^*$-algebras corresponding to $\mathfrak{A}$ and $\mathfrak{A}(\mathbb{R}_+) \subset \tilde{\mathfrak{A}}(\mathbb{R_+})$ when each $\mathcal{A}_i$ is replaced with $\mathcal{M}_i:=\overline{\mathcal{A}_i}^w$. Observe that the original $\mathfrak{A}$ and $\mathfrak{A}(\mathbb{R}_+) \subset \tilde{\mathfrak{A}}(\mathbb{R_+})$ are naturally embedded into $\mathfrak{M}$ and $\mathfrak{M}(\mathbb{R}_+) \subset \tilde{\mathfrak{M}}(\mathbb{R_+})$. See Proposition \ref{PA3}. Notations $\lambda_i, \rho_{t,i}, \rho_t$ of morphisms are used simultaneously in what follows. To this end, we need several technical, purely operator algebraic facts (Lemmas \ref{L6.7}--\ref{L6.9}). 

\medskip
The first lemma seems a folklore among operator algebraists, but we do give its proof because it plays a key role in the discussion below.

\begin{lemma}\label{L6.7} Let $\mathcal{A}$ be a $\sigma$-weakly dense, unital $C^*$-subalgebra of a $W^*$-algebra $\mathcal{M}$ and $\varphi$ be a normal state on $\mathcal{M}$. Let $\pi : \mathcal{A} \curvearrowright \mathcal{H}$ be a unital $*$-representation with a distinguished vector $\xi_0 \in \mathcal{H}$ such that $\xi_0$ is separating for $\pi(\mathcal{A})$ and that $(\pi(a)\xi_0|\xi_0)_\mathcal{H} = \varphi(a)$ holds for every $a \in \mathcal{A}$. Then there is a unique normal unital $*$-representation $\bar{\pi} : \mathcal{M} \curvearrowright \mathcal{H}$ extending $\pi$ such that $\bar{\pi}(\mathcal{M}) = \overline{\pi(\mathcal{A})}^w$. 
\end{lemma} 
\begin{proof}
Let $(\mathcal{H}_\varphi,\pi_\varphi,\xi_\varphi)$ be the GNS triple of $(\mathcal{M,\varphi})$. Set $\mathcal{K} := \overline{\pi(\mathcal{A})\xi_0}$, a reducing subspace for $\pi(\mathcal{A})$. Observe, by the uniqueness of GNS representations, that the restriction of $\pi$ to $\mathcal{K}$ with $\xi_0$ is a realization of $(\mathcal{H}_\varphi, {\pi_\varphi}\!\upharpoonright_\mathcal{A},\xi_\varphi)$. Since $\xi_0$ is separating for $\pi(\mathcal{A})$, $\pi$ is quasi-equivalent to $\pi_\varphi$ by \cite[Theorem 10.3.3(ii)]{KadisonRingrose:Book2}. This means that there exists a normal unital, bijective $*$-homomorphism $\rho : \pi_\varphi(\mathcal{M}) = \overline{\pi_\varphi(\mathcal{A})}^w \to \overline{\pi(\mathcal{A})}^w$ sending $\pi_\varphi(a)$ to $\pi(a)$ for every $a \in \mathcal{A}$. Thus, $\bar{\pi} := \rho\circ\pi_\varphi : \mathcal{M} \to \overline{\pi(\mathcal{A})}^w$ is the desired $*$-homomorphism. 
\end{proof} 

We need the next two state extension properties. The proofs crucially use the previous lemma with the universality of universal free products. 

\begin{lemma}\label{L6.8} Any $\sigma_0 \in TS(\mathfrak{A})$ with $\lambda_i^*(\sigma_0) = \tau$ on $\mathcal{A}_i$ for all $1 \leq i \leq n+1$ has a unique extension $\bar{\sigma}_0 \in TS(\mathfrak{M})$ with $\lambda_i^*(\bar{\sigma}_0) = \tau$ on $\mathcal{M}_i$ for all $1 \leq i \leq n+1$. 
\end{lemma}
\begin{proof} Let $(\mathcal{H}_{\sigma_0}, \pi_{\sigma_0},\xi_{\sigma_0})$ be the GNS triple of $(\mathcal{A},\sigma_0)$. Since $\sigma_0$ is tracial, $\xi_{\sigma_0}$ must be separating for $\pi_{\sigma_0}(\mathfrak{A})$. In particular, $\xi_{\sigma_0}$ is separating for each $\pi_{\sigma_0}(\lambda_i(\mathcal{A}_i))$ too. Set $\pi_{\sigma_0,i} := \pi_{\sigma_0}\circ\lambda_i : \mathcal{A}_i \curvearrowright \mathcal{H}_{\sigma_0}$. Then we have $(\pi_{\sigma_0,i}(a)\xi_{\sigma_0}|\xi_{\sigma_0})_{\mathcal{H}_{\sigma_0}} = \sigma_0\circ\lambda_i(a) = \lambda_i^*(\sigma_0)(a) = \tau(a)$ for every $a \in \mathcal{A}_i$. Thus, the previous lemma shows that there exists a unique normal extension $\bar{\pi}_{\sigma_0,i} : \mathcal{M}_i := \overline{\mathcal{A}_i}^w \curvearrowright \mathcal{H}_{\sigma_0}$ such that $\bar{\pi}_{\sigma_0,i}(\mathcal{M}_i) = \overline{\pi_{\sigma_0}(\lambda_i(\mathcal{A}_i))}^w$ and $\bar{\pi}_{\sigma_0,i}\!\upharpoonright_{\mathcal{A}_i} = \pi_{\sigma_0,i}$. By the universality of universal free products, there exists a unique $*$-homomorphism $\bar{\pi}_{\sigma_0} : \mathfrak{M} \to B(\mathcal{H}_{\sigma_0})$ such that $\bar{\pi}_{\sigma_0}\circ\lambda_i = \bar{\pi}_{\sigma_0,i} : \mathcal{M}_i \curvearrowright \mathcal{H}_{\sigma_0}$ is normal for every $1 \leq i \leq n+1$. By construction, it is clear that $\bar{\pi}_{\sigma_0}\!\upharpoonright_\mathfrak{A} = \pi_{\sigma_0}$. Set $\bar{\sigma}_0 := (\bar{\pi}_{\sigma_0}(\,\cdot\,)\xi_{\sigma_0}|\xi_{\sigma_0})_{\mathcal{H}_{\sigma_0}} \in TS(\mathfrak{M})$. Trivially, $\bar{\sigma}_0\!\upharpoonright_{\mathfrak{A}} = \sigma_0$. For each $x_k \in \mathcal{M}_{i_k}$, $1 \leq k \leq m$, by the Kaplansky density theorem, one can choose a net $a_k^{(\kappa)} \in \mathcal{A}_i$ (with a common index set) such that $\Vert a_k^{(\kappa)}\Vert_\infty \leq \Vert x_k\Vert_\infty$ and $a_k^{(\kappa)} \to x_k$ in the $\sigma$-strong$^*$ topology on $\mathcal{M}_{i_k}$. Since each $\bar{\pi}_{\sigma_0,i}$ is normal on $\mathcal{M}_i$, we observe that 
\begin{align*}
\pi_{\sigma_0}(\lambda_{i_1}(a_1^{(\kappa)})\cdots\lambda_{i_m}(a_m^{(\kappa)})) 
&= 
\pi_{\sigma_0,i_1}(a_1^{(\kappa)})\cdots\pi_{\sigma_0,i_m}(a_m^{(\kappa)}) \\
&=
\bar{\pi}_{\sigma_0,i_1}(a_1^{(\kappa)})\cdots\bar{\pi}_{\sigma_0,i_m}(a_m^{(\kappa)}) \\ 
&\to 
\bar{\pi}_{\sigma_0,i_1}(x_1)\cdots\bar{\pi}_{\sigma_0,i_m}(x_m) 
=
\bar{\pi}_{\sigma_0}(\lambda_{i_1}(x_1)\cdots\lambda_{i_m}(x_m)),
\end{align*}
and hence $\bar{\sigma}_0(\lambda_{i_1}(x_1)\cdots\lambda_{i_m}(x_m)) = \lim_\kappa \sigma_0(\lambda_{i_1}(a_1^{(\kappa)})\cdots\lambda_{i_m}(a_m^{(\kappa)}))$. Since the $\lambda_i(\mathcal{M}_i)$ generate $\mathfrak{M}$ as a $C^*$-algebra, we conclude that $\bar{\sigma}_0$ is a unique extension of $\sigma_0$. Moreover, $\lambda_i^*(\bar{\sigma}_0)(x) = \bar{\sigma}_0(\lambda_i(x)) = \lim_\kappa \sigma_0(\lambda_i(a_\kappa)) = \lim_\kappa \lambda_i^*(\sigma_0)(a_\kappa) = \lim_\kappa \tau(a_\kappa) = \tau(x)$ for every $x \in \mathcal{M}_i$ with approximation $a_\kappa \to x$ as above. 
\end{proof} 

\begin{lemma}\label{L6.9} 
Any $\varphi \in TS^c(\mathfrak{A}(\mathbb{R}_+))$ with $\rho_{t,i}^*(\varphi) = \tau$ on $\mathcal{A}_i$ for all $t \geq 0$ and $1 \leq i \leq n+1$ has a unique extension $\bar{\varphi} \in TS^c(\mathfrak{M}(\mathbb{R}_+))$ with $\rho_{t,i}^*(\bar{\varphi}) = \tau$ on $\mathcal{M}_i$ for all $t \geq 0$ and $1 \leq i \leq n+1$. 
\end{lemma}
\begin{proof} Let $(\mathcal{H}_\varphi, \pi_\varphi,\xi_\varphi)$ be the GNS triple of $(\mathfrak{A}(\mathbb{R}_+),\varphi)$. The same argument as in the previous lemma shows that there exists a $*$-representation $\bar{\pi}_\varphi : \mathfrak{M}(\mathbb{R}_+) \curvearrowright \mathcal{H}_\varphi$ such that $\bar{\pi}_\varphi\circ\rho_{t,i} : \mathcal{M}_i \to B(\mathcal{H}_\varphi)$ is normal as well as that $\bar{\pi}_\varphi\circ\rho_{t,i}\!\upharpoonright_{\mathcal{A}_i} = \pi_\varphi\circ\rho_{t,i}$ holds for every $t \geq 0$ and $1 \leq i \leq n+1$. Define $\bar{\varphi} := (\bar{\pi}_\varphi(\,\cdot\,)\xi_\varphi|\xi_\varphi)_{\mathcal{H}_\varphi} \in TS(\mathfrak{M}(\mathbb{R}_+)$. Remark that $\rho_{t,i}^*(\bar{\varphi}) = \tau$ on $\mathcal{M}_i$ holds for every $t \geq 0$ and $1 \leq i \leq n+1$. By the uniqueness of GNS representations, the triple $(\mathcal{H}_\varphi, \bar{\pi}_\varphi,\xi_\varphi)$ is identified with the GNS triple of $(\mathfrak{M}(\mathbb{R}_+),\bar{\varphi})$. Namely, we may and do assume that $\pi_{\bar{\varphi}} = \bar{\pi}_\varphi$, $\mathcal{H}_{\bar{\varphi}} = \mathcal{H}_\varphi$ and $\xi_{\bar{\varphi}} = \xi_\varphi$. 

Since the given $\varphi$ is continuous, the mapping $t \mapsto \pi_{\bar{\varphi}}(\rho_{t,i}(a)) = \pi_\varphi(\rho_{t,i}(a))$ is strongly continuous for every $a \in \mathcal{A}_i$. We claim that this is the case even when $a \in \mathcal{A}_i$ is replaced with an arbitrary $x \in \mathcal{M}_i$. By the Kaplansky density theorem, we can choose a net $a_\kappa \in \mathcal{A}_i$ in such a way that $\Vert a_\kappa\Vert_\infty \leq \Vert x \Vert_\infty$ and $\Vert a_\kappa - x \Vert_{\tau,2} := \sqrt{\tau((a_\kappa - x)^* (a_\kappa - x))} \to 0$. We have 
\begin{align*} 
\Vert \pi_{\bar{\varphi}}(\rho_{t,i}(a_\kappa - x))\xi_{\bar{\varphi}}\Vert_{\mathcal{H}_{\bar{\varphi}}} 
&= 
\sqrt{\rho_{t,i}^*(\bar{\varphi})((a_\kappa - x)^* (a_\kappa - x))} \\
&= 
\sqrt{\tau((a_\kappa - x)^* (a_\kappa - x))} 
= 
\Vert a_\kappa - x \Vert_{\tau,2}.
\end{align*} 
For any $\eta \in \mathcal{H}_{\bar{\varphi}}$ and any $\varepsilon > 0$, there is a $Y' \in \pi_{\bar{\varphi}}(\mathfrak{M}(\mathbb{R}_+))'$ such that $\Vert \eta - Y'\xi_{\bar{\varphi}}\Vert_{\mathcal{H}_{\bar{\varphi}}} < \varepsilon$ ({\it n.b.}, $\xi_\varphi$ is separating for $\pi_{\bar{\varphi}}(\mathfrak{M}(\mathbb{R}_+))$, and the existence of such a $Y'$ is guaranteed). Then 
\begin{align*}
\Vert \pi_{\bar{\varphi}}(\rho_{t,i}(a_\kappa - x))\eta\Vert_{\mathcal{H}_{\bar{\varphi}}} 
&\leq 
2\Vert x\Vert_\infty \Vert \eta - Y'\xi_{\bar{\varphi}}\Vert_{\mathcal{H}_{\bar{\varphi}}} 
+ 
\Vert Y'\Vert_\infty \Vert \pi_{\bar{\varphi}}(\rho_{t,i}(a_\kappa - x))\xi_{\bar{\varphi}}\Vert_{\mathcal{H}_{\bar{\varphi}}} \\
&\leq 
2\Vert x\Vert_\infty \varepsilon + \Vert Y'\Vert_\infty \Vert a_\kappa - x \Vert_{\tau,2},
\end{align*}
and hence 
\[
\lim_\kappa \big(\sup_{t \geq 0} \Vert \pi_{\bar{\varphi}}(\rho_{t,i}(a_\kappa - x))\eta\Vert_{\mathcal{H}_{\bar{\varphi}}}\big) = 0.  
\]
Then, we can see that $t \mapsto \pi_{\bar{\varphi}}(\rho_{t,i}(x))$ is strongly continuous for every $x \in \mathcal{M}_i$. It follows thanks to Lemma \ref{L6.1} (iii) that $\bar{\varphi}$ is continuous. 
\end{proof} 

Here is an important remark obtained from the above proof. 

\begin{remark}\label{R6.10} We keep the notations $\varphi$, $\bar{\varphi}$, etc., of the previous lemma. If a bounded net $a^{(\kappa)}$ in $\mathcal{A}_i$ converges to $x \in \mathcal{M}_i$ in $\Vert\,\cdot\,\Vert_{\tau,2}$ or equivalently, in the $\sigma$-strong$^*$ topology on $\mathcal{M}_i$, then 
\[
\lim_\kappa \big(\sup_{t \geq 0} \Vert \pi_{\bar{\varphi}}(\rho_{t,i}(a^{(\kappa)} - x))\xi\Vert_{\mathcal{H}_{\bar{\varphi}}}\big) = 0
\]
for every $\xi \in \mathcal{H}_{\bar{\varphi}}$, that is, the convergence $\pi_{\bar{\varphi}}(\rho_{t,i}(a^{(\kappa)})) \to \pi_{\bar{\varphi}}(\rho_{t,i}(x))$ in the strong operator topology is uniform for $t \geq 0$. 
\end{remark}

\begin{lemma}\label{L6.11} For any $\varphi \in TS^c(\mathfrak{A}(\mathbb{R}_+)$ with $\rho_{t,i}^*(\varphi) = \tau$ on $\mathcal{A}_i$ for all $t \geq 0$ as well as $\lambda_i^*(\sigma_0) = \tau$ on $\mathcal{A}_i$ for all $1 \leq i \leq n+1$, we have $\mathcal{I}^\mathrm{lib}_{\sigma_0}(\varphi) = \mathcal{I}^\mathrm{lib}_{\bar{\sigma}_0}(\bar{\varphi})$ as well as $\mathcal{I}^\mathrm{lib}_{\sigma_0,\infty}(\varphi) = \mathcal{I}^\mathrm{lib}_{\bar{\sigma}_0,\infty}(\bar{\varphi})$ with the notations in the previous lemmas. 
\end{lemma}
\begin{proof} 
The same pattern as in the proof of Proposition \ref{P6.4} (and Lemma \ref{L6.3}) works well by replacing the norm convergence $x_\ell^{(p)} \to x_\ell$ with a bounded net convergence $a_\ell^{(\kappa)} \to x_\ell$ in the $\sigma$-strong$^*$ topology with the help of Remark \ref{R6.10}.  
\end{proof} 

Here is the desired statement. Namely, the next proposition tells us that taking the $\sigma$-weak closure does not give any effect to $\mathcal{J}^\mathrm{lib}_{\sigma_0}$ as well as $\mathcal{J}^\mathrm{lib}_{\sigma_0,\infty}$. This is analogous to \cite[Remarks 10.2]{Voiculescu:AdvMath99}.  

\begin{proposition}\label{P6.12} With the notations as in the previous lemmas we have
\begin{align*} 
\mathcal{J}_{\sigma_0}^\mathrm{lib}(\mathcal{A}_1;\dots;\mathcal{A}_n:\mathcal{A}_{n+1}) 
&= 
\mathcal{J}^\mathrm{lib}_{\bar{\sigma}_0}(\mathcal{M}_1;\cdots;\mathcal{M}_n:\mathcal{M}_{n+1}), \\
\mathcal{J}_{\sigma_0,\infty}^\mathrm{lib}(\mathcal{A}_1;\dots;\mathcal{A}_n:\mathcal{A}_{n+1}) 
&= 
\mathcal{J}^\mathrm{lib}_{\bar{\sigma}_0,\infty}(\mathcal{M}_1;\cdots;\mathcal{M}_n:\mathcal{M}_{n+1})
\end{align*} 
as long as $\lambda_i^*(\sigma_0) = \tau$ on $\mathcal{A}_i$ for all $1 \leq i \leq n+1$. 
\end{proposition}
\begin{proof} 
For the ease of notations we will write $\sigma := \Upsilon^*(\tau) \in TS(\mathfrak{A})$ and $\bar{\sigma} := \bar{\Upsilon}^*(\tau) \in TS(\mathfrak{M})$, where $\Upsilon : \mathfrak{A} \to \mathcal{M}$ and $\bar{\Upsilon} : \mathfrak{M} \to \mathcal{M}$ the unital $*$-homomorphisms sending each $\lambda_i(a)$ with $a \in \mathcal{A}_i$ to $a$ and $\lambda_i(x)$ with $x \in \mathcal{M}_i$ to $x$, respectively. In particular, $\bar{\Upsilon}$ is an extension of $\Upsilon$, and hence $\bar{\sigma}$ is an extension of $\sigma$ too.   

We denote by $W$ a word whose letters from the $\lambda_i(\mathcal{A}_i)$ and also by $\bar{W}$ a word whose letters from the $\lambda_i(\mathcal{M}_i)$. According to this notation, we will also denote by $\mathcal{W}$ a finite collection of words $W$ and by $\bar{\mathcal{W}}$ a finite collection of words $\bar{W}$. These play parts of parameters to define neighborhood base of the weak$^*$ topologies on $TS(\mathfrak{A})$ and $TS(\mathfrak{M})$, respectively. 

Let $T \geq 0$, $\delta > 0$, and $\psi \in TS^c(\mathfrak{M}(\mathbb{R}_+))$ be arbitrarily chosen. Denote by $\underline{\psi}$ the restriction of $\psi$ to $\mathfrak{A}(\mathbb{R}_+)$, which clearly falls into $TS^c(\mathfrak{A}(\mathbb{R}_+))$. By construction, it is easy to see that $\mathcal{I}_{\sigma_0}^\mathrm{lib}(\underline{\psi}) \leq \mathcal{I}_{\bar{\sigma}_0}^\mathrm{lib}(\psi)$ holds in general. Hence 
\begin{align*} 
&\inf\{\mathcal{I}_{\sigma_0}^\mathrm{lib}(\varphi) \mid \varphi \in TS^c(\mathfrak{A}(\mathbb{R}_+)), \rho_T^*(\varphi) \in O_{\mathcal{W},\delta}(\sigma) \} \\
&\leq 
\inf\{\mathcal{I}_{\sigma_0}^\mathrm{lib}(\underline{\psi}) \mid \psi \in TS^c(\mathfrak{M}(\mathbb{R}_+)), \rho_T^*(\underline{\psi}) \in O_{\mathcal{W},\delta}(\sigma)\} \\
&\leq 
\inf\{\mathcal{I}_{\bar{\sigma}_0}^\mathrm{lib}(\psi) \mid \psi \in TS^c(\mathfrak{M}(\mathbb{R}_+)), \rho_T^*(\psi) \in O_{\mathcal{W},\delta}(\bar{\sigma})\},  
\end{align*}
where we use that $\rho_T^*(\underline{\psi}) \in O_{\mathcal{W},\delta}(\sigma) \Leftrightarrow \rho_T^*(\psi) \in O_{\mathcal{W},\delta}(\bar{\sigma})$, since every $W \in \mathcal{W}$ falls into $\mathfrak{A}$ (and hence $\sigma(W) = \bar{\sigma}(W)$ and $\underline{\psi}(\rho_t(W)) = \psi(\rho_t(W))$). Taking the $\varlimsup_{T\to\infty}$ of the above inequality, we get 
\begin{align*}
&\varlimsup_{T\to\infty}\inf\{\mathcal{I}_{\sigma_0}^\mathrm{lib}(\varphi) \mid \varphi \in TS^c(\mathfrak{A}(\mathbb{R}_+)), \rho_T^*(\varphi) \in O_{\mathcal{W},\delta}(\sigma) \} \\
&\leq 
\varlimsup_{T\to\infty} \inf\{\mathcal{I}_{\bar{\sigma}_0}^\mathrm{lib}(\psi) \mid \psi \in TS^c(\mathfrak{M}(\mathbb{R}_+)), \rho_T^*(\psi) \in O_{\mathcal{W},\delta}(\bar{\sigma})\} \\
&\leq 
\sup_{\bar{\mathcal{W}},\delta} \varlimsup_{T\to\infty} \inf\{\mathcal{I}_{\bar{\sigma}_0}^\mathrm{lib}(\psi) \mid \psi \in TS^c(\mathfrak{M}(\mathbb{R}_+)), \rho_T^*(\psi) \in O_{\bar{\mathcal{W}},\delta}(\bar{\sigma})\} 
= 
\mathcal{J}_{\bar{\sigma}_0}^\mathrm{lib}(\bar{\sigma}).  
\end{align*} 
Since $(\mathcal{W},\delta)$ is arbitrary, $\mathcal{J}_{\sigma_0}^\mathrm{lib}(\mathcal{A}_1;\dots;\mathcal{A}_n:\mathcal{A}_{n+1}) =\mathcal{J}_{\sigma_0}^\mathrm{lib}(\sigma) \leq \mathcal{J}_{\bar{\sigma}_0}^\mathrm{lib}(\bar{\sigma}) = \mathcal{J}^\mathrm{lib}_{\bar{\sigma}_0}(\mathcal{M}_1;\cdots;\mathcal{M}_n:\mathcal{M}_{n+1})$. The same assertion also holds with the same proof even if $\mathcal{J}_{\sigma_0}^\mathrm{lib}$ and $\mathcal{J}_{\bar{\sigma}_0}^\mathrm{lib}$ are replaced with $\mathcal{J}_{\sigma_0,\infty}^\mathrm{lib}$ and $\mathcal{J}_{\bar{\sigma}_0,\infty}^\mathrm{lib}$, respectively. We remark that the discussion in this paragraph uses only inclusion relation $\mathcal{A}_i \subset \mathcal{M}_i$, $1 \leq i \leq n+1$. This remark will be summarized into the corollary following this proposition.   

\medskip
We will then prove the reverse inequality. To this end, we may assume that $\mathcal{J}_{\sigma_0}^\mathrm{lib}(\mathcal{A}_1;\dots;\mathcal{A}_n:\mathcal{A}_{n+1}) =\mathcal{J}_{\sigma_0}^\mathrm{lib}(\sigma) < +\infty$; otherwise the reverse inequality trivially holds as $-\infty = -\infty$ by the first part of this proof. Let $(\bar{\mathcal{W}},\delta)$ is arbitrarily given. For each $\bar{W} \in \bar{\mathcal{W}}$, we can choose a word $W$ in such a way that 
\[
|\sigma(W) - \bar{\sigma}(\bar{W})| < \frac{\delta}{3}, \quad  
\sup_{T\geq 0}|\rho_T^*(\varphi)(W) - \rho_T^*(\bar{\varphi})(\bar{W})| < \frac{\delta}{3}
\]
whenever $\varphi \in TS^c(\mathfrak{A}(\mathbb{R}_+))$ satisfies that $\rho_{t,i}^*(\varphi) = \tau$ on $\mathcal{A}_i$ for all $t \geq 0$ and $1 \leq i \leq n+1$, where $\bar{\varphi}$ is in the sense of Lemma \ref{L6.11}. This fact can be confirmed by the iterative use of the following observation: Let $X, Y \in \mathfrak{M}$ be given. For any $x \in \mathcal{M}_i$ and $a \in \mathcal{A}_i$ we have 
\begin{align*} 
|\bar{\varphi}(\rho_t(X)\rho_{t,i}(x-a)\rho_t(Y))| 
&\leq 
|(\pi_{\bar{\varphi}}(\rho_t(X))\pi_{\bar{\varphi}}(\rho_{t,i}(x-a))\pi_{\bar{\varphi}}(\rho_t(Y))\xi_{\bar{\varphi}}|\xi_{\bar{\varphi}})_{\mathcal{H}_{\bar{\varphi}}} \\
&\leq 
\Vert X\Vert_\infty \Vert \pi_{\bar{\varphi}}(\rho_{t,i}(x-a))J_{\bar{\varphi}}\pi_{\bar{\varphi}}(\rho_t(Y^*))J_{\bar{\varphi}}\xi_{\bar{\varphi}}\Vert_ {\mathcal{H}_{\bar{\varphi}}} \\
&\leq 
\Vert X\Vert_\infty \Vert J_{\bar{\varphi}}\pi_{\bar{\varphi}}(\rho_t(Y^*))J_{\bar{\varphi}}\pi_{\bar{\varphi}}(\rho_{t,i}(x-a))\xi_{\bar{\varphi}}\Vert_ {\mathcal{H}_{\bar{\varphi}}} \\
&\leq 
\Vert X\Vert_\infty \Vert Y\Vert_\infty \Vert \pi_{\bar{\varphi}}(\rho_{t,i}(x-a))\xi_{\bar{\varphi}}\Vert_ {\mathcal{H}_{\bar{\varphi}}} \\
&= 
\Vert X\Vert_\infty \Vert Y\Vert_\infty \Vert x - a\Vert_{\tau,2} 
\end{align*}
for every $t \geq 0$, where $(\mathcal{H}_{\bar{\varphi}},\pi_{\bar{\varphi}},\xi_{\bar{\varphi}})$ is the GNS triple of $(\mathfrak{M}(\mathbb{R}_+),\bar{\varphi})$ and $J_{\bar{\varphi}}$ is the the so-called modular conjugation, that is, a conjugate-linear isometric map  defined by $J_{\bar{\varphi}}Z\xi_{\bar{\varphi}} = Z^*\xi_{\bar{\varphi}}$ for every $Z \in \pi_{\bar{\varphi}}(\mathfrak{M}(\mathbb{R}_+))''$, the double commutant is taken on $\mathcal{H}_{\bar{\varphi}}$. Similarly, we have
\[
|\bar{\sigma}(X\lambda_i(x-a)Y)| \leq \Vert X\Vert_\infty \Vert Y\Vert_\infty \Vert x-a\Vert_{\tau,2}.
\]
We denote by $\mathcal{W}$ the collection of $W$ with $\bar{W} \in \bar{\mathcal{W}}$ obtained in this way. Let $\varphi \in TS^c(\mathfrak{A}(\mathbb{R}_+))$ be arbitrarily chosen in such a way that $\rho_T^*(\varphi) \in O_{\mathcal{W},\delta/3}(\sigma)$ as well as $\mathcal{I}_{\sigma_0}^\mathrm{lib}(\varphi) < +\infty$. The latter requirement guarantees, by the same proof as in \cite[Proposition 5.7]{Ueda:JOTP19}, that $\rho_{t,i}^*(\varphi) = \tau$ on $\mathcal{A}_i$ for all $t \geq 0$ and $1 \leq i \leq n+1$. By the above consideration we observe that $\bar{\varphi} \in O_{\bar{\mathcal{W}},\delta}(\bar{\sigma})$. Therefore, we conclude that 
\begin{align*} 
&\inf\{ \mathcal{I}_{\bar{\sigma}_0}^\mathrm{lib}(\psi) \mid \psi \in TS^c(\mathfrak{M}(\mathbb{R}_+)), \rho_T^*(\psi) \in O_{\bar{W},\delta}(\bar{\sigma})\} \\
&\leq 
\inf\{ \mathcal{I}_{\bar{\sigma}_0}^\mathrm{lib}(\bar
{\varphi}) = \mathcal{I}_{\sigma_0}^\mathrm{lib}(\varphi) \mid \varphi \in TS^c(\mathfrak{A}(\mathbb{R}_+)), \mathcal{I}_{\sigma_0}^\mathrm{lib}(\varphi) < +\infty, \rho_T^*(\varphi) \in O_{W,\delta/3}(\sigma)\} \\
&= 
\inf\{\mathcal{I}_{\sigma_0}^\mathrm{lib}(\varphi) \mid \varphi \in TS^c(\mathfrak{A}(\mathbb{R}_+)), \rho_T^*(\varphi) \in O_{W,\delta/3}(\sigma)\}. 
\end{align*} 
Taking $\varlimsup_{T\to\infty}$ of this inequality we obtain that
\[
\varlimsup_{T\to\infty}\inf\{ \mathcal{I}_{\bar{\sigma}_0}^\mathrm{lib}(\psi) \mid \psi \in TS^c(\mathfrak{M}(\mathbb{R}_+)), \rho_T^*(\psi) \in O_{\bar{W},\delta}(\bar{\sigma})\} \leq \mathcal{J}_{\sigma_0}^\mathrm{lib}(\sigma),
\]
which implies the desired inequality since $(\bar{\mathcal{W}},\delta)$ is arbitrary. The discussion so far in this paragraph also works again when $\mathcal{J}_{\sigma_0}^\mathrm{lib}$ and $\mathcal{J}_{\bar{\sigma}_0}^\mathrm{lib}$ are replaced with $\mathcal{J}_{\sigma_0,\infty}^\mathrm{lib}$ and $\mathcal{J}_{\bar{\sigma}_0,\infty}^\mathrm{lib}$, respectively. Hence we are done.
\end{proof} 

As remarked in the above proof, we have essentially proved the next monotonicity fact too. 

\begin{corollary}\label{C6.13} If $\mathcal{B}_i \subseteq \mathcal{A}_i$ be a unital $C^*$-subalgebra (possibly $W^*$-subalgebra) for each $1 \leq i \leq n+1$, then 
\[
\mathcal{J}_{\sigma_0}^\mathrm{lib}(\mathcal{B}_1;\cdots;\mathcal{B}_n : \mathcal{B}_{n+1}) \leq 
\mathcal{J}_{\sigma_0}^\mathrm{lib}(\mathcal{A}_1;\cdots;\mathcal{A}_n : \mathcal{A}_{n+1}), 
\]
where $\sigma_0$ on the left-hand side should be understood as the restriction of $\sigma_0$ to the universal $C^*$-algebra obtained from the $\mathcal{B}_i$. 
\end{corollary}  

\subsection{Summary of basic properties}
We have established the next properties of $i^{**}$ so far. 
\begin{itemize} 
\item $i^{**}(\mathcal{A}_1;\cdots;\mathcal{A}_n:\mathcal{A}_{n+1}) = i^{**}(W^*(\mathcal{A}_1);\cdots;W^*(\mathcal{A}_n):W^*(\mathcal{A}_{n+1}))$. 
\item If $\mathcal{B}_i \subset \mathcal{A}_i$, then $i^{**}(\mathcal{B}_1;\cdots;\mathcal{B}_n:\mathcal{B}_{n+1}) \leq i^{**}(\mathcal{A}_1;\cdots;\mathcal{A}_n:\mathcal{A}_{n+1})$. 
\item $i^{**}(\mathcal{A}_1;\cdots;\mathcal{A}_n:\mathcal{A}_{n+1}) = 0$ if and only if $\mathcal{A}_1,\dots,\mathcal{A}_{n+1}$ are freely independent. 
\item $\chi_\mathrm{orb}(\mathbf{X}_1,\dots,\mathbf{X}_{n+1}) \leq -i^{**}(W^*(\mathbf{X}_1);\dots;W^*(\mathbf{X}_n):W^*(\mathbf{X}_{n+1}))$. 
\end{itemize}  
Here $W^*(\mathcal{A}_i)$ and $W^*(\mathbf{X}_i)$ denote the von Neumann subalgebras generated by $\mathcal{A}_i$ and $\mathbf{X}_i$, respectively. An important question is whether or not $i^* = i^{**}$. It is also an interesting question whether or not $\mathcal{J}_{\sigma_0}^\mathrm{lib}$ and $\mathcal{J}_{\sigma_0,\infty}^\mathrm{lib}$ are independent of the choice of $\sigma_0$. 

\section{Unitary Brownian motions} 

Let $\Xi(N)$ and $U_N^{(i)}(t)$, $1\leq i \leq n$ be as in subsection 4.7, that is, $\Xi(N)$ is a countable family of deterministic $N\times N$ self-adjoint matrices and the $U_N^{(i)}(t)$ are independent, left-increment unitary Brownian motions on $\mathrm{U}(N)$. For the ease of notation, we number the elements of $\Xi(N)$ as $\xi_j(N)$ rather than $\xi_{ij}(N)$. In this section, we will explain how the proofs in \cite{Ueda:JOTP19} work well for the $U_N^{(i)}(t)$ together with $\Xi(N)$ and compare their consequences on the matrix liberation process $\Xi^\mathrm{lib}(N)$ with the corresponding results on the $U_N^{(i)}(t)$ together with $\Xi(N)$. 

\subsection{Malliavin derivatives of unitary Brownian motions}
We begin with the SDE representation of $U_N^{(k)}(t)$: Let $B_{\alpha\beta}^{(i)}(t)$, $1 \leq \alpha,\beta \leq N$, $1 \leq i \leq n$, be the $nN^2$ independent Brownian motions on the real line with natural filtration $\mathcal{F}_t$. Consider the system of SDEs in the $2nN^2$-dimensional Euclidean space $(M_N)^n$: 
\begin{equation} \label{Eq7.1}
\mathrm{d}X^{(i)}(t) = \frac{\sqrt{-1}}{\sqrt{N}}\sum_{1 \leq \alpha,\beta \leq N}C_{\alpha\beta}\,X^{(i)}(t)\,\mathrm{d}B^{(i)}_{\alpha\beta}(t) - \frac{1}{2}X^{(i)}(t)\,\mathrm{d}t \quad (1 \leq i \leq n), 
\end{equation} 
where $C_{\alpha\beta}$, $1 \leq \alpha,\beta \leq N$, form an orthonormal basis of the Euclidean space $M_N^{sa}$. This system of SDEs are linear, and thus each system of them admits a unique strong solution after fixing initial $X^{(i)}(0)$. The unitary Brownian motions $U_N^{(i)}(t)$, $1 \leq i \leq n$, are constructed as a unique strong solution $X^{(i)}(t)$ of the system \eqref{Eq7.1} under initial condition $X^{(i)}(0) = I$. 

\begin{lemma}\label{L7.1} Let $\mathrm{D}_s^{(k;\alpha,\beta)}$ be the Malliavin derivative along the Brownian motion $B_{\alpha\beta}^{(k)}$. Then 
\begin{align*}
\mathrm{D}_s^{(k;\alpha,\beta)}U_N^{(i)}(t) 
&= \delta_{k,i}\,\mathbf{1}_{[0,t]}(s)\Big(\sqrt{-1}\,U_N^{(k)}(t)U_N^{(k)}(s)^* \Big(\frac{1}{\sqrt{N}}\,C_{\alpha\beta}\Big) U_N^{(k)}(s)\Big), \\
\mathrm{D}_s^{(k;\alpha,\beta)}U_N^{(i)}(t)^* 
&= \delta_{k,i}\,\mathbf{1}_{[0,t]}(s)\Big(-\sqrt{-1}\,U_N^{(k)}(s)^*\Big(\frac{1}{\sqrt{N}}C_{\alpha\beta}\Big)U_N^{(k)}(s) U_N^{(k)}(t)^*\Big)
\end{align*}
for almost every $t \geq 0$. 
\end{lemma}
\begin{proof}
We also consider the system of SDEs
\begin{equation} \label{Eq7.2}
\mathrm{d}Y^{(i)}(t) = \frac{-\sqrt{-1}}{\sqrt{N}}\sum_{1 \leq \alpha,\beta \leq N} Y^{(i)}(t)\,C_{\alpha\beta}\,\mathrm{d}B^{(i)}_{\alpha\beta}(t) - \frac{1}{2}Y^{(i)}(t)\,\mathrm{d}t \quad (1 \leq i \leq n). 
\end{equation}
For a given $X \in M_N$, it is easy to see that $X^{(i)}(t) := U_N^{(i)}(t)X$ and $Y^{(i)}(t) := X U_N^{(i)}(t)^*$ satisfy the systems \eqref{Eq7.1}, \eqref{Eq7.2} of SDEs, respectively. Thus, the unique strong solutions of the system of SDEs \eqref{Eq7.1},\eqref{Eq7.2} with initial condition $X^{(i)}(0)=X$, $Y^{(i)}(0)=X$ must be $U_N^{(i)}(t) X$, $X U_N^{(i)}(t)^*$. Thus, $U^{(i)}_N(t) X$, $X U^{(i)}_N(t)^*$ are both linear in the variable $X$, and hence their gradients (or `Jacobian matrix') in $X$ become the linear transformations $L_{U_N^{(i)}(t)}$ and $R_{U_N^{(i)}(t)^*}$ on $M_N$, respectively, where $L_A X := AX$, $R_B X := XB$ for $A, B, X \in M_N$. By a standard fact on Malliavin derivatives for strong solutions of SDEs \cite[Theorem 2.2.1; Eq.(2.59)]{Nualart:Book06} it follows that
\begin{align*} 
\mathrm{D}_s^{(k;\alpha,\beta)} U_N^{(i)}(t) 
&= 
\delta_{k,i}\,\mathbf{1}_{[0,t]}(s)\,L_{U_N^{(k)}(t)} (L_{U_N^{(k)}(s)})^{-1} \Big(\frac{\sqrt{-1}}{\sqrt{N}} C_{\alpha\beta} U_N^{(k)}(s)\Big) \\
&= 
\delta_{k,i}\,\mathbf{1}_{[0,t]}(s)\Big(\sqrt{-1}\,U_N^{(k)}(t)U_N^{(k)}(s)^* \Big(\frac{1}{\sqrt{N}}\,C_{\alpha\beta}\Big) U_N^{(k)}(s)\Big), \\
\mathrm{D}_s^{(k;\alpha,\beta)} U_N^{(i)}(t)^* 
&= 
\delta_{k,i}\,\mathbf{1}_{[0,t]}(s)\,R_{U_N^{(k)}(t)^*} (R_{U_N^{(k)}(s)^*})^{-1} \Big(\frac{-\sqrt{-1}}{\sqrt{N}} U_N^{(k)}(s)^* C_{\alpha\beta}\Big) \\
&= 
\delta_{k,i}\,\mathbf{1}_{[0,t]}(s)\Big(-\sqrt{-1}\,U_N^{(k)}(s)^*\Big(\frac{1}{\sqrt{N}}C_{\alpha\beta}\Big)U_N^{(k)}(s) U_N^{(k)}(t)^*\Big). 
\end{align*}
Hence we are done.    
\end{proof}

By the linearity and the Leibniz rule of $\mathrm{D}_s^{(k;\alpha,\beta)}$ we have, for a monomial $W$ in $U_N^{(i)}(t), U_N^{(i)}(t)^*$ and $\xi_j(N)$, 
\begin{equation}\label{Eq7.3}
\begin{aligned} 
&\mathrm{D}_s^{(k;\alpha,\beta)}\mathrm{tr}_N(W)
=\sum_{\substack{ W = W_1 U_N^{(k)}(t) W_2 \\ s \leq t}} 
\mathrm{tr}_N\Big(W_1\Big(\sqrt{-1}\,U_N^{(k)}(t)U_N^{(k)}(s)^* \Big(\frac{1}{\sqrt{N}}\,C_{\alpha\beta}\Big) U_N^{(k)}(s)\Big)W_2\Big) \\
&\qquad\quad+ 
\sum_{\substack{ W = W_3 U_N^{(k)}(t)^* W_4 \\ s \leq t}} 
\mathrm{tr}_N\Big(W_3\Big(-\sqrt{-1}\,U_N^{(k)}(s)^*\Big(\frac{1}{\sqrt{N}}C_{\alpha\beta}\Big)U_N^{(k)}(s) U_N^{(k)}(t)^*\Big)W_4\Big).
\end{aligned}
\end{equation}
With these remarks it is a straightforward task to modify the proof of the large deviation upper bound for the matrix liberation process in \cite{Ueda:JOTP19} to the case of unitary Brownian motions with deterministic matrices. The consequence is as follows.  

\subsection{Non-commutative derivations}
We assume the norm constraint $\Vert \xi_j(N)\Vert_\infty \leq R$ for all $j \geq 1$, and moreover that $\Xi(N)$ has a limit distribution as $N \to \infty$. Thus we consider the universal $C^*$-algebras $C^*_R\langle x_\diamond\rangle \subset C^*_R\langle x_\diamond, u_\bullet(\,\cdot\,) \rangle \subset C^*_R\langle x_\diamond, u_\bullet(\,\cdot\,), v_\bullet(\,\cdot\,)\rangle$ generated by $x_j= x_j^*$, $j \geq 1$, and $u_i(t), v_i(t)$, $1 \leq i \leq n$, $t \geq 0$, with subject to $\Vert x_j\Vert_\infty \leq R$ and $u_i(t)^* u_i(t) = u_i(t)u_i(t)^* = v_i(t)^* v_i(t) = v_i(t)v_i(t)^* = u_i(0) = v_i(0) = 1$, $1 \leq i \leq n$, $t \geq 0$. Remark that the universal $*$-algebra $\mathbb{C}\langle x_\diamond, u_\bullet(\,\cdot\,) \rangle$ generated by the same indeterminates with the same algebraic constraints (and without the norm constraint) is naturally embedded into $C^*_R\langle  x_\diamond, u_\bullet(\,\cdot\,)\rangle$ as a norm-dense $*$-subalgebra. By formula \eqref{Eq7.3} we introduce derivations $\delta_s^{(k)} : \mathbb{C}\langle x_\diamond, u_\bullet(\,\cdot\,) \rangle \to \mathbb{C}\langle x_\diamond, u_\bullet(\,\cdot\,) \rangle\otimes_\mathrm{alg}\mathbb{C}\langle x_\diamond, u_\bullet(\,\cdot\,) \rangle$ determined by 
\begin{align*} 
\delta_s^{(k)} u_i(t)\ \,  
&:= 
\delta_{k,i}\mathbf{1}_{[0,t]}(s)\,\big(\sqrt{-1}\,u_k(t) u_k(s)^*\otimes u_k(s)\big), \\  
\delta_s^{(k)} u_i(t)^* 
&:= 
\delta_{k,i}\mathbf{1}_{[0,t]}(s)\big(-\sqrt{-1}\,u_k(s)^*\otimes u_k(s)u_k(t)^*\big), \\  
\delta_s^{(k)} x_j\ \  &:= 0. 
\end{align*}
(In fact, one can easily check $(u\delta_s^{(k)}u_k(t))\cdot u_k(t)^* - u_k(t)\cdot(u\delta_s^{(k)}u_k(t)^*) = 0$ for example, and hence the above definition works well.) With the linear mapping $\theta : a\otimes b \mapsto ba$ we define cyclic derivatives $\mathfrak{D}_s^{(k)} := \theta\circ \delta_s^{(k)} : \mathbb{C}\langle x_\diamond, u_\bullet(\,\cdot\,) \rangle \to \mathbb{C}\langle x_\diamond, u_\bullet(\,\cdot\,) \rangle$. If we denote by $P(\xi_\diamond(N),U_\bullet^{(i)}(\,\cdot\,))$ the specialization of a given $P \in \mathbb{C}\langle x_\diamond, u_\bullet(\,\cdot\,) \rangle$ with  $x_j = \xi_j(N)$ and $u_i(t) = U_N^{(i)}(t)$, then formula \eqref{Eq7.3} admits a `compact' expression
\[
\mathrm{D}_s^{(k;\alpha,\beta)}\mathrm{tr}_N(P(\xi_\diamond(N),U_N^{(\bullet)}(\,\cdot\,))) = 
\mathrm{tr}_N\Big((\mathfrak{D}_s^{(k)}P)(\xi_\diamond(N),U_N^{(\bullet)}(\,\cdot\,))\,\Big(\frac{1}{\sqrt{N}}\,C_{\alpha\beta}\Big)\Big)
\]
for any $P \in \mathbb{C}\langle x_\diamond, v_\bullet(\,\cdot\,) \rangle $. Thus, the Clark--Ocone formula (see e.g., \cite[Proposition 6.11]{Hu:Book16} for any dimension and \cite[subsection 1.3.4]{Nualart:Book06} for $1$ dimension) shows that 
\begin{align*}
&\mathbb{E}[\mathrm{tr}_N(P(\xi_\diamond(N),U_N^{(\bullet)}(\,\cdot\,)))\mid\mathcal{F}_t] 
= 
\mathbb{E}[\mathrm{tr}_N(P(\xi_\diamond(N),U_N^{(\bullet)}(\,\cdot\,)))] \\ 
&\phantom{aaaaa}+ \sum_{k=1}^n \sum_{\alpha,\beta=1}^N \int_0^t \mathbb{E}\Big[\mathrm{tr}_N\Big((\mathfrak{D}_s^{(k)}P)(\xi_\diamond(N),U_N^{(\bullet)}(\,\cdot\,))\,\Big(\frac{1}{\sqrt{N}}\,C_{\alpha\beta}\Big)\Big)\mid\mathcal{F}_s\Big]\,\mathrm{d}B_{\alpha\beta}^{(k)}(s).
\end{align*}

\subsection{Continuous tracial states}
A tracial state $\varphi$ on $C^*_R\langle x_\diamond, u_\bullet(\,\cdot\,)\rangle$ (or $C^*_R\langle x_\diamond, u_\bullet(\,\cdot\,),v_\bullet(\,\cdot\,)\rangle$) is said to be continuous if $t \mapsto u_i^\varphi(t) := \pi_\varphi(u_i(t))$ is strongly continuous (resp.\, $t \mapsto \pi_\varphi(u_i(t)), \pi_\varphi(v_i(t))$ are strongly continuous) for every $1 \leq i \leq n$, where $\pi_\varphi : C^*_R\langle x_\diamond,u_\bullet(\,\cdot\,) \rangle \curvearrowright \mathcal{H}_\varphi$ (resp.\, $\pi_\varphi : C^*_R\langle x_\diamond,u_\bullet(\,\cdot\,),v_\bullet(\,\cdot\,) \rangle \curvearrowright \mathcal{H}_\varphi$) is the GNS representation associated with $\varphi$. We then denote by $TS^c(C^*_R\langle x_\diamond,u_\bullet(\,\cdot\,) \rangle)$ and $TS^c(C^*_R\langle x_\diamond,u_\bullet(\,\cdot\,), v_\bullet(\,\cdot\,) \rangle)$ all the continuous tracial states on $C^*_R\langle x_\diamond, u_\bullet(\,\cdot\,) \rangle$ and $C^*_R\langle x_\diamond,u_\bullet(\,\cdot\,), v_\bullet(\,\cdot\,) \rangle$, respectively. Set $x_j(t) := x_j$, $t \geq 0$, for each $j$ for the ease of notation below. Then, the same facts as \cite[Lemmas 2.1,2.2]{Ueda:JOTP19} holds and the metric $d$ on $TS^c(C^*_R\langle x_\diamond, u_\bullet(\,\cdot\,)\rangle)$ can be defined in the exactly same manner as \eqref{Eq1.1} by considering words in $x_j(t)$ and $u_i(t), u_i(t)^*$ in place of $x_{i_1 j_1}(t_1)\cdots x_{i_m j_m}(t_m)$ for $w(t_1,\dots,t_m)$.  We remark that $\tau((x_{ij}(s)-x_{ij}(t))^2)$ in \cite[Lemma 2.2(2)]{Ueda:JOTP19} should be replaced with $\varphi((u_i(s)-u_i(t))^*(u_i(s) - u_i(t))) = 2(1 - \mathrm{Re}\,\varphi(u_i(s)^*u_i(t)))$ in this context. 

\subsection{Rate function}
By universality, we have the $*$-homomorphism 
\[
\Pi^s : C^*_R\langle x_\diamond,u_\bullet(\,\cdot\,)\rangle \to C^*_R\langle x_\diamond,u_\bullet(\,\cdot\,),v_\bullet(\,\cdot\,)\rangle
\] 
for each $s \geq 0$, which sends each $u_i(t)$ to $u_i^s(t)$ and keeping each $x_j$ as it is, where 
\[
u^s_i(t) := v_i((t-s)\vee0)u_i(s\wedge t), \quad 1 \leq i \leq n, t \geq 0. 
\]
We can extend each $\varphi \in TS^c(C^*_R\langle x_\diamond,u_\bullet(\,\cdot\,)\rangle)$ to a unique $\tilde{\varphi} \in TS^c(C^*_R\langle x_\diamond,u_\bullet(\,\cdot\,),v_\bullet(\,\cdot\,)\rangle)$ in such a way that the $v_i(t)$ are freely independent of $C^*_R\langle x_\diamond,u_\bullet(\,\cdot\,)\rangle$ and form a freely independent family of left-multiplicative free unitary Brownian motions under $\tilde{\varphi}$. For each $\varphi \in TS^c(C^*_R\langle x_\diamond,u_\bullet(\,\cdot\,)\rangle)$ we define $\varphi^s := \tilde{\varphi}\circ\Pi^s \in TS^c(C^*_R\langle x_\diamond,u_\bullet(\,\cdot\,)\rangle)$, $s \geq 0$, and also write   
\[
(\mathcal{N}(\varphi) \subset \mathcal{M}(\varphi)) := \big(\pi_{\tilde{\varphi}}(C^*_R\langle x_\diamond,u_\bullet(\,\cdot\,) \rangle)'' \subset \pi_{\tilde{\varphi}}(C^*_R\langle x_\diamond,u_\bullet(\,\cdot\,),v_\bullet(\,\cdot\,)\rangle)''\big)
\] 
on $\mathcal{H}_{\tilde{\varphi}}$, where $\pi_{\tilde{\varphi}} : C_R^*\langle x_\diamond,u_\bullet(\,\cdot\,),v_\bullet(\,\cdot\,)\rangle \curvearrowright \mathcal{H}_ {\tilde{\varphi}}$ is the GNS representation associated with ${\tilde{\varphi}}$. We fix a distribution of the $x_j$, say $\sigma_0 \in TS(C^*_R\langle x_\diamond \rangle)$. Let $\sigma_0^\mathrm{frBM}$ be $\varphi^0$ with $\varphi \in TS^c(C^*_R\langle x_\diamond,u_\bullet\rangle)$ such that the restriction of $\varphi$ to $C^*_R\langle x_\diamond\rangle$ is $\sigma_0$. Such a continuous tracial state $\varphi^0$ is uniquely determined; in fact, it is the joint distribution of the $x_j$'s and the $v_i(t)$'s such that the $v_i(t)$ form a freely independent family of left-multiplicative free unitary Brownian motions and are freely independent of the $x_j$'s, and moreover that the distribution of the $x_j$'s is $\sigma_0$.  For any $\varphi \in TS^c(C^*_R\langle x_\diamond,u_\bullet(\,\cdot\,)\rangle)$, $P = P^* \in \mathbb{C}\langle x_\diamond,u_\bullet(\,\cdot\,)\rangle$ and $t \in [0,\infty]$ we define 
\[
I^\mathrm{uBM}_{\sigma_0,t}(\varphi,P) := \varphi^t(P) - \sigma_0^\mathrm{frBM}(P) - \frac{1}{2}\sum_{k=1}^n \int_0^t \Vert E_{\mathcal{N}(\tau)}(\pi_{\tilde{\varphi}}(\Pi^s(\mathfrak{D}_s^{(k)}P)))\Vert_{\tilde{\varphi},2}^2\,ds
\]
with regarding $\varphi$ as $\varphi^\infty$. Then we introduce two functionals $I^\mathrm{uBM}_{\sigma_0}, I^\mathrm{uBM}_{\sigma_0,\infty} : TS^c(C^*_R\langle x_\diamond,u_\bullet(\,\cdot\,)\rangle) \to [0,+\infty]$ defined by  
\[
I^\mathrm{uBM}_{\sigma_0}(\varphi) := \sup_{\substack{ P = P^* \in \mathbb{C}\langle x_\diamond,u_\bullet(\,\cdot\,)\rangle \\ t > 0}} I^\mathrm{uBM}_{\sigma_0,t}(\varphi,P), \quad 
I^\mathrm{uBM}_{\sigma_0,\infty}(\varphi) := \sup_{P = P^* \in \mathbb{C}\langle x_\diamond,u_\bullet(\,\cdot\,)\rangle} I^\mathrm{uBM}_{\sigma_0,\infty}(\varphi,P) 
\]
for $\varphi \in TS^c(C^*_R\langle x_\diamond, u_\bullet(\,\cdot\,) \rangle)$. 

\subsection{Consequences} 
Here is the main consequence of this section. 

\begin{theorem}\label{T7.2} Assume that $\sigma_0 \in TS(C^*_R\langle x_\diamond\rangle)$ is the limit distribution of $\Xi(N)$ as $N\to\infty$. We denote by $P \in C^*_R\langle x_\diamond,u_\bullet(\,\cdot,)\rangle \mapsto P(\xi_\diamond(N),U_N^{(\bullet)}(\,\cdot\,)) \in M_N$ the $*$-homomorphism sending $u_i(t)$ and $x_j$ to $U_N^{(i)}(t)$ and $\xi_j(N)$, respectively. Let $\varphi_{\Xi(N)}^\mathrm{uBM} \in TS^c(C^*_R\langle x_\diamond,u_\bullet(\,\cdot\,)\rangle)$ be the random tracial state sending $P \in C^*_R\langle x_\diamond,u_\bullet(\,\cdot,)\rangle$ to $\mathrm{tr}_N(P(\xi_\diamond(N),U_N^{(\bullet)}(\,\cdot\,)))$. Then we have the following large deviation upper bound: 
\[
\varlimsup_{N\to\infty} \frac{1}{N^2}\log\mathbb{P}(\varphi_{\Xi(N)} ^\mathrm{uBM} \in \Lambda) \leq -\inf\{ I^\mathrm{uBM}_{\sigma_0}(\varphi) \mid \varphi \in \Lambda\}
\]
for every closed $\Lambda \subset TS^c(C^*_R\langle x_\diamond,u_\bullet(\,\cdot\,)\rangle)$.  Moreover, both $I^\mathrm{uBM}_{\sigma_0} \geq I^\mathrm{uBM}_{\sigma_0,\infty}$ are good rate functions and admit the same unique minimizer $\sigma_0^\mathrm{frBM}$.  
\end{theorem} 

Proving that the rate functions are good along the line of the proof of \cite[Proposition 5.6]{Ueda:JOTP19} needs the formula 
\begin{align*}
&E_{\mathcal{N}(\varphi)}(\Pi^s(\mathfrak{D}_s^{(k)}((u_i(t_1)-u_i(t_2))^*(u_i(t_1)-u_i(t_2))) \\
&\qquad= 
\delta_{k,i}\sqrt{-1} e^{-\frac{1}{2}(t_1\vee t_2 -s)}\mathbf{1}_{(t_1\wedge t_2,t_1\vee t_2]}(s)(u_k(t_1\wedge t_2)u_k(s)^* - u_k(s)u_k(t_1\wedge t_2)^*). 
\end{align*}

\medskip
Similarly to \cite[Corollary 5.9]{Ueda:JOTP19} the standard Borel--Cantelli argument shows the next corollary. 

\begin{corollary}\label{C7.3} Keep the same setting as in Theorem \ref{T7.2}. Let $\sigma_0^{\mathrm{frBM}} \in TS^c(C^*_R\langle x_\diamond,u_\bullet(\,\cdot\,)\rangle)$ be constructed in such a way that the distribution of the $x_j$ is $\sigma_0$ under $\sigma_0^{\mathrm{frBM}}$ and also that the $u_i(t)$ form a freely independent family of left-multiplicative free unitary Brownian motions and are freely independent of the $x_j$ under $\sigma_0^{\mathrm{frBM}}$. Then $d(\varphi_{\Xi(N)} ^\mathrm{uBM},\sigma_0^{\mathrm{frBM}}) \to 0$ almost surely as $N\to\infty$. 
\end{corollary} 

This is a precise statement about the almost sure convergence \emph{as continuous process} for an independent family of unitary Brownian motions together with deterministic matrices, and seems to have been missing so far, even though the almost sure strong convergence for its time marginals was already established by Collins, Dahlqvist and Kemp \cite{CollinsDahlqvistKemp:PTRF18}. 

\subsection{Haar-distributed unitary random matrices} 
As in section 4, using Lemma \ref{L2.1} we can derive a large deviation upper bound for an independent family of $N\times N$ Haar-distributed unitary random matrices $U^{(i)}_N$, $1 \leq i \leq n$, with deterministic matrices $\Xi(N)$ from Theorem \ref{T7.2}. The resulting rate function is given as in Lemma \ref{L4.1}. Let $C^*_R\langle x_\diamond, u_\bullet\rangle$ be the universal $C^*$-algebra generated by $x_j$, $j \geq 1$, and $u_i$, $1 \leq i \leq n$, with subject to $\Vert x_j \Vert_\infty \leq R$ and $u_i^* u_i = u_i u_i^* = 1$. We denote by $P \in C^*_R\langle x_\diamond, u_\bullet\rangle \mapsto P(\xi_\diamond(N),U_N^{(\bullet)}) \in M_N$ the $*$-homomorphism sending $x_j$ and $u_i$ to $\xi_j(N)$ and $U_N^{(i)}$, respectively. Then we have the random tracial state $\varphi_{\Xi(N)}^\mathrm{uHaar} \in TS(C^*_R\langle x_\diamond, u_\bullet\rangle) \to \mathbb{C}$ defined by $\varphi_{\Xi(N)}^\mathrm{uHaar}(P) := \mathrm{tr}_N(P(\xi_\diamond(N),U_N^{(\bullet)}))$ for $P \in  
C^*_R\langle x_\diamond, u_\bullet\rangle$.  
Namely, let $\pi_T : C^*_R\langle x_\diamond, u_\bullet\rangle \to C^*_R\langle x_\diamond, u_\bullet(\,\cdot\,)\rangle$ be the $*$-homomorphism sending $x_j$ and $u_i$ to $x_j$ and $u_i(T)$, respectively, as before. Then we have the large deviation upper bound for the probability measures $\mathbb{P}(\varphi_{\Xi(N)}^\mathrm{uHaar} \in \,\cdot\,)$ with speed $N^2$ and the rate function
\begin{align*}
&\psi \in TS(C^*_R\langle x_\diamond, u_\bullet\rangle) \\
&\quad 
\mapsto \lim_{\substack{m\to\infty \\ \delta \searrow 0}}\varlimsup_{T\to\infty}\inf\{I_{\sigma_0}^\mathrm{uBM}(\varphi) \mid \varphi \in TS^c(C^*_R\langle u_\bullet(\,\cdot\,),x_\diamond\rangle), \pi_T^*(\varphi) \in \mathcal{O}_{m,\delta}(\psi) \} \in [0,+\infty],
\end{align*} 
where as before the infimum over the empty set is taken as $+\infty$ and $\mathcal{O}_{m,\delta}(\psi)$ is the open neighborhood consisting of all the $\chi \in TS(C^*_R\langle x_\diamond, u_\bullet\rangle)$ such that $|\chi(w) - \psi(w)| < \delta$ for all words $w$ in $x_j,u_i, u_i^*$ ($j \leq m$, $1 \leq i \leq n$) of length not greater than $m$. 

We remark that Cabanal Duvillard and Guionnet \cite[Corollary 4.2]{CabanalDuvillardGuionnet:AnnProbab01} have also obtained a large deviation upper bound for the $U_N^{(i)}$ with seemingly different rate function based on self-adjoint matrix Brownian motions. 

\subsection{Relation to the matrix liberation process} 
We will compare Theorem \ref{T7.2} with \cite[Theorem 5.8]{Ueda:JOTP19}. To this end, we re-number $\xi_j(N)$ and $x_j$ as $\xi_{ij}(N)$ and $x_{ij}$, respectively. Let $\pi_\mathrm{lib} : C^*_R\langle x_{\bullet\diamond}(\,\cdot\,)\rangle \to C^*_R\langle x_{\bullet\diamond},u_\bullet(\,\cdot\,)\rangle$ be the $*$-homomorphism sending $x_{ij}(t)$ to $u_i(t)x_{ij}u_i(t)^*$. This induces a continuous map $\pi_\mathrm{lib}^* : TS^c(C^*_R\langle x_{\bullet\diamond},u_\bullet(\,\cdot\,)\rangle) \to TS^c(C^*_R\langle x_{\bullet\diamond}(\,\cdot\,)\rangle)$ defined by $\pi_\mathrm{lib}^*(\varphi) := \varphi\circ\pi_\mathrm{lib}$.  We observe that $\pi_\mathrm{lib}^*(\varphi_{\Xi(N)} ^\mathrm{uBM}) = \tau_{\Xi^\mathrm{lib}(N)}$. Therefore, the contraction principle in large deviation theory implies the large deviation upper bound for $\mathbb{P}(\tau_{\Xi^\mathrm{lib}(N)} \in \,\cdot\,)$ in the same scale with the good rate function: 
\begin{equation}
\begin{aligned} 
&\tau \in TS^c(C^*_R\langle x_{\bullet\diamond}(\,\cdot\,)\rangle) \\
&\quad \mapsto I_{\sigma_0}^\mathrm{ulib}(\tau) :=  \inf\{ I^\mathrm{uBM}_{\sigma_0}(\varphi) \mid \varphi \in TS^c(C^*_R\langle u_\bullet(\,\cdot\,),x_\diamond\rangle), \pi_\mathrm{lib}^*(\varphi) = \tau \} \in [0,+\infty], 
\end{aligned}
\end{equation}
where the infimum over the empty set is taken as $+\infty$.   
Therefore, we have two large deviation upper bounds with (seemingly different) rate functions for $\mathbb{P}(\tau_{\Xi^\mathrm{lib}(N)} \in \,\cdot\,)$. 

Let $\tau \in TS^c(C^*_R\langle x_{\bullet\diamond}(\,\cdot\,)\rangle)$ be given. Consider an arbitrary $\varphi \in TS^c(C^*_R\langle x_{\bullet\diamond},u_\bullet(\,\cdot\,)\rangle)$ with $\pi_\mathrm{lib}^*(\varphi) = \tau$. It is not difficult to show that 
\[
\varphi^s(\pi_\mathrm{lib}(P)) = \tau^s(P), \quad 
E_{\mathcal{N}(\varphi)}(\Pi^s(\mathfrak{D}_s^{(k)} \pi_\mathrm{lib}(P))) = 
E_{\mathcal{N}(\tau)}(\Pi^s(\mathfrak{D}_s^{(k)}P)) 
\]
for every $P \in \mathbb{C}\langle x_{\bullet\diamond}(\,\cdot\,)\rangle$ and every $s \geq 0$. Therefore, $I_{\sigma_0,t}^\mathrm{lib}(\tau,P) = I_{\sigma_0,t}(\varphi,\pi_\mathrm{lib}(P))$ for every $P \in \mathbb{C}\langle x_{\bullet\diamond}(\,\cdot\,)\rangle$ and every $t \geq 0$, and hence 
\begin{equation}
I_{\sigma_0}^\mathrm{lib}(\tau) \leq I_{\sigma_0}^\mathrm{ulib}(\tau), \quad 
I_{\sigma_0,\infty}^\mathrm{lib}(\tau) \leq I_{\sigma_0,\infty}^\mathrm{ulib}(\tau), 
\end{equation}   
where $I_{\sigma_0,\infty}^\mathrm{ulib}(\tau) :=  \inf\{ I_{\sigma_0,\infty}^\mathrm{uBM}(\varphi) \mid \varphi \in TS^c(C^*_R\langle u_\bullet(\,\cdot\,).x_\diamond\rangle), \pi_\mathrm{lib}^*(\varphi) = \tau \}$. Therefore, the current approach using unitary Brownian motions directly gives an improved large deviation upper bound for the matrix liberation process, though the description of the resulting rate function is `indirect'. Remark that the above inequalities between two kinds of rate functions guarantee that $I_{\sigma_0}^\mathrm{ulib} \geq I_{\sigma_0,\infty}^\mathrm{ulib}$ also have a unique minimizer, which is given by $\sigma_0^\mathrm{lib}$. Remark that this fact on the rate functions $I_{\sigma_0}^\mathrm{ulib} \geq I_{\sigma_0,\infty}^\mathrm{ulib}$ holds even when $\sigma_0$ does not fall into $TS_\mathrm{fda}(C^*\langle x_{\bullet\diamond}\rangle)$. 

\section{Conditional expectations of liberation cyclic derivatives}

We will give a technical result on liberation cyclic derivatives $\mathfrak{D}_s^{(k)}$, $1 \leq k \leq n$, for future work. The most non-trivial component of the rate functions $I^\mathrm{lib}_{\sigma_0}, I^\mathrm{lib}_{\sigma_0,\infty}$ is $E_{\mathcal{N}(\tau)}(\pi_{\tilde{\tau}}(\Pi^s(\mathfrak{D}^{(k)}_s P)))$, which will be described in terms of free cumulants when $P$ is a monomial. In what follows, we use the notations in section 4.  

\medskip
We first introduce some terminology: Let $(\mathcal{A},\varphi)$ be a non-commutative probability space, and $a_1,\dots,a_n \in \mathcal{A}$ be arbitrarily chosen. For a `block' $V = (i_1 < \cdots < i_s)$ of $[n] = \{1,\dots,n\}$, we define $\mathrm{id}(V)[a_1,\dots,a_n] := a_{i_1}\cdots a_{i_s}$ (i.e., the  word obtained by arranging $a_{i_1},\dots,a_{i_s}$ in order). For a partition $\pi = \{V_1,\dots,V_m\}$ of $[n]$, we define 
\[
C(\varphi;\pi)[a_1,\dots,a_n] := \sum_{k=1}^m \Big(\prod_{\substack{1 \leq \ell \leq m \\ \ell \neq k}} \varphi(V_\ell)[a_1,\dots,a_n]\Big) \mathrm{id}(V_k)[a_1,\dots,a_n], 
\]
where $\varphi(V_\ell)[a_1,\dots,a_n]$ is defined as in \cite[Lecture 11]{NicaSpeicher:Book}; namely, we have $\varphi(V_\ell)[a_1,\dots,a_n] = \varphi(\mathrm{id}(V_\ell)[a_1,\dots,a_n])$.  
  
\begin{proposition}\label{P8.1} Write
\[
w_\ell := v_{i_{\ell-1}}((t_{\ell-1}-s)_+)^* v_{i_\ell}((t_\ell-s)_+), \quad 1 \leq \ell \leq n, 
\]
with $i_0 := i_n$ and $(t-s)_+ := 0 \vee (t-s)$. Then, we have 
\begin{align*} 
&E_{\mathcal{N}(\tau)}(\pi_{\tilde{\tau}}(\Pi^s(\mathfrak{D}^{(k)}_s x_{i_1 j_1}(t_1)\cdots x_{i_n j_n}(t_n)))) \\
&= 
\sum_{\pi \in NC(n)} \kappa_\pi[w_1,\dots,w_n]\,\pi_{\tilde{\tau}}(\mathfrak{D}_s^{(k)}C(\tau;K(\pi))[x_{i_1 j_1}(s\wedge t_1),\dots,x_{i_n j_n}(s\wedge t_n)]),      
\end{align*}
where $NC(n)$ denotes the non-crossing partitions of $[n]$, $\kappa_\pi$ the free cumulant associated with $\pi$, and $K : NC(n) \to NC(n)$ the Kreweras complementation map; see \cite[Lecture 11]{NicaSpeicher:Book}. 
\end{proposition}
\begin{proof} 
Write $P = x_{i_1 j_1}(t_1)\cdots x_{i_n j_n}(t_n)$ for simplicity. Let $y \in C^*_R\langle x_{\bullet\diamond}(\,\cdot\,)\rangle$ be arbitrarily chosen. Then we compute 
\[
\tilde{\tau}(E_{\mathcal{N}(\tau)}(\pi_{\tilde{\tau}}(\Pi^s(\mathfrak{D}^{(k)}_s P)))\pi_{\tilde{\tau}}(y)) = 
\tilde{\tau}(\Pi^s(\mathfrak{D}^{(k)}_s P)y), 
\]
where we use the same symbol $\tilde{\tau}$ as a different meaning on each side; see subsection 4.D. By a direct computation using the trace property, we have 
\begin{align*} 
&\tilde{\tau}(\Pi^s(\mathfrak{D}^{(k)}_s P)y) 
= 
\sum_{\substack{i_\ell = k \\ s \leq t_\ell}} 
\tilde{\tau}([w_{\ell+1} x_{i_{\ell+1} j_{\ell+1}}(s\wedge t_{\ell+1}) w_{\ell+1} \cdots x_{i_{\ell-1} j_{\ell-1}}(s\wedge t_{\ell-1})w_\ell, x_{i_\ell j_\ell}(s\wedge t_\ell)]y) \\
&\qquad= 
\sum_{\substack{i_\ell = k \\ s \leq t_\ell}} 
\tilde{\tau}(w_1 x_{i_1 j_1}(s\wedge t_1) \cdots w_\ell [x_{i_\ell j_\ell}(s\wedge t_\ell),y] w_{\ell+1} x_{i_{\ell+1} j_{\ell+1}}(s\wedge t_{\ell+1}) \cdots w_n x_{i_n j_n}(s\wedge t_n)), 
\end{align*} 
each of whose terms is the $\tilde{\tau}$-value of the monomial obtained from $\Pi^s(P)$ by replacing $x_{i_\ell j_\ell}(s\wedge t_\ell)$ with $[x_{i_\ell j_\ell}(s\wedge t_\ell),y]$. By \cite[Theorem 14.4]{NicaSpeicher:Book} we obtain that
\begin{align*} 
&\sum_{\substack{i_\ell = k \\ s \leq t_\ell}} \tilde{\tau}(w_1 x_{i_1 j_1}(s\wedge t_1) \cdots w_\ell [x_{i_\ell j_\ell}(s\wedge t_\ell),y] w_{\ell+1} x_{i_{\ell+1} j_{\ell+1}}(s\wedge t_{\ell+1}) \cdots w_n x_{i_n j_n}(s\wedge t_n)) \\
&= 
\sum_{\substack{i_\ell = k \\ s \leq t_\ell}} \sum_{\pi \in NC(n)} \kappa_\pi[w_1,\dots,w_n]\,\tilde{\tau}_{K(\pi)}[x_{i_1 j_1}(s\wedge t_1),\dots,[x_{i_\ell j_\ell}(s\wedge t_\ell),y],\dots,x_{i_n j_n}(s\wedge t_n)] \\
&= 
\sum_{\pi \in NC(n)} \kappa_\pi[w_1,\dots,w_n] \Big(\sum_{\substack{i_\ell = k \\ s \leq t_\ell}} \tau_{K(\pi)}[x_{i_1 j_1}(s\wedge t_1),\dots,[x_{i_\ell j_\ell}(s\wedge t_\ell),y],\dots,x_{i_n j_n}(s\wedge t_n)]\Big) 
\end{align*}
When $K(\pi) = \{V_1,\dots,V_m\}$ with $\ell \in V_p$ ($1 \leq p \leq m$), we have 
\begin{align*} 
&\sum_{\substack{i_\ell = k \\ s \leq t_\ell}} \tau(w_1 x_{i_1 j_1}(s\wedge t_1) \cdots w_l [x_{i_\ell j_\ell}(s\wedge t_\ell),y] w_{\ell+1} x_{i_{\ell+1} j_{\ell+1}}(s\wedge t_{\ell+1}) \cdots w_n x_{i_n j_n}(s\wedge t_n)) \\
&= 
\sum_{\substack{i_\ell = k \\ s \leq t_\ell}} \Big(\prod_{\substack{1 \leq q \leq m \\ q \neq p}} \tau(V_q)[x_{i_1 j_1}(s\wedge t_1),\dots,[x_{i_\ell j_\ell}(s\wedge t_\ell),y],\dots,x_{i_n j_n}(s\wedge t_n)]\Big) \\
&\qquad\qquad\qquad \times \tau(V_p)[x_{i_1 j_1}(s\wedge t_1),\dots,[x_{i_\ell j_\ell}(s\wedge t_\ell),y],\dots,x_{i_n j_n}(s\wedge t_n)] \\
&= 
\sum_{\substack{i_\ell = k \\ s \leq t_\ell}} \Big(\prod_{\substack{1 \leq q \leq m \\ q \neq p}} \tau(V_q)[x_{i_1 j_1}(s\wedge t_1),\dots,x_{i_\ell j_\ell}(s\wedge t_\ell),\dots,x_{i_n j_n}(s\wedge t_n)]\Big) \\
&\qquad\qquad\qquad \times \tau(V_p)[x_{i_1 j_1}(s\wedge t_1),\dots,[x_{i_\ell j_\ell}(s\wedge t_\ell),y],\dots,x_{i_n j_n}(s\wedge t_n)]
\end{align*}  
If $V_p = (s_1 < \cdots < s_f)$ with $s_g = \ell$, then 
\begin{align*} 
&\tau(V_p)[x_{i_1 j_1}(s\wedge t_1),\dots,[x_{i_\ell j_\ell}(s\wedge t_\ell),y],\dots,x_{i_n j_n}(s\wedge t_n)] \\
&\qquad=
\tau([x_{i_{s_{g+1}} j_{s_{g+1}}}(s\wedge t_{s_{g+1}})\cdots x_{i_{s_{g-1}} j_{s_{g-1}}}(s\wedge t_{s_{g-1}}), x_{i_\ell j_\ell}(s\wedge t_\ell)]y), 
\end{align*}
which together with the definition of $\mathfrak{D}_s^{(k)}$ implies that 
\begin{align*} 
&\sum_{\substack{i_\ell = k \\ s \leq t_\ell}} \Big(\prod_{\substack{1 \leq q \leq m \\ q \neq p}} \tau(V_q)[x_{i_1 j_1}(s\wedge t_1),\dots,x_{i_\ell j_\ell}(s\wedge t_\ell),\dots,x_{i_n j_n}(s\wedge t_n)]\Big) \\
&\qquad\qquad \times \tau([x_{i_{s_{g+1}} j_{s_{g+1}}}(s\wedge t_{s_{g+1}})\cdots x_{i_{s_{g-1}} j_{s_{g-1}}}(s\wedge t_{s_{g-1}}), x_{i_\ell j_\ell}(s\wedge t_\ell)]y) \\
&\qquad= 
\tilde{\tau}((\mathfrak{D}_s^{(k)}C(\tau;K(\pi))[x_{i_1 j_1}(s\wedge t_1),\dots,x_{i_\ell j_\ell}(s\wedge t_\ell),\dots,x_{i_n j_n}(s\wedge t_n)])y) \\
&\qquad= 
\tilde{\tau}(\pi_{\tilde{\tau}}(\mathfrak{D}_s^{(k)}C(\tau;K(\pi))[x_{i_1 j_1}(s\wedge t_1),\dots,x_{i_\ell j_\ell}(s\wedge t_\ell),\dots,x_{i_n j_n}(s\wedge t_n)])\pi_{\tilde{\tau}}(y)).   
\end{align*}  
Hence we conclude that 
\begin{align*}
&\tilde{\tau}(E_{\mathcal{N}(\tau)}(\pi_{\tilde{\tau}}(\Pi^s(\mathfrak{D}^{(k)}_s P)))\pi_{\tilde{\tau}}(y)) \\
&= 
\sum_{\pi \in NC(n)}\kappa_\pi[w_1,\dots,w_n] 
\tilde{\tau}(\pi_{\tilde{\tau}}(\mathfrak{D}_s^{(k)}C(\tau;K(\pi))[x_{i_1 j_1}(s\wedge t_1),\dots,x_{i_n j_n}(s\wedge t_n)])\pi_{\tilde{\tau}}(y)). 
\end{align*}
Hence we are done.  
\end{proof} 

It is interesting to compute $\kappa_\pi[w_1,\dots,w_n]$ in the above explicitly. 

\appendix
\section{Universal free products of unital $C^*$-algebras} 

The concept of universal free products in the category of unital $C^*$-algebras has been studied in detail by several hands, including Blackadar \cite{Blackadar:IUMJ78}, Pedersen \cite{Pedersen:JFA99} and others. However, almost all existing works deal with only universal free products of \emph{two} unital $C^*$-algebras. We have used universal free products of \emph{uncountably many} unital $C^*$-algebras crucially (even in \cite{Ueda:JOTP19} without any references). Hence, we will collect a few facts on universal free products of \emph{arbitrary number} of unital $C^*$-algebras with explicit explanations for the reader's convenience. However, we do not claim any credit to the materials in this appendix, because they all seem to be known among specialists. 

\medskip
Let $\mathcal{A}_i$, $i \in I$, be unital $C^*$-algebras. Consider their universal free product $\bigstar_{i\in I}\mathcal{A}_i$ with canonical unital $*$-homomorphisms $\lambda_i : \mathcal{A}_i \to \bigstar_{i\in I}\mathcal{A}_i$, $i\in I$, which is characterized by the universality asserting that for any family $\pi_i : \mathcal{A}_i \to \mathcal{B}$ of unital $*$-homomorphisms into a common unital $C^*$-algebra, then there exists a unital $*$-homomorphism $\pi : \bigstar_{i\in I}\mathcal{A}_i \to \mathcal{B}$ such that $\pi\circ\lambda_i = \pi_i$ for all $i \in I$. Note that the injectivity of each $\lambda_i$ was established in \cite[Theorem 3.1]{Blackadar:IUMJ78} (or \cite[Theorem 4.2]{Pedersen:JFA99}).

\begin{lemma}\label{LA1} For any disjoint decomposition $I = \bigsqcup_{j \in J} I_j$ of $I$ into non-empty subsets, we consider the universal free product $C^*$-algebras $\bigstar_{i\in I_j}\mathcal{A}_i$, $j \in J$. Then $\bigstar_{i\in I}\mathcal{A}_i \cong \bigstar_{j\in J}(\bigstar_{i\in I_j}\mathcal{A}_i)$ naturally, that is, each $\lambda_i(a)$ with $a \in \mathcal{A}_i$ is sent to the corresponding element in the $j$th free product component $\bigstar_{i\in I_j}\mathcal{A}_i$ on the right-hand side when $i \in I_j$.    
\end{lemma} 
\begin{proof} 
This follows from the universality of the involved universal free product $C^*$-algebras. 
\end{proof} 

\begin{lemma}\label{LA2} For each finite subset $F \Subset I$, we consider the universal free product $C^*$-algebra $\mathfrak{A}_F := \bigstar_{i\in F}\mathcal{A}_i$ with setting $\mathfrak{A}_\emptyset := \mathbb{C}1$.  Then the following hold true: 
\begin{itemize} 
\item[(1)] If $F_1 \subset F_2$, then the canonical unital $*$-homomorphism $\mathfrak{A}_{F_1} \to \mathfrak{A}_{F_1} \bigstar \mathfrak{A}_{F_2\setminus F_2} = \mathfrak{A}_{F_2}$ via Lemma \ref{LA1} is injective. 
\item[(2)] $\bigstar_{i\in I}\mathcal{A}_i \cong \varinjlim_F \mathfrak{A}_F$ naturally (see e.g.\ \cite[Proposition 11.4.1(i)]{KadisonRingrose:Book2} for the latter), that is, the isomorphism sends each $\lambda_i(a)$ with $a \in \mathcal{A}_i$ to the corresponding one in $\mathfrak{A}_F$ with $i \in F$. 
\end{itemize} 
\end{lemma}   
\begin{proof} (1) follows from Blackadar's result \cite[Theorem 3.1]{Blackadar:IUMJ78}. (2) follows from \cite[Theorem 3.1]{Blackadar:IUMJ78} and \cite[Proposition 11.4.1(ii)]{KadisonRingrose:Book2} for example. 
\end{proof} 

\begin{proposition}\label{PA3} Let $\mathcal{B}_i \subseteq \mathcal{A}_i$, $i \in I$, be unital $C^*$-subalgebras. Then the universal free product $C^*$-algebra $\bigstar_{i\in I}\mathcal{B}_i$ is naturally embedded into $\bigstar_{i\in I}\mathcal{A}_i$. Namely, $\bigstar_{i\in I}\mathcal{B}_i$ can be identified with the $C^*$-subalgebra generated by the $\lambda_i(\mathcal{B}_i)$ and the canonical unital $*$-homomorphisms from $\mathcal{B}_i$ into $\bigstar_{i\in I}\mathcal{B}_i$ is given by the restriction of $\lambda_i$ to $\mathcal{B}_i$. 
\end{proposition}
\begin{proof}
Write $\mathfrak{B}_F := \bigstar_{i \in F}\mathcal{B}_i$  for each finite subset $F \Subset I$ with $\mathfrak{B}_\emptyset := \mathbb{C}1$. By the iterative use of Pedersen's result \cite[Theorem 4.2]{Pedersen:JFA99} with the help of Lemma \ref{LA1} we can see that $\mathfrak{B}_F \hookrightarrow \mathfrak{A}_F$ naturally. Then, by e.g.\ \cite[Proposition 11.4.1(ii)]{KadisonRingrose:Book2} we have a natural unital injective $*$-homomorphism from $\varinjlim_F\mathfrak{B}_F$ into $\varinjlim_F\mathfrak{A}_F$ by means of inductive limits. Thus the desired assertion follows thanks to Lemma \ref{LA2}(2).   
\end{proof} 

\begin{proposition}\label{PA4} 
Let $\bigstar_{i\in I}^\mathrm{alg}\mathcal{A}_i$ be the free product of the $\lambda_i(\mathcal{A}_i)$, $i \in I$, in the category of unital $*$-algebras, in which we regard each $\mathcal{A}_i$ as a unital $*$-subalgebra. Let $\lambda : \bigstar_{i\in I}^\mathrm{alg}\mathcal{A}_i \to \bigstar_{i\in I}\mathcal{A}_i$ be the unique $*$-homomorphism sending $a \in \mathcal{A}_i \subset \bigstar_{i\in I}^\mathrm{alg}\mathcal{A}_i$ to $\lambda_i(a) \in \bigstar_{i\in I}\mathcal{A}_i$, whose existence is guaranteed by universality. Then $\lambda$ must be injective. Namely, the $*$-subalgebra algebraically generated by the $\lambda_i(\mathcal{A}_i)$ in $\bigstar_{i\in I}\mathcal{A}_i$ can be identified with $\bigstar_{i\in I}^\mathrm{alg}\mathcal{A}_i$. 
\end{proposition}
\begin{proof} 
We have to show that if $a \in \bigstar_{i\in I}^\mathrm{alg}\mathcal{A}_i$ satisfies $\lambda(a) = 0$, then $a = 0$. To this end we will use the reduced free product construction, see e.g.\ \cite{VDN}, following Avitzour's idea \cite[Proposition 2.3]{Avitzour:TAMS82}. 

\medskip
Let $a \in \bigstar_{i\in I}^\mathrm{alg}\mathcal{A}_i$ be given. Then $a$ is nothing but a linear combination of words whose letters from the $\mathcal{A}_i$. For each $i \in I$ we let $\mathcal{A}_{i0}$ be the unital $C^*$-subalgebra of $\mathcal{A}_i$ generated by the letters from $\mathcal{A}_i$ (with fixed $i$) appearing in the words in the linear combination description of $a$.   Since there are only finitely many letters for each $i \in I$, $\mathcal{A}_{i0}$ must be separable. By Proposition \ref{PA3} we may and do regard $\bigstar_{i\in I}\mathcal{A}_{i0}$ as a unital $C^*$-algebra of $\bigstar_{i\in I}\mathcal{A}_i$ naturally, and $\lambda(a)$ falls into $\bigstar_{i\in I}\mathcal{A}_{i0}$. Hence we may and do regard each $\mathcal{A}_i$ as a separable unital $C^*$-algebra.  

\medskip
We claim that for each $i \in I$ there exists a faithful state $\omega_i$ on $\mathcal{A}_i$. Since $\mathcal{A}_i$ is separable, it faithfully acts on a separable Hilbert space, say $\pi : \mathcal{A}_i \curvearrowright \mathcal{K}$. See \cite[Theorem I.9.12]{Davidson:Book}. Then we choose a dense sequence of non-zero vectors $\xi_n \in \mathcal{K}$ and set $\omega_i(a) := \sum_{n=1}^\infty \frac{1}{2^n\Vert\xi_n\Vert_\mathcal{K}}(\pi(a)\xi_n|\xi_n)_\mathcal{K}$ for $a \in \mathcal{A}_i$. This clearly defines a faithful state. 

Consider the reduced $C^*$-free product $(\mathfrak{A},\omega) = \bigstar_{i\in I} (\mathcal{A}_i,\omega_i)$ with canonical $*$-homomorphisms $\gamma_i : \mathcal{A}_i \to \mathfrak{A}$. See e.g.\ \cite{VDN}. By universality, we have a unique $*$-homomorphism $\gamma : \bigstar_{i \in I}\mathcal{A}_i \to \mathfrak{A}$ such that $\gamma\circ\lambda_i = \gamma_i$ for every $i\in I$. Write 
\[
\bigstar_{i\in I}^\mathrm{alg}\mathcal{A}_i 
= 
\mathbb{C}1+\sum_{m\geq1}\sum_{\substack{i_k \neq i_{k+1}\\ (1\leq k \leq m-1)}} \mathcal{A}_{i_1}^\circ\cdots\mathcal{A}_{i_m}^\circ
\]
with $\mathcal{A}_i^\circ := \mathrm{Ker}(\omega_i)$, where $\mathcal{A}_{i_1}^\circ\cdots\mathcal{A}_{i_m}^\circ$ denotes all the linear combinations of words $a_1^\circ\cdots a_m^\circ$ with $a_k^\circ \in \mathcal{A}_{i_k}^\circ$. According to this representation we write 
\[
a = \alpha1 + \sum_{m\geq1}\sum_{\substack{i_k \neq i_{k+1}\\ (1\leq k \leq m-1)}} a(i_1,\dots,i_m),  
\]
where $a^\circ(i_1,\dots,i_m)$ is an element in $\mathcal{A}_{i_1}^\circ\cdots\mathcal{A}_{i_m}^\circ$. Remark that $a(i_1,\dots,i_m) = 0$ for all but except finitely many $(i_1,\dots,i_m)$. We denote by $a^\circ(i_1,\dots,i_m)^\otimes$ in the spacial (or minimal) $C^*$-tensor product $\mathcal{A}_1\otimes\cdots\otimes\mathcal{A}_{i_m}$ the corresponding elements obtained by changing each word $a_1^\circ\cdots a_m^\circ$ appearing in $a^\circ(i_1,\dots,i_m)$ to a simple tensor $a_1^\circ\otimes\cdots\otimes a_m^\circ \in \mathcal{A}_1\otimes\cdots\otimes\mathcal{A}_{i_m}$. By universality of algebraic tensor products sitting inside $\mathcal{A}_1\otimes\cdots\otimes\mathcal{A}_{i_m}$ (which is simply confirmed by the iterative use of a well-known fact, see e.g.\ \cite[Proposition 11.18]{KadisonRingrose:Book2} or a more direct statement \cite[Corollary 3.1]{BrownOzawa:Book}), we observe that  $a^\circ(i_1,\dots,i_m)^\otimes = 0$ implies $a^\circ(i_1,\dots,i_m) = 0$. 

Assume that $\lambda(a) = 0$. Since 
\[
\pi_\omega(\gamma(\lambda(a)))\xi_\omega 
= 
\alpha\xi_\omega + \sum_{m\geq1} \sum_{\substack{i_k \neq i_{k+1}\\ (1\leq k \leq m-1)}} \pi_\omega(\gamma(\lambda(a^\circ(i_1,\dots,i_m))))\xi_\omega,
\] 
where $(\mathcal{H}_\omega,\pi_\omega,\xi_\omega)$ is the GNS triple of $(\mathfrak{A},\omega)$. By the free independence among the $\lambda_i(\mathcal{A}_i)$, we can easily see that $\alpha\xi_\omega$ and the $\pi_\omega(\gamma(\lambda(a^\circ(i_1,\dots,i_m))))\xi_\omega$ are mutually orthogonal in $\mathcal{H}_\omega$. In particular, $\alpha$ as well as all the $\pi_\omega(\gamma(\lambda(a^\circ(i_1,\dots,i_m))))\xi_\omega$ must be $0$. Let $(\mathcal{H}_{\omega_i},\pi_{\omega_i},\xi_{\omega_i})$ be the GNS triple of $(\mathcal{A}_i,\omega_i)$. Then, it is easy to see that the norm of each $\pi_\omega(\gamma(\lambda(a^\circ(i_1,\dots,i_m))))\xi_\omega$ is the same as that of 
\[
(\pi_{\omega_{i_1}}\otimes\cdots\otimes\pi_{\omega_{i_m}})(a^\circ(i_1,\dots,i_m)^\otimes)(\xi_{\omega_{i_1}}\otimes\cdots\otimes\xi_{\omega_{i_m}}), 
\]
which must be $0$ too. Since $\omega_i$ is faithful, so is $\pi_{\omega_i}$ and hence the tensor product representation $\pi_{\omega_{i_1}}\otimes\cdots\otimes\pi_{\omega_{i_m}} : \mathcal{A}_{i_1}\otimes\cdots\otimes\mathcal{A}_{i_m} \curvearrowright \mathcal{H}_{\omega_{i_1}}\otimes\cdots\otimes\mathcal{H}_{i_m}$ too (see e.g.\ \cite[Theorem 11.1.3]{KadisonRingrose:Book2}). We conclude that $a^\circ(i_1,\dots,i_m)^\otimes = 0$ so that $a^\circ(i_1,\dots,i_m) = 0$. Consequently, $a$ must be $0$. 
\end{proof}

}

\end{document}